\documentclass[11pt, reqno]{amsart}
\usepackage{amssymb}
\usepackage{graphicx}
\usepackage{amssymb}
 \usepackage{amsthm}
\usepackage{epsf}
\usepackage{amsbsy,amsmath}
\usepackage{mathtools}
\usepackage{mathrsfs}
\usepackage{amsfonts}
\usepackage{amssymb}
\usepackage{enumerate}
\usepackage{eucal}
\usepackage{graphics,mathrsfs}
\usepackage{amsthm}
\usepackage{secdot}
\usepackage{esint}
\usepackage{varwidth}
\usepackage{tasks}
\usepackage{cite}
\usepackage{mathtools,bm}
\usepackage{cmap}
\usepackage{bigints}
\usepackage{esint}
\theoremstyle{plain}
\newtheorem{theorem}{Theorem}

\newtheorem{lemma}{Lemma}

\theoremstyle{definition}
\newtheorem{definition}{Definition}%\ignorespaces}

\newtheorem{remark}{Remark}

\allowdisplaybreaks

\usepackage{xcolor}

\long\def\symbolfootnote[#1]#2{\begingroup
\def\thefootnote{\fnsymbol{footnote}}\footnote[#1]{#2}\endgroup}
\usepackage{xcolor}
\usepackage[colorlinks,citecolor=blue,pagebackref=false,hypertexnames=false]{hyperref}

\numberwithin{equation}{section}

\numberwithin{equation}{section}

\begin{document}
\title[Equivalence of solutions to non-homogeneous nonlocal double phase problems]{On equivalence of weak and viscosity solutions to nonlocal double phase problems with nonhomogeneous data}
\author[S. Ghosh, R. Lakshmi and C. Zhang]{Sekhar Ghosh, R. Lakshmi and Chao Zhang}

\address[Sekhar Ghosh]{Department of Mathematics, National Institute of Technology Calicut, Kozhikode, Kerala, India - 673601}
\email{sekharghosh1234@gmail.com / sekharghosh@nitc.ac.in}
\address[R. Lakshmi]{Department of Mathematics, National Institute of Technology Calicut, Kozhikode, Kerala, India - 673601}
\email{lakshmir1248@gmail.com / lakshmi\_p220223ma@nitc.ac.in}
\address[Chao Zhang]{School of Mathematics and Institute for Advanced Study in Mathematics, Harbin Institute of Technology, Harbin 150001, China}
\email{czhangmath@hit.edu.cn}

\begin{abstract} 
This work focuses on the nonhomogeneous nonlocal double phase problem 
\begin{align*}
    L_au(x)=f(x,u,D_s^p u, D_{a,t}^q u) \text{ in } \Omega,
\end{align*}
where $\Omega\subset\mathbb{R}^N$ is a bounded domain with Lipschitz boundary, $0<s,t<1<p\leq q<\infty$ with $tq\leq sp$ and the operator $L_a$ is defined as 
\begin{align*}
    L_a u(x)&=2\operatorname{P.V.}\int_{\mathbb{R}^N}|u(x)-u(y)|^{p-2}(u(x)-u(y))K_{s,p}(x,y)\\
    &\ \ \ +2\operatorname{P.V.}\int_{\mathbb{R}^N}a(x,y)|u(x)-u(y)|^{q-2}(u(x)-u(y))K_{t,q}(x,y)dy.
\end{align*}
We establish the equivalence between weak and viscosity solutions under boundedness and continuity assumptions. In addition, the local boundedness of weak solutions in some special cases on $f$ is also obtained using the notion of De Giorgi classes.
\end{abstract}

\maketitle

\tableofcontents

\section{Introduction}

The relation between various types of solutions to partial differential equations has been a prominent area of research over the past years. The study of equivalence of solutions for the linear equations started with the celebrated works due to Lions \cite{Lions1983} and Ishii \cite{I1995}. For the quasilinear case, the equivalence of solutions for the homogeneous $p$-Laplace equation,
\begin{equation}\label{p-homo}
    -\Delta_p u =0,
\end{equation}
where $\Delta_p$ is the $p$-Laplacian defined by $\Delta_p u:=\operatorname{div}(|\nabla u|^{p-2}\nabla u)$ and the parabolic counterpart of \eqref{p-homo} has been explored by Juutinen \textit{et al.} \cite{JLM2001}. Later, a new proof for the equivalence of weak and viscosity solutions to \eqref{p-homo} was given by Julin and Juutinen \cite{JJ2012}.  Medina and Ochoa \cite{MO2019} investigated the relation between weak and viscosity solutions to the nonhomogeneous problem
\begin{equation}\label{p-nonhomo}
     -\Delta_p u =f(x,u,\nabla u).
\end{equation}
The equivalence of various notions of solutions to the normalized $p(x)$-Laplace equation has been explored by Siltakoski \cite{S2018}. For the studies on the regularity of solutions to equations involving $p$-Laplacian, we refer to \cite{E1982, Lewis1983, D1983}.  
\par The equivalence of weak and viscosity solutions to the homogeneous equation for the fractional $p$-Laplacian $(-\Delta_p)^s$ defined by
$$(-\Delta_p)^s u(x)=C_{N,p,s}\operatorname{P.V. }\int_{\mathbb{R}^N}\frac{| u(x)-u(y)|^{p-2}(u(x)-u(y))}{| x-y |^{N+ps}} dy,$$
where $C_{N,p,s}$ is the normalizing constant has been discussed by Korvenp{\"a}{\"a} \textit{et al.} \cite{KKL2019}.  Barrios and Medina \cite{BM2021} proved that the weak and viscosity solutions agree for the nonhomogeneous fractional $p$-Laplace equation
\begin{equation}\label{frac p-nonhomo}
    (-\Delta_p)^s u =f(x,u, D_s^p u)
\end{equation}
on a bounded domain $\Omega$ under certain conditions on $f$. Here, $D_s^p u$ is the fractional $p$-gradient defined by \eqref{p-frac grad}. Caffarelli and Silvestre \cite{CS2007} proved a nonlocal Harnack inequality for the homogeneous fractional Laplacian via an extension method. Di Castro \textit{et al.} \cite{DKP2014} obtained a Harnack-type inequality for homogeneous nonlocal equations of the fractional $p$-Laplacian type. The authors in \cite{DKP2016} also established the local boundedness and local H\"older continuity of solutions to homogeneous fractional $p$-Laplace type equations.
The local boundedness, local H\"older continuity and a Harnack type inequality for solutions to nonhomogeneous fractional $p$-Laplace type equations under certain conditions were obtained by Cozzi \cite{C2017} using fractional De Giorgi classes, extending the work of  Caffarelli \textit{et al.} \cite{CCV2011} for fractional type parabolic equations with $p=2$. For the fractional $p$-Laplace type equations, H\"older continuity of viscosity solutions has been explored by Lindgren \cite{L2016}. For more results on the regularity of solutions to similar nonlocal problems, we refer to \cite{IMS2016, BKO2023, CL2018, BL2017} and the references therein.
\par A relation between weak and viscosity solutions to the mixed problem
\begin{equation}\label{mix}
    -\Delta_p u+ (-\Delta_p)^s u=c(x)|u|^{p-2}u
\end{equation}
for a continuous function $c$, has been studied by Shang and Zhang \cite{SZ2023}. Lakshmi and Ghosh \cite{LG2025}, established the equivalence of weak and viscosity solutions for the homogeneous mixed local and nonlocal type equation on a bounded domain in $\mathbb{R}^N$. There is a vast literature on the regularity of solutions equations of mixed local and nonlocal type (see \cite{BDVV2022, BV2023, GK2022, GL2023, SVWZ2022, BMS2023, DM2024} and the references therein). 

\par The equivalence of solutions to the local double phase equation,
\begin{equation}\label{DP-nonhomo}
    -\operatorname{div}(|\nabla u|^{p-2}\nabla u+a(x)|\nabla u|^{q-2}\nabla u)=f(x, u , \nabla u), \ 1<p\leq q<\infty, \ a(x)\geq 0
\end{equation}
has been studied for the case $f\equiv 0$ by Fang and Zhang \cite{FZ2022} and the nonhomogeneous case by Fang \textit{et al.} \cite{FRZ2024}. The authors in \cite{FRZ2024} also proved the Lipschitz continuity of bounded viscosity solutions to \eqref{DP-nonhomo}. A Harnack inequality for double phase functionals was studied by Baroni \textit{et al.}\cite{BCM2015}. Later, Byun and Oh \cite{BO2020} utilized the De Giorgi approach to study the optimum conditions on $a(\cdot)$ for the regularity of generalized double phase functionals. The regularity of solutions to double phase equations were further explored in \cite{BCM2018, CM2015, M1991}.
Recently, equations involving the nonlocal double phase equation given by
\begin{align}\label{D-nonhomo}
   \operatorname{P.V.}\int_{\mathbb{R}^N}&\frac{|u(x)-u(y)|^{p-2}(u(x)-u(y))}{|x-y|^{N+sp}}dy \nonumber\\
 &+\operatorname{P.V.}\int_{\mathbb{R}^N}\frac{a(x,y)|u(x)-u(y)|^{q-2}(u(x)-u(y))}{|x-y|^{N+tq}}dy,
\end{align}
where $p,q\in (1,\infty)$ and $s,t\in(0,1)$ has also gained considerable interest among researchers. De Filippis and Palatucci \cite{DP2019} first established the H\"older continuity of viscosity solutions to nonhomogeneous equations of operators with the model \eqref{D-nonhomo} under certain conditions. Scott and Mengesha \cite{SM2022} studied a self-improving property of weak solutions to nonhomogeneous equations of \eqref{D-nonhomo}. A relation between weak and viscosity solutions to the homogeneous nonlocal double phase equation and the the local H\"older continuity of weak solutions were established by Fang and Zhang in \cite{FZ2023}. The local boundedness and local H\"older continuity of weak solutions to the homogeneous nonlocal double phase equation has been proved by Byun \textit{et al.} \cite{BOS2022}.

\par Let $\Omega \subset \mathbb{R}^N$ be a bounded domain with Lipschitz boundary and $0<s,t<1<p\leq q<\infty$ with $tq \leq sp$. We study the equivalence of weak and viscosity solutions to the problem
\begin{align}\label{D}
   L_a u(x)=f(x,u,D_s^p u, D_{a,t}^q u) \text{ in } \Omega,
\end{align}
where $f :\Omega \times \mathbb{R} \times \mathbb{R} \times \mathbb{R} \rightarrow \mathbb{R}$, $a(x,y)\geq 0$ for all $x,y \in \mathbb{R}^N$ and the operator $L_a$ is defined by
\begin{align*}
    L_a u(x)&=2\operatorname{P.V.}\int_{\mathbb{R}^N}|u(x)-u(y)|^{p-2}(u(x)-u(y))K_{s,p}(x,y)dy\\
    &\ \ \ +2\operatorname{P.V.}\int_{\mathbb{R}^N}a(x,y)|u(x)-u(y)|^{q-2}(u(x)-u(y))K_{t,q}(x,y)dy.
\end{align*}
Moreover, the local boundedness of weak solutions to problem \eqref{D} for a particular scenario is also obtained. 
\par Before stating the main theorems, we initially give the comparison principle between the weak sub and supersolutions to problem \eqref{D} (see, Definition \ref{D-WS}).
\begin{definition}[Comparison Principle]\label{D-CP}
    Let $\mathcal{D}\subset \Omega$ be an open set and $(u,f)$ be a weak supersolution to problem \eqref{D} in $\mathcal{D}$. Then, $(u,f)$ satisfies the comparison principle if for every weak subsolution $(v,f)$ to problem \eqref{D} in $\mathcal{D}$ with $u\geq v$ in $\mathbb{R}^N \setminus \mathcal{D}$, we have $u\geq v$ in $\mathbb{R}^N$.
\end{definition}
The comparison principle has not been obtained yet even for problem \eqref{D} with $a\equiv 0$, i.e., the non-homogeneous fractional $p$-Laplace equation. In this case, several types of the functions $f$ for which the comparison principle holds has been discussed in \cite{BM2021}. We provide some examples and discuss the proof of the comparison principle for some special cases of problem \eqref{D}.
\par Next, we state our first main result.
\begin{theorem}\label{D-T1}
    Let $u$ be a continuous weak supersolution to problem \eqref{D} and let the comparison principle (Definition \ref{D-CP}) hold. Assume that the map $f=f(x,t,\zeta,\eta)$ is continuous in $x$ and $t$ and is Lipschitz continuous in $\zeta$ and $\eta$. Then, $u$ is a viscosity supersolution to problem \eqref{D}.
\end{theorem}
The reverse implication of Theorem \ref{D-T1} is discussed in the following theorem.

\begin{theorem}\label{D-T2}
    Let $u$ be a bounded viscosity supersolution to problem \eqref{D} and $f:\Omega \times \mathbb{R}^3 \rightarrow\mathbb{R}$ satisfy the following conditions:
    \begin{enumerate}[(a)]
        \item $f=f(x,t,\zeta,\eta)$ is uniformly continuous in $\Omega \times \mathbb{R}^3$ and Lipschitz continuous in $\zeta$ and $\eta$.
        \item $f=f(x,t,\zeta,\eta)$ is non-increasing in $t$.
        \item $f$ satisfies the equation
          \begin{equation}\label{D-T2-f}
              |f(x,t,\zeta,\eta)|\leq \gamma_1(|t|)|\zeta|^{\frac{p-1}{p}}+\gamma_2(|t|)|\eta|^{\frac{q-1}{q}}+h(x),
          \end{equation}
          where $\gamma_i:\mathbb{R}\rightarrow \mathbb{R}$ are continuous functions with $\gamma_i\geq 0$ for $i\in\{1,2\}$ and $h\in L^\infty_{\text{loc}}(\Omega)$.
    \end{enumerate}
    Then, $u$ is a weak supersolution to problem \eqref{D}.
\end{theorem}
{Motivated by the results in \cite{C2017}, we employ the idea of double phase De Giorgi classes to obtain the boundedness of certain weak solutions to problem} \eqref{D}, when the function $f=f(x,t,\zeta,\eta)$ is independent of $\zeta$ and $\eta$ and 
 \begin{equation}\label{D-f-H}
        |f(x,t)|\leq c_1+c_2|t|^{l-1}, \text{ for every }x\in\Omega,
\end{equation}
where $1<l<q_t^*=\frac{Nq}{N-tq}$. {Establishing the local H\"older continuity property of solutions to \eqref{D}, where $f$ satisfies \eqref{D-f-H} remains challenging due to the crucial role of both the fractional $p$-seminorm and the weighted fractional $q$-seminorm terms in the definition of double phase De Giorgi classes and proof of local H\"older continuity of functions in these classes. Our result on the boundedness of weak solutions to \eqref{D}} is as follows.
\begin{theorem}\label{D-T3}
    Let $tq\leq N$ and {$u\in W^{t,q}_{\text{loc}}(\mathbb{R}^N)\cap  W^{s,p}(\mathbb{R}^N)\cap  L_{s,p}^{p-1}(\mathbb{R}^N)\cap L_{a,t,q}^{q-1}(\mathbb{R}^N)$} be a weak solution to the problem 
    \begin{align}\label{D-1}
        L_a u(x)=f(x,u) \text{ in }\Omega,
    \end{align}
    where there exist constants $c_1,c_2\geq0$ such that the function $f:\mathbb{R}^N \times \mathbb{R} \rightarrow \mathbb{R}$ satisfies \eqref{D-f-H} for some $l\in (1,q_t^*)$ with $q_t^*=\frac{Nq}{N-tq}$. Then, $u$ is locally bounded in $\Omega$, i.e., there exist constants $\tilde{R}\geq 0, \ C=C(N,s,p,t,q,c_2)\geq 1$ such that for any $0<2R<\max\left\{\operatorname{dist}(x_0,\partial\Omega),\tilde{R}\right\}$,
     \begin{align*}
        \|u\|_{L^\infty(B_R(x_0))}&\leq \operatorname{Tail}_{s,p}((u-k')_+,x_0,R)+ \operatorname{Tail}_{a,t,q}((u-k')_+,x_0,R)+R^{\frac{\rho+N\tilde{\epsilon}}{q}}\theta+2k'\\
        &\ \ \ +C\left[\frac{1}{R^N}\left(\|(u-k')_+\|_{L^p(B_{2R}(x_0))}^p+\left(\frac{R}{2}\right)^{sp-tq}\|(u-k')_+\|_{L^q(B_{2R}(x_0))}^q\right)\right]^{\frac{1}{p}}.
    \end{align*}
    {Here} the choice of the constants $\theta, k',\tilde{\epsilon},\rho, \tilde{R}$ are as follows.
    \begin{itemize}
        \item If $c_2=0$, then $\theta=c_1^\frac{1}{q-1},\ k'=0,\ \tilde{\epsilon}=\frac{tq}{N}, \ \rho=\frac{tq}{q-1}$ and $\tilde{R}=\infty$.
        \item If $c_2>0$ and $1<l\leq q$, then $\theta=c_1^\frac{1}{q-1},\ k'=1,\ \tilde{\epsilon}=\frac{tq}{N}, \ \rho=\frac{tq}{q-1}$ and $\tilde{R}=1$.
        \item If $c_2>0$ and $q<l\leq q_t^*$, then $\theta=c_1^\frac{1}{q-1},\ k'=0, \ \tilde{\epsilon}=1-\frac{N-\alpha q}{Nq}, \ \rho=\frac{tq}{q-1}$ and $\tilde{R}=\tilde{R}(N,q,t,\Lambda_2,c_2,\|u\|_{L^{q_\alpha^*}(\Omega)})$, where $0<\alpha \leq t$ is fixed.
    \end{itemize}
\end{theorem}
An overview of the paper is as follows. The preliminary definitions, notations and results required for our work will be included in Section \ref{D-S2}. The proof of the comparison principle for particular instances of problem \eqref{D} will be given in Section \ref{D-S3}. Section \ref{D-S4} starts with a supporting lemma and establishes Theorem \ref{D-T1} using a contradiction argument. In Section \ref{D-S5}, we prove Theorem \ref{D-T2}, which uses the idea of infimal convolution of functions and follows a limiting argument. Finally, we use the idea of De Giorgi classes in Section \ref{D-S6} to obtain the boundedness result given in Theorem \ref{D-T3}.

\section{Preliminaries and space setup}\label{D-S2}
This section focuses on the basic definitions and notations required to study problem \eqref{D}. We first recall the definition of fractional Sobolev spaces. % Let $\Omega \subset \mathbb{R}^N$ be an open set. The Sobolev space is defined by
%\begin{equation*}
 %   W^{1,p}(\Omega)=\{u \in L^p(\Omega): \frac{\partial u}{\partial x_i} \in L^p(\Omega) \text{ for } 1\leq i \leq N \},
%\end{equation*}
%where $1<p<\infty$. It is endowed with the norm
%$$\|u\|_{W^{1,p}(\Omega)}=\left( \|u\|_{L^p(\Omega)}^p +\|\nabla u\|_{L^p(\Omega)}^p \right)^{\frac{1}{p}}.$$
For $0<s<1<p<\infty$, the fractional Sobolev space on an open set $\Omega\subset \mathbb{R}^N$ is defined by
$$ W^{s,p}(\Omega)=\left\{u \in L^p(\Omega): \int_\Omega \int_\Omega \frac{| u(x)-u(y)|^p}{| x-y |^{N+ps}}dxdy <\infty\right\}.$$
It is endowed with the norm $ \|u\|_{W^{s,p}(\Omega)}=\left( \|u\|_{L^p(\Omega)}^p +[u]_{W^{s,p}(\Omega)}^p \right)^{\frac{1}{p}}$,
where the Gagliardo semi-norm $[u]_{W^{s,p}(\Omega)}$ is defined by
$$[u]_{W^{s,p}(\Omega)}^p=\int_\Omega \int_\Omega \frac{| u(x)-u(y)|^p}{| x-y |^{N+ps}}dxdy.$$ 
We have the following fractional Sobolev inequality from Theorem 6.5 in \cite{NPV2012}.
\begin{theorem}
    Let $0<s<1\leq p<\infty$ and $sp\leq N$. Then, there exists a constant $S=S(N,s,p)>0$ such that
    $$S\|u\|_{L^{p_s^*}(\mathbb{R}^N)}^p \leq \int_{\mathbb{R}^N} \int_{\mathbb{R}^N} \frac{| u(x)-u(y)|^p}{| x-y |^{N+ps}}dxdy$$
    for every $u\in W^{s,p}(\mathbb{R}^N)$, where $p_s^*=\frac{Np}{N-sp}$.
\end{theorem}
We also have the next continuous embedding theorem from Theorem 4.47 and Corollary 4.34 in \cite{DD2012}.
\begin{theorem}\label{fSS-CE}
    Let $0<s<s'<1<p<\infty$. Then, the following embeddings are continuous.
    \begin{enumerate}[(1)]
            \item  If $sp<N$, $W^{s,p}(\mathbb{R}^N) \subset L^q(\mathbb{R}^N), \ q \in [1,p_s^*]$. 
            \item If $sp=N,$  $W^{s,p}(\mathbb{R}^N) \subset L^q(\mathbb{R}^N), \ q \in [1,\infty)$. 
            \item If $sp>N,$   $W^{s,p}(\mathbb{R}^N) \subset L^{\infty}(\mathbb{R}^N),\  W^{s,p}(\mathbb{R}^N) \subset C^{0,\alpha}(\mathbb{R}^N), \ \alpha=s-\frac{N}{p}$. 
    \end{enumerate}
    The embedding $W^{s,p}(\mathbb{R}^N) \subset W^{s',p}(\mathbb{R}^N)$ is also continuous.
\end{theorem}
Further properties of these spaces can be seen in \cite{DD2012, NPV2012, L2023}.
Throughout the rest of the paper, we assume that $\Omega \subset \mathbb{R}^N$ is a bounded open set, $0<t\leq s<1<p\leq q<\infty$ satisfies $tq\leq sp$. Now, we characterize the {Kernals} $K_{s,p},\ K_{t,q}:\mathbb{R}^N \times\mathbb{R}^N \rightarrow (0, \infty)$. They have the following properties.
\begin{enumerate}
    \item There exist $\Lambda_1,\Lambda_2>0$ such that for all $x,y\in\mathbb{R}^N$ with $x\neq y$,
    \begin{align*}
        \frac{1}{\Lambda_1|x-y|^{N+sp}}\leq K_{s,p}(x,y)\leq \frac{\Lambda_1}{|x-y|^{N+sp}},\\
        \frac{1}{\Lambda_2|x-y|^{N+tq}}\leq K_{t,q}(x,y)\leq \frac{\Lambda_2}{|x-y|^{N+tq}}.
    \end{align*}
    \item $K_{s,p}(x,y)=K_{s,p}(y,x)$ and $K_{t,q}(x,y)=K_{t,q}(y,x)$ for all $x,y\in\mathbb{R}^N$.
    \item The functions {$y \mapsto K_{s,p}(x,y)$ and $y \mapsto K_{t,q}(x,y)$} are continuous in $\mathbb{R}^N\setminus \{x\}$.
    \item $K_{s,p}(x+z,y+z)=K_{s,p}(x,y)$ and $K_{t,q}(x+z,y+z)=K_{t,q}(x,y)$ for all $x,y,z \in \mathbb{R}^N$.
\end{enumerate}
The weight function $a:\mathbb{R}^N \times\mathbb{R}^N \rightarrow (0, \infty)$ has the following properties.
\begin{enumerate}
    \item $a(.,.)$ is symmetric. i.e., $a(x,y)=a(y,x)$ for all $x,y\in\mathbb{R}^N$.
    \item $a(.,.)$ is continuous.
    \item $a(.,.)$ is translation invariant. i.e., $a(x+z,y+z)=a(x,y)$ for all $x,y,z \in \mathbb{R}^N$.
    \item There exists $M>0$ such that $0<a(x,y) \leq M$ for all $x,y\in\mathbb{R}^N$.
\end{enumerate}
The fractional $p$-gradient and the weighted fractional $q$-gradient are given by
\begin{align}
    D_s^p u (x)&=\int_{\mathbb{R}^N}\frac{|u(x)-u(y)|^p}{|x-y|^{N+sp}}dy,\label{p-frac grad}\\
    D_{a,t}^q u(x)&=\int_{\mathbb{R}^N}a
    (x,y)\frac{|u(x)-u(y)|^q}{|x-y|^{N+tq}}dy\label{q-frac grad}.
\end{align}
We need the concept of tail spaces and weighted tail spaces to define the weak solutions to problem \eqref{D}. The Tail space $L_{s,p}^{p-1}(\mathbb{R}^N)$ is given by
$$L_{s,p}^{p-1}(\mathbb{R}^N)=\bigg\{v\in L^{p-1}_{\text{loc}}(\mathbb{R}^N) : \int_{\mathbb{R}^N}\frac{|v(x)|^{p-1}}{(1+|x|)^{N+sp}}dx<\infty \bigg\}.$$
For a function $v\in L_{s,p}^{p-1}(\mathbb{R}^N),\ x\in \mathbb{R}^N$ and $r>0$, Tail of $v$ concerning $B_r(x)$ is
$$\operatorname{Tail}_{s,p}(v;x,r)=\left(r^{sp}\int_{\mathbb{R}^N \setminus B_r(x)}\frac{|v(x)|^{p-1}}{|x-y|^{N+sp}}dy\right)^{\frac{1}{p-1}}.$$
The weighted Tail space $L_{a,t,q}^{q-1}(\mathbb{R}^N)$ is defined by {
$$L_{a,t,q}^{q-1}(\mathbb{R}^N)=\bigg\{v\in L^{q-1}_{\text{loc}}(\mathbb{R}^N) : \sup\limits_{x\in\mathbb{R}^N}\int_{\mathbb{R}^N}a(x,y)\frac{|v(y)|^{q-1}}{(1+|y|)^{N+tq}}dy<\infty \bigg\}.$$}
The weighted Tail of a function $v\in L_{a,t,q}^{q-1}(\mathbb{R}^N)$ concerning $B_r(x)\subset \mathbb{R}^N$ is defined by
$$\operatorname{Tail}_{a,t,q}(v;x,r)=\left(r^{tq}\sup\limits_{z\in\mathbb{R}^N}\int_{\mathbb{R}^N \setminus B_r(x)}a(z,y)\frac{|v(y)|^{q-1}}{|x-y|^{N+tq}}dy\right)^{\frac{1}{q-1}}.$$
We add the following remark prior to the definition of weak and viscosity solutions to problem \eqref{D}.
\begin{remark}\label{D-R1}
The notations below shall be followed in the rest of this work.
\begin{enumerate}
    \item Given $l>1$, $h_l(t):=|t|^{l-2}t$ for $t\in\mathbb{R}$.
    \item For $\Omega'\subset\Omega$, 
    $$Q(\Omega')=(\mathbb{R}^N\times \mathbb{R}^N)\setminus \left((\mathbb{R}^N \setminus\Omega')\times (\mathbb{R}^N \setminus\Omega')\right)=(\Omega' \times \mathbb{R}^N) \cup ((\mathbb{R}^N\setminus \Omega')\times \Omega').$$
    \item For $u,v:\mathbb{R}^N \rightarrow\mathbb{R}$, 
     \begin{align*}
         H_a(u,v)&=\int_{\mathbb{R}^N}\int_{\mathbb{R}^N}h_p(u(x)-u(y))(v(x)-v(y))K_{s,p}(x,y)dydx\\
         & \ \ \ +\int_{\mathbb{R}^N}\int_{\mathbb{R}^N}a(x,y)h_q(u(x)-u(y))(v(x)-v(y))K_{t,q}(x,y)dydx.
     \end{align*}
     \item For an arbitrary $r>0$, $$\Omega_{r}=\{x\in\Omega:\operatorname{dist}(x,\partial\Omega)>r\}.$$
     \item Given $u\in W^{t,q}(\mathcal{D})$,
     $$[u]_{W_{t,q,a}(\mathcal{D})}^q=\int_{\mathcal{D}}  \int_{\mathcal{D}} a(x,y)\frac{| u(x)-u(y)|^q}{| x-y |^{N+tq}}dxdy.$$
\end{enumerate}
\end{remark}
Following \cite{KKL2019}, it can be seen that $L_a u$ may not be defined for all functions $u\in L_{s,p}^{p-1}(\mathbb{R}^N)\cap L_{a,t,q}^{q-1}(\mathbb{R}^N)$. Thus, in order to define the notion of viscosity solutions to problem \eqref{D}, we need the class of functions
\begin{equation*}
    C_\beta^2(A)=\left\{ v \in C^2(A): \sup\limits_{x\in A} \left( \frac{\min\{d_v(x),1\}^{\beta-1}}{|\nabla v(x)} +\frac{|D^2v(x)|}{d_v(x)^{\beta-2}} \right)<\infty\right\},
\end{equation*}
where {$N_v=\{x \in \Omega: |\nabla v(x)|=0\}$ and $d_v(x):=\operatorname{dist}(x,N_v)$}. Now, we move to the definition of viscosity solutions motivated by Definition 2.2 in \cite{BM2021}.
\begin{definition}\label{D-VS}
    Let $u:\mathbb{R}^N\rightarrow\mathbb{R}$. Then, $u$ is said to be a viscosity supersolution (subsolution) to problem \eqref{D} if
    \begin{enumerate}[1.]
        \item $u$ is lower (upper) semi-continuous.
        \item $u_- \in L_{s,p}^{p-1}(\mathbb{R}^N)\cap L_{a,t,q}^{q-1}(\mathbb{R}^N) \ \left(u_+ \in L_{s,p}^{p-1}(\mathbb{R}^N)\cap L_{a,t,q}^{q-1}(\mathbb{R}^N)\right)$.
        \item $u>-\infty$ in $\Omega$ and $u<\infty$ a.e in $\mathbb{R}^N$.
        \item For every $B_r(x_0)\subset \Omega$ and  $\psi \in L_{s,p}^{p-1}(\mathbb{R}^N) \cap L_{a,t,q}^{q-1}(\mathbb{R}^N) \cap C^2(B_r(x_0))$ with $\psi \leq u \ (\psi \geq u)$ and $\psi(x_0)=u(x_0)$ satisfying one among the following two conditions
        \begin{enumerate}[(a)]
            \item $p>\frac{2}{2-s}$ or $\nabla \psi(x_0) \neq 0$,
            \item $p\leq\frac{2}{2-s}$, $x_0$ is an isolated critical point of $\psi$ and $\psi \in C_\beta^2(B_r(x_0))$ for a $\beta>\frac{sp}{p-1}$,
        \end{enumerate}
     we have $$L_a\psi(x_0)\geq (\leq ) f(x_0,u(x_0), D_s^p \psi(x_0), D_{a,t}^q \psi(x_0)).$$
    \end{enumerate}
    A function $u$ is said to be a viscosity solution to problem \eqref{D} if it is both a viscosity supersolution and a viscosity subsolution to problem \eqref{D}.
\end{definition}
We define weak solutions to problem \eqref{D} as follows.
\begin{definition}\label{D-WS}
    Let $u\in W^{s,p}(\mathbb{R}^N)\cap  L_{s,p}^{p-1}(\mathbb{R}^N)\cap L_{a,t,q}^{q-1}(\mathbb{R}^N)$. Then, $u$ is said to be a weak supersolution (subsolution) to problem \eqref{D} in $\Omega$ if for every $v\in C_c^\infty(\Omega)$ with $v\geq 0$, $u$ satisfies
    \begin{equation}\label{D-WSE}
        H_a(u,v)\geq (\leq )\int_{\Omega}f(x,u(x),D_s^p u(x), D_{a,t}^q u(x))v(x)dx.
    \end{equation}
    A function $u$ is said to be a weak solution to problem \eqref{D} if equality holds in \eqref{D-WSE} for all $v\in C_c^\infty(\Omega)$.
\end{definition}
Next, we introduce the infimal convolutions of a function.
\begin{definition}\label{D-IC}
    Let $u:\mathbb{R}^N \rightarrow\mathbb{R}$ and choose $l>0$ such that %$$l=2 \text{ if } p>\frac{2}{2-t} \text{ and } l>\frac{sp}{p-1}\geq 2 \text{ if } 1<q\leq \frac{2}{2-t}.$$
    $$l=\max\bigg\{2,\frac{sp}{p-1}\bigg\}.$$
    Given $\epsilon>0$, we define
    $$u_\epsilon(x_0)=\inf\limits_{x\in \mathbb{R}^N}\left(u(x)+\frac{|x_0-x|^l}{l\epsilon^{l-1}}\right).$$
\end{definition}
We have the following result by Lemma 3.1(i) in \cite{BM2021}. 
\begin{lemma}\label{D-u_ep}
    Let $u:\mathbb{R}^N\rightarrow \mathbb{R}$ be a bounded lower semicontinuous function. Then, there exists $r(\epsilon)>0$ satisfying $r(\epsilon)\rightarrow0$ as $\epsilon\rightarrow0$ and
    $$u_\epsilon(x_0)=\inf\limits_{x\in B_{r(\epsilon)}(x_0)}\left(u(x)+\frac{|x_0-x|^l}{l\epsilon^{l-1}}\right).$$
\end{lemma}
Finally, we discuss the concept of double phase De Giorgi classes.
\begin{definition}\label{D-DGC}
    Let $\mathcal{D}\subset\mathbb{R}^N$ be an open set and $v\in W^{t,q}(\mathcal{D})\cap W^{s,p}(\mathcal{D})\cap  L_{s,p}^{p-1}(\mathbb{R}^N)\cap L_{a,t,q}^{q-1}(\mathbb{R}^N)$, where $N\in\mathbb{N}$ and $1\leq p\leq q<\infty, \ 0<s,t<1$ with $tq\leq sp$. For $\theta\geq0, \ \tilde{C}\geq 1,\ \tilde{k}\in \mathbb{R}\cup\{-\infty\},\ 0<\epsilon<\frac{tq}{N},\ \rho\geq 0$ and $0<\tilde{R}\leq \infty$, the function $v$ belongs to the double phase De Giorgi class $\operatorname{DG}_a^+(\mathcal{D}, \theta, \tilde{C}, \tilde{k}, \tilde{\epsilon}, \rho, \tilde{R})$ if for every $x_0\in\Omega$ and $0<r<R\leq \max\{\operatorname{dist(x_0,\partial\mathcal{D})},\tilde{R}\},\ k\geq\tilde{k}$, we have
    \begin{align}\label{D-DG}
         [\psi_+&]_{W_{s,p}(B_{r}(x_0))}^p+[\psi_+]_{W_{t,q,a}(B_{r}(x_0))}^q\nonumber\\
         &+\int_{B_{r}(x_0)}\int_{B_{2\tilde{R}}(x_0)}\psi_+(x)\bigg(\frac{\psi_-(y)^{p-1}}{|x-y|^{N+sp}}+\frac{a(x,y)\psi_-(y)^{q-1}}{|x-y|^{N+tq}}\bigg)dydx \nonumber\\
        &\ \ \  \leq \tilde{C}\Bigg(\frac{R^{(1-s)p}}{(R-r)^p}\|\psi_+\|_{L^p(B_R(x_0))}^p+\frac{R^{(1-t)q}}{(R-r)^q}\|\psi_+\|_{L^q(B_R(x_0))}^q\nonumber\\
        &\ \ \quad +\bigg(\frac{R}{R-r}\bigg)^{N+sp}r^{-sp}\|\psi_+\|_{L^1(B_R(x_0))}\left(\operatorname{Tail}_{s,p}(\psi_+,x_0,r)\right)^{p-1}\nonumber\\
        &\ \ \quad +\bigg(\frac{R}{R-r}\bigg)^{N+tq}r^{-tq}\|\psi_+\|_{L^1(B_R(x_0))}\left(\operatorname{Tail}_{g,t,q}(\psi_+,x_0,r)\right)^{q-1}\nonumber\\
        &\ \ \quad +\left(R^\rho\theta^q+\frac{|k|^q}{R^{N\tilde{\epsilon}}}\right)|\operatorname{supp}\psi_+ \cap B_R(x_0)|^{1-\frac{tq}{N}+\tilde{\epsilon}}\Bigg),
    \end{align}
    where $\psi=v-k$. The function $v$ is in $\operatorname{DG}_a^-(\mathcal{D}, \theta, \tilde{C}, \tilde{k}, \tilde{\epsilon}, \rho, \tilde{R})$ if \eqref{D-DG} holds for $\psi=k-v$. The function $v$ is in the double phase {De Giorgi} class $\operatorname{DG}_a(\mathcal{D}, \theta, \tilde{C}, \tilde{k}, \tilde{\epsilon}, \rho, \tilde{R})$ if it is in both $\operatorname{DG}_a^+(\mathcal{D}, \theta, \tilde{C}, \tilde{k}, \tilde{\epsilon}, \rho, \tilde{R})$ and $\operatorname{DG}_a^-(\mathcal{D}, \theta, \tilde{C}, \tilde{k}, \tilde{\epsilon}, \rho, \tilde{R})$.
\end{definition}
We conclude this section with the following remark.
\begin{remark}\label{D-DGC-R}
    Since $u_-=(-u)_+$ for a function $u:\mathbb{R}^N \rightarrow\mathbb{R}$, if $u\in \operatorname{DG}_a^+(\mathcal{D}, \theta, \tilde{C}, \tilde{k}, \tilde{\epsilon}, \rho, \tilde{R})$, then we have $-u\in \operatorname{DG}_a^-(\mathcal{D}, \theta, \tilde{C}, \tilde{k}, \tilde{\epsilon}, \rho, \tilde{R})$ and vice-versa.
\end{remark}

\section{Comparison principle}\label{D-S3}
\noindent The comparison principle between weak subsolutions and supersolutions to problem \eqref{D} plays a crucial role in proving Theorem \ref{D-T1}. Unfortunately, literature on Definition \ref{D-CP} has not been obtained for all $f:\Omega \times \mathbb{R}^3\rightarrow \mathbb{R}$. Some of the examples where Definition \ref{D-CP} is satisfied are the following.
\begin{enumerate}
    \item the homogeneous non-local double phase equation, i.e., the case $f\equiv 0$ in problem \eqref{D}.
    \item problem \eqref{D} with $f=f(x,u):=g(x)|u|^{p-2}u$ with $g\leq 0$ and an additional condition $u\geq 0$. 
\end{enumerate}
In this section, we discuss the proof of the comparison principle for a sub-case of problem \eqref{D}.
\begin{theorem}\label{D-CP1}
    Let $\mathcal{D}\subset \mathbb{R}^N$ be an open set and $u,v$ be weak supersolution and subsolution respectively to problem \eqref{D} in $\mathcal{D}$, where $f:=f(x,t,\zeta,\eta)$ is independent of $\zeta$ and $\eta$ and $f$ is non-increasing in $t$. Assume also that $u\geq v$ a.e. in $\mathbb{R}^N \setminus \mathcal{D}$. Then, we have $u\geq v$ a.e. in $\mathbb{R}^N$.
\end{theorem}
\begin{proof}
    Consider the test function $\phi=(v-u)_+\geq 0$.  Denote
  \begin{align*}
        I_1&=\int_{\mathbb{R}^N}\int_{\mathbb{R}^N}\bigg(h_p(u(x)-u(y))-h_p(v(x)-v(y))\bigg)\\
        &\hspace{2cm}\times\bigg((v-u)_+(x)-(v-u)_+(y)\bigg)K_{s,p}(x,y)dxdy, \\
        I_2&=\int_{\mathbb{R}^N}\int
        _{\mathbb{R}^N}a(x,y)\bigg(h_q(u(x)-u(y))-h_q(v(x)-v(y))\bigg)\\
        &\hspace{2cm}\times\bigg((v-u)_+(x)-(v-u)_+(y)\bigg)K_{t,q}(x,y)dxdy.
    \end{align*}
    From the definition of weak sub and supersolutions to problem \eqref{D}, we deduce
    \begin{align}\label{D-CP1-1}
        I_1+I_2 \geq \int_{\mathcal{D}}\big(f(x,u,D_s^p u,D_{a.t}^q u)-f(x,v,D_s^p v,D_{a.t}^q v)\big)(v-u)_+(x) dx \geq 0.
    \end{align}
    The second inequality holds in \eqref{D-CP1-1} since $f(x,u,D_s^p u,D_{a.t}^q u)\geq f(x,v,D_s^p v,D_{a.t}^q v)$ in the set {$\operatorname{supp}(v-u)_+$}.
    From Lemma 2.4 in \cite{BM2021}, we have 
    \begin{equation}\label{D-CP1-2}
        |a|^{l-2}a-|b|^{l-2}b =(l-1)(a-b)\int_0^1 |\tau a+(1-\tau)b|^{p-2} d\tau, \ a,b \in \mathbb{R}, \ 1<l<\infty.
    \end{equation}
    Applying \eqref{D-CP1-2} for $a=u(x)-u(y)$, $b=v(x)-v(y)$ and $l=p$ in $I_1$, we obtain
    \begin{align}\label{D-CP1-3}
        I_1=(p-1)\int_{\mathbb{R}^N}\int_{\mathbb{R}^N}&\bigg((v-u)_+(x)-(v-u)_+(y)\bigg)(u(x)-u(y)-(v(x)-v(y)))\nonumber\\
        & \times \int_0^1 \bigg|u(x)-u(y)+\tau (v(x)-v(y))\bigg|^{p-2}d\tau K_{s,p}(x,y)dxdy.
    \end{align}
    It can be easily observed that
    \begin{align}\label{D-CP1-4}
        \bigg((v-u)_+(x)-(v-u)_+(y)\bigg)&\bigg(u(x)-u(y)-(v(x)-v(y))\bigg)\nonumber\\
        &\leq -\bigg((v-u)_+(x)-(v-u)_+(y)\bigg)^2.
    \end{align}
    Applying \eqref{D-CP1-4} in \eqref{D-CP1-3}, we get 
    \begin{align}\label{D-CP1-5}
        I_1 \leq -(p-1)\int_{\mathbb{R}^N}\int_{\mathbb{R}^N}&(v(x)-u(x))_+-(v(y)-u(y))_+\bigg)^2\nonumber\\
        & \times \int_0^1 \bigg|u(x)-u(y)+\tau (v(x)-v(y))\bigg|^{p-2}d\tau K_{s,p}(x,y)dxdy \leq0.
    \end{align}
    Similarly, we also obtain
    \begin{align}\label{D-CP1-6}
        I_2\leq  -(q-1)\int_{\mathbb{R}^N}\int_{\mathbb{R}^N}&a(x,y)\bigg((v(x)-u(x))_+-(v(y)-u(y))_+\bigg)^2\nonumber\\
        & \times \int_0^1 \bigg|u(x)-u(y)+\tau (v(x)-v(y))\bigg|^{q-2}d\tau K_{t,q}(x,y)dxdy \leq0.
    \end{align}
    Therefore, from \eqref{D-CP1-1}, \eqref{D-CP1-5} and \eqref{D-CP1-6}, we deduce
    \begin{align*}
        &\int_{\mathbb{R}^N}\int_{\mathbb{R}^N}\bigg((v(x)-u(x))_+-(v(y)-u(y))_+\bigg)^2\\
         &\hspace{1cm}\times \int_0^1 \bigg|u(x)-u(y)+\tau (v(x)-v(y))\bigg|^{p-2}d\tau K_{s,p}(x,y)dxdy \nonumber\\
         &+ \int_{\mathbb{R}^N}\int_{\mathbb{R}^N}a(x,y)\bigg((v(x)-u(x))_+-(v(y)-u(y))_+\bigg)^2 \\
         &\hspace{1cm}\times \int_0^1 \bigg|u(x)-u(y)+\tau (v(x)-v(y))\bigg|^{q-2}d\tau K_{t,q}(x,y)dxdy =0.
    \end{align*}
    Thus, we obtain that $(u(x)-v(x))_+=(u(y)-v(y))_+$ for a.e. $x,y \in \mathbb{R}^N$. Then, $(u-v)_+$ is a constant a.e. 
    in $\mathbb{R}^N$. Since we have $(u-v)_+=0$ a.e. in $\mathbb{R}^N\setminus \mathcal{D}$, we conclude that $(u-v)_+=0$ a.e. in $\mathbb{R}^N$. This completes the proof.
\end{proof}

\section{Weak solutions are viscosity solutions}\label{D-S4}
This section is devoted to prove Theorem \ref{D-T1}. We first establish the following lemma.
\begin{lemma}\label{D-L1}
Let $v\in L_{a, t,q}^{q-1}(\mathbb{R}^N)$ be Lipschitz continuous in $B_r(x_0)\subset \mathbb{R}^N$ and $0<\delta<r$. Then, for any $\epsilon>0$ and $\zeta \in C_c^2(B_r(x_0))$ with $0\leq \zeta\leq 1$, there {exists} $\alpha'>0$ such that whenever $0\leq \alpha<\alpha'$,
$$\sup\{|D_{a,t}^qv(x)-D_{a,t}^q(v+\alpha\zeta)(x)|:x\in B_\delta(x_0)\}<\epsilon.$$    
\end{lemma}
\begin{proof}
    Let $\epsilon>0$ and let $x\in B_\delta(x_0)$. For simplicity, we denote $$w_\alpha(x,y)=v(x)+\alpha\zeta(x)-v(y)-\alpha\zeta(y)~\text{for}~\alpha\geq0.$$
    Therefore, we have
    \begin{align}\label{D-L1-1}
        |D_{a,t}^qv(x)-D_{a,t}^q(v+\alpha\zeta)(x)|&=\bigg|\int_{B_\mu(x_0)}a(x,y)(|w_0(x,y)|^q-|w_\alpha(x,y)|^q)K_{t,q}(x,y)dy \nonumber\\
        &+\int_{\mathbb{R}^N \setminus B_\mu(x_0)}a(x,y)(|w_0(x,y)|^q-|w_\alpha(x,y)|^q)K_{t,q}(x,y)dy\bigg|\nonumber\\
        &\leq \int_{B_\mu(x_0)}a(x,y)\bigg||w_0(x,y)|^q-|w_\alpha(x,y)|^q\bigg|K_{t,q}(x,y)dy\nonumber\\
        & +\int_{\mathbb{R}^N \setminus B_\mu(x_0)}a(x,y)\bigg||w_0(x,y)|^q-|w_\alpha(x,y)|^q\bigg|K_{t,q}(x,y)dy.
    \end{align}
    Taking advantage of the Lipschitz continuity of $v,\zeta$, for all $y\in B_\mu(x_0)$ we attain 
    \begin{align}\label{D-L1-3}
        \bigg||w_0(x,y)|^q-|w_\alpha(x,y)|^q\bigg|&\leq C(q)\left(|v(x)-v(y)|^q+\alpha^q |\zeta(x)-\zeta(y)|^q\right)\nonumber\\
        &\leq C(q, \|\nabla v\|_{L^{\infty}(B_{r+\delta}(x_0))},\|\nabla \zeta\|_{L^{\infty}(B_{r+\delta}(x_0))})(1+\alpha^q)|x-y|^q.
    \end{align}
    Now, from Lemma 3.4 in \cite{KKL2019}, we have
    \begin{equation}\label{D-L1-2}
        \bigg||t_1|^q-|t_2|^q\bigg|\leq C(q)|t_1-t_2|(|t_2|+|t_1-t_2|)^{q-1} ,\ t_1,t_2 \in \mathbb{R}.
    \end{equation}
    Note that 
    $$|w_0(x,y)-w_\alpha(x,y)|=\alpha|\zeta(x)-\zeta(y)|\leq 2\alpha,$$
    Consequently, employing \eqref{D-L1-3} in the first integral in \eqref{D-L1-1} and applying \eqref{D-L1-2} for $t_1=w_\alpha(x,y)$, $t_2=w_0(x,y)$ in the second integral in the RHS of \eqref{D-L1-1}, we obtain
    \begin{align*}
        |D_{a,t}^qv(x)-D_{a,t}^q(v+\alpha\zeta)(x)|&\leq C\int_{B_\mu(x_0)}a(x,y)(1+\alpha^q)|x-y|^q K_{t,q}(x,y)dy\nonumber\\
        &\ \ +C(q)\alpha \int_{\mathbb{R}^N \setminus B_\mu(x_0)}a(x,y)(2\alpha+|v(x)|+|v(y)|)^{q-1}K_{t,q}(x,y)dy\nonumber\\
        &\leq CM{\Lambda_2}\mu^{q(1-t)}(1+\alpha^q)+CM{\Lambda_2}\alpha^{q}\mu^{-tq}\nonumber\\
        &\ \ +CM{\Lambda_2}\alpha\|v\|_{L^{\infty}(B_r(x_0))}\mu^{-tq}+C{\Lambda_2}\alpha\mu^{-tq}\left(\operatorname{Tail}_{a,t,q}(v;x_0,\delta)\right)^{q-1}.
    \end{align*}
    Thus, first choosing $\mu$  and then choosing an $\alpha'$ sufficiently small and taking supremum over all $x\in B_\delta(x_0)$, we deduce that for $0\leq \alpha<\alpha'$,
    \begin{equation*}
        \sup\{|D_{a,t}^qv(x)-D_{a,t}^q(v+\alpha\zeta)(x)|:x\in B_\delta(x_0)\}<\epsilon.
    \end{equation*}
    Hence, the proof is complete.
\end{proof}
We now conclude this section by proving Theorem \ref{D-T1}. The proof requires $f$ to satisfy certain continuity properties and the comparison principle.  
\medskip

\noindent {\it{\bf{Proof of Theorem} \ref{D-T1}.}}
    Assume that $u$ is not a viscosity supersolution. Then, there exist $B_r(x_0)\subset \Omega$, $\psi \in  L_{s,p}^{p-1}(\mathbb{R}^N) \cap L_{a,t,q}^{q-1}(\mathbb{R}^N) \cap C^2(B_r(x_0))$ with $\psi(x_0)=u(x_0)$, $\psi \leq u$ in $\mathbb{R}^N$ and satisfying the condition $4(a)$ or $4(b)$ in Definition \ref{D-VS} such that the following inequality
    \begin{equation}\label{D-T1-1}
        L_a\psi(x_0) < f(x_0, \psi(x_0), D_s^p \psi(x_0), D_{a,t}^q \psi(x_0))=f(x_0, u(x_0), D_s^p \psi(x_0),D_{a,t}^q \psi(x_0))
    \end{equation}
    holds. Applying Lemma 4.5 in \cite{FZ2023}, we deduce that $L_a\psi$ is continuous in $B_r(x_0)$. Also, from the hypotheses we have the map $x \mapsto f(x, u(x), D_s^p \psi(x), D_{a,t}^q \psi(x))$ is continuous in $B_r(x_0)$. Since $\overline{B_{r'}(x_0)}$ is compact for all $0<r'<r$, from \eqref{D-T1-1}, there exist $0<\mu<r$ and $\epsilon>0$ such that 
    \begin{equation}\label{D-T1-2}
        L_a\psi(x) \leq f(x, u(x), D_s^p \psi(x), D_{a,t}^q \psi(x))-\epsilon, \text{ for } x \in B_\mu(x_0).
    \end{equation}
    By Lemma 4.6 in \cite{FZ2023}, there exist $\beta'>0$, $0<\mu'<\mu$ and $\zeta \in C_c^2\Big(B_\frac{\mu'}{2}(x_0)\Big)$ with $0 \leq \zeta \leq 1$ and $\zeta(x_0)=1$ such that on defining $\psi_\beta=\psi+\beta \zeta$, we have
    \begin{equation}\label{D-T1-3}
        \sup\limits_{B_{\mu'(x_0)}} |L_a\psi(x)-L_a\psi_\beta(x)|<\frac{\epsilon}{4}, \text{ for } 0\leq \beta <\beta'.
    \end{equation}
    By the Lipschitz continuity of $\zeta \mapsto f(x,t,\zeta,\eta)$, choose $\delta>0$ satisfying
    \begin{equation}\label{D-T1-4}
        |f(x,u(x),\zeta,\eta)-f(x,u(x),\zeta',\eta)|<\frac{\epsilon}{8} \text{ whenever } |\zeta-\zeta'|<\delta.
    \end{equation}
    From Lemma 2.6 in \cite{BM2021}, we get a $0<\beta''<\beta'$ such that
   \begin{equation}\label{D-T1-5}
        \sup\limits_{B_{\mu'(x_0)}}|D_s^p \psi-D_s^p \psi_\beta|<\delta, \text{ for } 0\leq \beta <\beta''.
    \end{equation}
    Similarly, by the Lipschitz continuity of $\eta \mapsto f(x,t,\zeta,\eta)$, there {exists} $\delta'>0$ such that
    \begin{equation}\label{D-T1-4'}
        |f(x,u(x),\zeta,\eta)-f(x,u(x),\zeta,\eta')|<\frac{\epsilon}{8} \text{ whenever } |\eta-\eta'|<\delta'.
    \end{equation}
    By Lemma \ref{D-L1}, we get a $0<\beta'''<\beta''$  such that
    \begin{equation}\label{D-T1-5'}
        \sup\limits_{B_{\mu'(x_0)}}|D_{a,t}^q \psi-D_{a,t}^q \psi_\beta|<\delta', \text{ for } 0\leq \beta <\beta'''.
    \end{equation}
    Now, choose $0<\beta <\beta'''$. Combining \eqref{D-T1-2}--\eqref{D-T1-5'}, we obtain
    \begin{align}\label{D-T1-6}
       L_a\psi_\beta(x) &\leq L_a\psi(x)+\frac{\epsilon}{4} \nonumber\\
       &\leq f(x, u(x), D_s^p \psi(x), D_{a,t}^q \psi(x))-\frac{3\epsilon}{4}\nonumber\\
       & \leq f(x, u(x), D_s^p \psi_\beta(x),D_{a,t}^q \psi(x))-\frac{5\epsilon}{8}\nonumber\\
       & \leq f(x, u(x), D_s^p \psi_\beta(x),D_{a,t}^q \psi_\beta(x))-\frac{\epsilon}{2}\nonumber\\
       &<f(x, u(x), D_s^p \psi_\beta(x),D_{a,t}^q \psi_\beta(x)),
    \end{align}
    for all $x \in B_{\mu'}(x_0)$. Let $\eta\in C_c^\infty(B_{\mu'}(x_0))$ with $\eta\geq 0$. Since $\psi_\beta\in C^2(B_{\mu'}(x_0))$, multiplying \eqref{D-T1-6} by $\eta$ and integrating, we get
    \begin{equation}\label{D-T1-6-1}
        \int_{B_{\mu'}(x_0)}L_a\psi_\beta(x)\eta(x)dx=\int_{B_{\mu'}(x_0)}f(x, u(x), D_s^p \psi_\beta(x),D_{a,t}^q \psi_\beta(x))\eta(x)dx.
    \end{equation}
    From the definition of $L_a\psi_\beta$ we have
    \begin{align}\label{D-T1-6-2}
        \int_{\mathbb{R}^N}L_a \psi_\beta(x)\eta(x)dx= 2&\int_{\mathbb{R}^N} \Bigg(\int_{\mathbb{R}^N}h_p(\psi_\beta(x)-\psi_\beta(y))K_{s,p}(x,y)\nonumber\\
        &+a(x,y)h_q(\psi_\beta(x)-\psi_\beta(y))K_{t,q}(x,y)dy\Bigg)\eta(x)dx.
    \end{align}
    Applying the change of variable in \eqref{D-T1-6-2}, we also deduce
    \begin{align}\label{D-T1-6-3}
        \int_{\mathbb{R}^N}L_a \psi_\beta(x)\eta(x)dx = -2&\int_{\mathbb{R}^N} \Bigg(\int_{\mathbb{R}^N}h_p(\psi_\beta(x)-\psi_\beta(y))K_{s,p}(x,y)\nonumber\\
        &+a(x,y)h_q(\psi_\beta(x)-\psi_\beta(y))K_{t,q}(x,y)dy\Bigg)\eta(y)dx.
    \end{align}
    Combining \eqref{D-T1-6-2} and \eqref{D-T1-6-3}, we obtain
    \begin{equation}\label{D-T1-6-4}
        \int_{\mathbb{R}^N}L_a \psi_\beta(x)\eta(x)dx=H_a(\psi_\beta,\eta).
    \end{equation}
    Hence, from \eqref{D-T1-6-1} and \eqref{D-T1-6-4}, we get 
    $$H_a(\psi_\beta,\eta)=f(x,u(x),D_s^p v(x), D_{a,t}^q v(x))\eta(x)dx.$$
    Thus, $\psi_\beta$ is a weak subsolution to the problem 
    \begin{equation}\label{D-ball}
        L_a v(x)=f(x,u(x),D_s^p v(x), D_{a,t}^q v(x))
    \end{equation}
    in $B_{\mu'}(x_0)$. Clearly, $u$ is a weak supersolution to problem \eqref{D-ball} in $B_{\mu'}(x_0)$. Since Definition \ref{D-CP} holds, we get that $\psi_\beta \leq u$ a.e. in $ B_{\mu'}(x_0)$. By the continuity of $\psi_\beta$ and $u$ in $ B_{\mu'}(x_0)$, we get that $\psi_\beta \leq u$ in $ B_{\mu'}(x_0)$. Then $\psi_\beta(x_0)=u(x_0)+\beta>u(x_0)$, which is a contradiction. Hence, $u$ is a weak supersolution to problem \eqref{D}. This concludes the proof.
\hfill\qedsymbol{}

\section{Viscosity solutions are weak solutions}\label{D-S5}
\noindent In this section, we prove that all the bounded viscosity solutions to problem \eqref{D} are weak solutions. For each $\epsilon>0$, we utilize the infimum convolution $u_\epsilon$ of the viscosity solution $u$ and a function $f_\epsilon$ defined using $f$ and use a limiting argument to prove Theorem \ref{D-T2}. We begin with the following Lemma.  

\begin{lemma}\label{D-L2}
Let {$u\in L^\infty_{\text{loc}}(\mathbb{R}^N)$} be a weak solution to problem \eqref{D}, where $f$ satisfies the conditions (a) and (c) in Theorem \ref{D-T2}, and let $u_\epsilon$ be given by Definition \ref{D-IC} and $f_\epsilon$ is defined by
\begin{equation}\label{D-f_ep}
    f_\epsilon(x,t,\zeta,\eta):=\inf\limits_{y\in B_{r(\epsilon)(x)}}f(y,t,\zeta,\eta).
\end{equation}
Consider $v\in C_c^\infty(\Omega)$ with $K=\operatorname{supp}v$ and $v\geq0$. Assume that 
\begin{align}\label{D-L2-1'}
    \lim\limits_{\epsilon\rightarrow 0}\Bigg(\int_K\int_{\mathbb{R}^N}&\left|u_\epsilon(x)-u_\epsilon(y)-(u(x)-u(y))\right|^p K_{s,p}(x,y)dydx\nonumber\\
    &+ \int_K\int_{\mathbb{R}^N}a(x,y)\left|u_\epsilon(x)-u_\epsilon(y)-(u(x)-u(y))\right|^q K_{t,q}(x,y)dydx \Bigg)=0.
\end{align}
Then, the following {one} holds:
\begin{equation}
    \lim\limits_{\epsilon
    \rightarrow0}\int_{\Omega}f_\epsilon(x,u_\epsilon, D_s^p u_\epsilon, D_{a,t}^q u_\epsilon)v dx= \int_{\Omega}(f(x,u, D_s^p u, D_{a,t}^q u)v dx.
\end{equation}
\end{lemma}
\begin{proof}
    Let $\delta>0$. For simplicity, we define for every $x,y \in \mathbb{R}^N$ and $\epsilon>0$,
    $$U(x,y):=u(x)-u(y), \ \ U_\epsilon(x,y):=u_\epsilon(x)-u_\epsilon(y)$$
    We have
    \begin{align}\label{D-L2-1}
        \bigg|\int_{\Omega}\bigg[(f(x,u, D_s^p u,& D_{a,t}^q u)-f_\epsilon(x,u_\epsilon, D_s^p u_\epsilon, D_{a,t}^q u_\epsilon)\bigg]v dx  \bigg| \nonumber\\
        &\leq \int_{K} \left|f(x,u, D_s^p u, D_{a,t}^q u)-f(x,u_\epsilon, D_s^p u, D_{a,t}^q u) \right|v dx \nonumber\\
        &\ \ \ + \int_{K} \left|f(x,u_\epsilon, D_s^p u, D_{a,t}^q u)-f(x,u_\epsilon, D_s^p u_\epsilon, D_{a,t}^q u) \right|v dx \nonumber\\
        &\ \ \ + \int_{K} \left|f(x,u_\epsilon, D_s^p u_\epsilon, D_{a,t}^q u)-f(x,u_\epsilon, D_s^p u_\epsilon, D_{a,t}^q u_\epsilon) \right|v dx \nonumber\\
        &\ \ \ + \int_{K} \left|f(x,u_\epsilon, D_s^p u_\epsilon, D_{a,t}^q u_\epsilon)-f_\epsilon(x,u_\epsilon, D_s^p u_\epsilon, D_{a,t}^q u_\epsilon) \right|v dx.
    \end{align}
    Now, since $f=f(x,t,\zeta,\eta)$ is Lipschitz continuous in $\eta$, there exist a constant $C>0$ such that for every $\epsilon>0$,
    \begin{align}\label{D-L2-2}
        \int_{K} \bigg|&f(x,u_\epsilon, D_s^p u_\epsilon, D_{a,t}^q u)-f(x,u_\epsilon, D_s^p u_\epsilon, D_{a,t}^q u_\epsilon) \bigg|v dx \nonumber\\
        &\leq C  \int_{K} |D_{a,t}^q u-D_{a,t}^q u_\epsilon|v dx\nonumber\\
        &=C(f,\Lambda_2)\int_K \left|\int_{\mathbb{R}^N} a(x,y)(|u(x)-u(y)|^q-|u_\epsilon(x)-u_\epsilon(y)|^q)K_{t,q}(x,y)dy \right|v(x)dx.
    \end{align} 
    Using Lemma 3.4 in \cite{KKL2019}, we have
    \begin{equation}\label{D-L2-3}
       \bigg||a|^q -|b|^q\bigg|\leq C(q)(|b|+|a-b|)^{q-1}|a-b|, a,b \in \mathbb{R}.
    \end{equation}
    Applying \eqref{D-L2-3} for $a=U_\epsilon(x,y)$, $b=U(x,y)$,  and using H\"older's inequality, we deduce
    \begin{align}\label{D-L2-4}
        \bigg|\int_{\mathbb{R}^N}& a(x,y)(|u(x)-u(y)|^q-|u_\epsilon(x)-u_\epsilon(y)|^q)K_{t,q}(x,y)dy \bigg|\nonumber\\
        &\leq C\int_{\mathbb{R}^N}a(x,y)\left(|U(x,y)|+|U_\epsilon(x,y)-U(x,y)|\right)^{q-1}|U_\epsilon(x,y)-U(x,y)|K_{t,q}(x,y)dy\nonumber\\
        &\leq C\bigg(\int_{\mathbb{R}^N} a(x,y)\big(|U(x,y)|+|U_\epsilon(x,y)-U(x,y)|\big)^{q}K_{t,q}(x,y)dy\bigg)^{\frac{q-1}{q}}\nonumber\\
        & \ \ \ \times \bigg(\int_{\mathbb{R}^N} a(x,y)|U_\epsilon(x,y)-U(x,y)|^{q}K_{t,q}(x,y)dy\bigg)^{\frac{1}{q}}\nonumber\\
        &\leq C\bigg(\int_{\mathbb{R}^N} a(x,y)\big(|U(x,y)|^{q}+|U_\epsilon(x,y)-U(x,y)|^{q}\big)K_{t,q}(x,y)dy\bigg)^{\frac{q-1}{q}}\nonumber\\
        & \ \ \ \times \bigg(\int_{\mathbb{R}^N} a(x,y)|U_\epsilon(x,y)-U(x,y)|^{q}K_{t,q}(x,y)dy\bigg)^{\frac{1}{q}}.
    \end{align}
    Since $v\in C_c^\infty(\Omega)$, substituting \eqref{D-L2-4} in \eqref{D-L2-2}, applying {H\"older's} inequality and using \eqref{D-L2-1'}, we get an $\epsilon_0>0$ such that for all $0<\epsilon<\epsilon_0$, 
    \begin{align}\label{D-L2-6}
        \int_{K} \bigg|f&(x,u_\epsilon, D_s^p u_\epsilon, D_{a,t}^q u)-f(x,u_\epsilon, D_s^p u_\epsilon, D_{a,t}^q u_\epsilon) \bigg|v dx\nonumber\\
        &\leq C\|v\|_{L^\infty(\Omega)}\left(\int_{K}\int_{\mathbb{R}^N} a(x,y)\big(|U(x,y)|^{q}+|U_\epsilon(x,y)-U(x,y)|^{q}\big) K_{t,q}(x,y)dydx\right)^{\frac{q-1}{q}} \nonumber\\
        & \ \ \ \ \ \left(\int_{K}\int_{\mathbb{R}^N} a(x,y)|U_\epsilon(x,y)-U(x,y)|^{q}K_{t,q}(x,y)dydx\right)^{\frac{1}{q}} \nonumber\\
        &\leq \frac{\delta}{4}.
    \end{align}
    %Similarly, from the proof of the equation (3.37) in Lemma 3.6 of \cite{BM2021}
    Using a similar approach, we obtain an $\epsilon_1>0$ such that for all $0<\epsilon<\epsilon_1$, 
    \begin{equation}\label{D-L2-7}
        \int_{K} \left|f(x,u_\epsilon, D_s^p u, D_{a,t}^q u)-f(x,u_\epsilon, D_s^p u_\epsilon, D_{a,t}^q u) \right|v dx\leq \frac{\delta}{4}.
    \end{equation}
     Since $f(x,t,\zeta,\eta)$ is uniformly continuous in $\Omega \times \mathbb{R}^3$, there {exists} $r>0$ such that for every $\epsilon>0$, we have
    \begin{equation}\label{D-L2-8}
        \left|f(x,u_\epsilon, D_s^p u, D_{a,t}^q u)-f(y,u_\epsilon, D_s^p u, D_{a,t}^q u) \right|\leq \frac{\delta}{4|K|\|v\|_{L^{\infty}(K)}}, \text{ whenever } |x-y|<r.
    \end{equation}
    Then, by Lemma \ref{D-u_ep}, we choose $\epsilon_2>0$ sufficiently small such that for all $0<\epsilon<\epsilon_2$, we have $r(\epsilon)<r$. Then, by \eqref{D-f_ep} and \eqref{D-L2-8}, we have
    \begin{equation*}
        |f(x,u_\epsilon, D_s^p u_\epsilon, D_{a,t}^q u_\epsilon)-f_\epsilon(x,u_\epsilon, D_s^p u_\epsilon, D_{a,t}^q u_\epsilon)|< \frac{\delta}{4|K|\|v\|_{L^{\infty}(K)}}, \ x\in K.
    \end{equation*}
    Therefore, \eqref{D-L2-7} becomes
    \begin{equation}\label{D-L2-9}
        \int_{K} \left|f(x,u_\epsilon, D_s^p u_\epsilon, D_{a,t}^q u_\epsilon)-f_\epsilon(x,u_\epsilon, D_s^p u_\epsilon, D_{a,t}^q u_\epsilon) \right|v dx <\frac{\delta}{4}.
    \end{equation}
    Choose $m>0$ such that $\|u_\epsilon\|_{L^\infty(K)}\leq m$ for all $\epsilon<1$. Now, by Lemma 3.1(ii) in \cite{BM2021} and the continuity of $f$, we have $|f(x,u, D_s^p u, D_{a,t}^q u)-f(x,u_\epsilon, D_s^p u, D_{a,t}^q u)|\rightarrow 0$ pointwise a.e. in $K$. Also, by \eqref{D-T2-f} and since $\gamma_i$ are bounded in $[-m,m]$, we get a $C>0$ such that whenever $0<\epsilon<1$, 
    \begin{align*}
        \left|f(x,u, D_s^p u, D_{a,t}^q u)-f(x,u_\epsilon, D_s^p u, D_{a,t}^q u) \right|v \leq \left(C(|D_s^p u|^{\frac{p-1}{p}}+|D_{a,t}^q u|^{\frac{q-1}{1}})+h(x) \right)v,
    \end{align*}
    which is integrable on $K$. Therefore, by the dominated convergence theorem, we obtain $0<\epsilon_3<1$ such that for all $0<\epsilon<\epsilon_3$,
    \begin{equation}\label{D-L2-10}
        \int_{K} \left|f(x,u, D_s^p u, D_{a,t}^q u)-f(x,u_\epsilon, D_s^p u, D_{a,t}^q u) \right|v dx <\frac{\delta}{4}.
    \end{equation}
    Choose $\epsilon'=\min\{\epsilon_i:i=0,1,2,3\}$. Then, for all $\epsilon<\epsilon'$, using \eqref{D-L2-1}, \eqref{D-L2-6}, \eqref{D-L2-7}, \eqref{D-L2-9} and \eqref{D-L2-10}, we have
    \begin{equation*}
         \left|\int_{\Omega}\Big[f(x,u, D_s^p u, D_{a,t}^q u)-f_\epsilon(x,u_\epsilon, D_s^p u_\epsilon, D_{a,t}^q u_\epsilon)\Big]v dx  \right|<\delta \text{ for all } \epsilon<\epsilon'.
    \end{equation*}
    Since $\delta>0$ is arbitrary, this proves the desired result.
\end{proof}

\noindent The following lemma proves that each infimal convolution $u_\epsilon$ of the viscosity supersolution $u$ to problem \eqref{D} is a {viscosity supersolution} to problem \eqref{D} with $f$ replaced by the function $f_\epsilon$ given by \eqref{D-f_ep} on a subset of $\Omega$. 

\begin{lemma}\label{D-L3}
    Let $u$ be a viscosity supersolution to problem \eqref{D}, where $f=f(x,t,\zeta,\eta)$ is non-increasing in $t$ and continuous in $\Omega\times \mathbb{R}^3$. Then, the function $u_\epsilon$ given by Definition \ref{D-IC} is a viscosity supersolution to the problem 
   \begin{align}\label{D-L3-prob1}
    L_a v(x)=f_\epsilon(x,v,D_s^p u, D_{a,t}^q v) \text{ in } \Omega_{r(\epsilon)},
\end{align}
where $f_\epsilon$ is defined by \eqref{D-f_ep}.
\end{lemma}
\begin{proof}
    Let $a\in B_{r(\epsilon)}(0)$. Define the function $w_a:\mathbb{R}^N \rightarrow\mathbb{R}$ by
    $$w_a(x)=u(x+a)+\frac{|a|^l}{l\epsilon^{l-1}}.$$
    Let $B_r(x_0)\subset \Omega_{r(\epsilon)}$. Now consider any function $\psi_0 \in  L_{s,p}^{p-1}(\mathbb{R}^N) \cap L_{a,t,q}^{q-1}(\mathbb{R}^N) \cap C^2(B_r(x_0))$ satisfying $\psi_0 \leq w_a$ in $\mathbb{R}^N, \ \psi_0(x_0)=w_a(x_0)$ and either
    \begin{itemize}
        \item $p>\frac{2}{2-s}$ or $\nabla \psi_0(x_0)\neq 0$, or
        \item $p\leq \frac{2}{2-s}$, $x_0$ is an isolated critical point of $\psi_0$ and $\psi_0 \in C_\beta^2(B_r(x_0))$ for a $\beta>\frac{sp}{p-1}$.
    \end{itemize}
    Now, define $\psi_a:B_r(x_0)\rightarrow\mathbb{R}$ by
    $$\psi_a(y)=\psi_0(y-a)-\frac{|a|^l}{l\epsilon^{l-1}}.$$
    For $y\in \mathbb{R}^N$, let $x=y-a$ and let $y_0=x_0+a \in \Omega$. Then, 
    \begin{align*}
        \psi_a(y)=\psi_0(y-a)-\frac{|a|^l}{l\epsilon^{l-1}}\leq w_a(x)-\frac{|a|^l}{l\epsilon^{l-1}}=u(y).
    \end{align*}
    Similarly, $\psi_a(y_0)=u(y_0)$. Since $u$ is a viscosity supersolution to \eqref{D}, we have
    \begin{align*}
        L_a \psi_0(x_0)=L_a\psi_a(y_0)&\geq f(y_0, \psi_a(y_0),D_s^p \psi_a(y_0), D_{a,t}^q \psi_a(y_0))\nonumber\\
        &=f (x_0+a, \psi_0(x_0)-\frac{|a|^l}{l\epsilon^{l-1}},D_s^p \psi_0(x_0), D_{a,t}^q \psi_0(x_0)) \nonumber\\
        &\geq f (x_0+a, \psi_0(x_0),D_s^p \psi_0(x_0), D_{a,t}^q \psi_0(x_0)) \nonumber\\
        &\geq f_\epsilon(x_0, \psi_0(x_0),D_s^p \psi_0(x_0), D_{a,t}^q \psi_0(x_0)).
    \end{align*}
    Hence, $w_a$ is a viscosity supersolution to \eqref{D-L3-prob1}. Now, by Lemma 3.1(v) in \cite{BM2021}, there {exists} $z\in B_{r(\epsilon)}(0)$ such that $u_\epsilon(x_0)=w_z(x_0)$. Clearly, $u_\epsilon\leq w_z$. Now, to prove that $u_\epsilon$ is a viscosity supersolution, let $B_r(x_0)\subset \Omega_{r(\epsilon)}$ and consider a function $\psi \in  L_{s,p}^{p-1}(\mathbb{R}^N) \cap L_{a,t,q}^{q-1}(\mathbb{R}^N) \cap C^2(B_r(x_0))$ such that $\psi \leq u_\epsilon,\ \psi(x_0)=u_\epsilon(x_0)$ and either one of the following two conditions are satisfied.
    \begin{itemize}
        \item $p>\frac{2}{2-s}$ or $\nabla \psi(x_0)\neq 0$, or
        \item $p\leq \frac{2}{2-s}$, $x_0$ is an isolated critical point of $\psi$ and $\psi \in C_\beta^2(B_r(x_0))$ for a $\beta>\frac{sp}{p-1}$.
    \end{itemize}
    Then, $\psi \leq u_\epsilon \leq w_z$. Since $w_z$ is a viscosity supersolution to \eqref{D-L3-prob1}, we deduce 
    $$L_a\psi(x_0)\geq f_\epsilon(x_0, \psi(x_0),D_s^p \psi(x_0), D_{a,t}^q \psi(x_0)).$$
    Hence, the proof is complete.
\end{proof}

\noindent In the next {lemma}, we establish an inequality for the weak supersolutions of problem \eqref{D} which is vital to prove Theorem \ref{D-T2}.

\begin{lemma}\label{D-L4}
    Let $u$ be a bounded weak supersolution to {problem} \eqref{D}, where $f$ is a continuous function such that \eqref{D-T2-f} holds. For $\phi \in C_c^\infty(\Omega)$ with $\operatorname{supp}\phi=K$ and $0\leq \phi\leq 1$, the following inequality is satisfied: 
    \begin{align*}
        \int_K \int_{\mathbb{R}^N}&|u(x)-u(y)|^p K_{s,p}(x,y)\phi(x)^qdydx+\int_K \int_{\mathbb{R}^N} a(x,y) |u(x)-u(y)|^q K_{t,q}(x,y)\phi(x)^{q}dydx \\
       &\leq C\max\{(\operatorname{osc}u)^p,(\operatorname{osc}u)^q,\operatorname{osc}u\}\bigg(\iint_{Q(K)} |\phi(x)-\phi(y)|^p K_{s,p}(x,y)dydx\\
       &\ \ \ +\iint_{Q(K)}a(x,y) |\phi(x)-\phi(y)|^q K_{t,q}(x,y)dydx+1\bigg),
    \end{align*}
    where $C=C(p,q,K,\gamma_1, \gamma_2,h,\Lambda_1,\Lambda_2)>0$ and $\operatorname{osc}u=\sup\limits_{\mathbb{R}^N}u-\inf\limits_{\mathbb{R}^N}u$. 
\end{lemma}
\begin{proof}
    Given $\phi \in C_c^\infty(\Omega)$ satisfying $0\leq \phi\leq 1$ with support $K$. Define, the function $v$ as
    \begin{equation*}
        v(x):=\begin{cases}
            \phi(x)^q\left(A -u(x)\right),& x\in \Omega,\\
            0, & x \notin \Omega,
        \end{cases}
    \end{equation*}
    where $A=\sup\limits_{\mathbb{R}^N} u$. Then, we have
    \begin{equation}\label{D-L4-1}
        v(x)-v(y)=-(u(x)-u(y))\phi(x)^q+(\phi(x)^q-\phi(y)^q)(A-u(y)).
    \end{equation}
    Since $u$ is a weak supersolution to problem \eqref{D} and $v \geq 0$, we obtain
    \begin{align}\label{D-L4-2}
        H_a(u,v)\geq \int_K f(x,v,D_s^p u, D_{a,t}^q u)v(x)dx.
    \end{align}
    Substituting \eqref{D-L4-1} in \eqref{D-L4-2}, we deduce
    \begin{align}\label{D-L4-3}
        \int_K \int_{\mathbb{R}^N} |u(x)-u(y)|^p &K_{s,p}(x,y)\phi(x)^qdydx+\int_K \int_{\mathbb{R}^N} a(x,y) |u(x)-u(y)|^q K_{t,q}(x,y)\phi(x)^qdydx \nonumber\\
        &\leq \int_{\mathbb{R}^N}\int_{\mathbb{R}^N} h_p(u(x)-u(y))(\phi(x)^q-\phi(y)^q)K_{s,p}(x,y)B dydx\nonumber\\
        &\ \ \ +\int_{\mathbb{R}^N}\int_{\mathbb{R}^N} a(x,y)h_q(u(x)-u(y))(\phi(x)^q-\phi(y)^q)K_{t,q}(x,y)B dydx\nonumber\\
        &\ \ \ +\int_K |f(x,v,D_s^p u, D_{a,t}^q u)|v(x)dx,
    \end{align}
    where $B=\operatorname{osc}u$. Observe that using Lemma 3.4 in \cite{KKL2019}, for $a,b \in \mathbb{R}$,
    \begin{equation}\label{D-L4-5}
       \left||a|^l -|b|^l\right|\leq C(l)(|a|^{l-1}+|b|^{l-1})|a-b|, \ l>1.
    \end{equation}
    Now for $\epsilon>0$, we have the Young's inequality
    \begin{equation}\label{D-L4-4}
        ab\leq \epsilon a^{l} +C(\epsilon,l)b^{\frac{l}{l-1}}, \ a,b\geq 0,\ l>1.
    \end{equation}
    Applying \eqref{D-L4-4} with $a=|h_q(u(x)-u(y))|\phi(x)^{q-1}$, $b=|\phi(x)-\phi(y)|B$, and $l=\frac{q}{q-1}$, we get
    \begin{equation}\label{D-L4-6}
        h_q(u(x)-u(y))\phi(x)^{q-1}|\phi(x)-\phi(y)|B \leq \epsilon|u(x)-u(y)|^q\phi(x)^q+C(\epsilon,q)|\phi(x)-\phi(y)|^q B^q.
    \end{equation}
    Similarly, we obtain
    \begin{equation}\label{D-L4-7}
         h_q(u(x)-u(y))\phi(y)^{q-1}|\phi(x)-\phi(y)|B \leq \epsilon|u(x)-u(y)|^q\phi(y)^q+C(\epsilon,q)|\phi(x)-\phi(y)|^q B^q.
    \end{equation}
    Using \eqref{D-L4-5} with $l=q,\ a=\phi(x),\ b=\phi(y)$, and applying \eqref{D-L4-6}, \eqref{D-L4-7}, we deduce
    \begin{align}\label{D-L4-8}
        h_q(u(x)-u(y))(\phi(x)^q-\phi(y)^q)B &\leq C(q)\epsilon|u(x)-u(y)|^q(\phi(x)^q+\phi(y)^q)\nonumber\\
        &\ \ \ +C(\epsilon,q)|\phi(x)-\phi(y)|^q B^q.
    \end{align}
    Using the symmetry of functions in the RHS of \eqref{D-L4-8}, we get
    \begin{align}\label{D-L4-8'}
        \int_{\mathbb{R}^N}\int_{\mathbb{R}^N} &a(x,y)h_q(u(x)-u(y))(\phi(x)^q-\phi(y)^q)K_{t,q}(x,y)B dydx \nonumber\\
        &\leq C(q)\epsilon \int_K\int_{\mathbb{R}^N}a(x,y)|u(x)-u(y)|^q\phi(x)^qK_{t,q}(x,y)dydx\nonumber\\
        & \ \ \ +C(q,\epsilon)B^q\iint_{Q(K)}a(x,y)|\phi(x)-\phi(y)|^q K_{t,q}(x,y)dydx.
    \end{align}
    Note that $0\leq \phi \leq 1$ and $q\leq \frac{p(q-1)}{(p-1)}$. Thus, we use \eqref{D-L4-4} with $a=h_p(u(x)-u(y))\phi(x)^{q-1}$, $b=|\phi(x)-\phi(y)|B$ and and $l=\frac{p}{p-1}$ to get
    \begin{align}\label{D-L4-9}
        h_p(u(x)-u(y))\phi(x)^{q-1}|\phi(x)-\phi(y)|B &\leq \epsilon|u(x)-u(y)|^p\phi(x)^{\frac{p(q-1)}{(p-1)}}+C(\epsilon,p)|\phi(x)-\phi(y)|^p B^p\nonumber\\
        &\leq \epsilon|u(x)-u(y)|^p\phi(x)^{q}\nonumber\\
        &\ \ \ +C(\epsilon,p)|\phi(x)-\phi(y)|^p B^p.
    \end{align}
    Similarly, we have 
    \begin{align}\label{D-L4-10}
        h_p(u(x)-u(y))\phi(y)^{q-1}|\phi(x)-\phi(y)|B \leq \epsilon|u(x)-u(y)|^p\phi(y)^{q}+C(\epsilon,p)|\phi(x)-\phi(y)|^p B^p.
    \end{align}
    Utilizing \eqref{D-L4-5} with $l=q,\ a=\phi(x),\ b=\phi(y)$ and using \eqref{D-L4-9}, \eqref{D-L4-10}, we obtain
    \begin{align}\label{D-L4-11}
        h_p(u(x)-u(y))(\phi(x)^q-\phi(y)^q)B &\leq C(q)\epsilon|u(x)-u(y)|^p(\phi(x)^{q}+\phi(y)^{q})\nonumber\\
        &\ \ \ +C(\epsilon,p,q)|\phi(x)-\phi(y)|^p B^p.
    \end{align}
    The symmetry of functions in the RHS of \eqref{D-L4-11} gives
    \begin{align}\label{D-L4-11'}
        \int_{\mathbb{R}^N}\int_{\mathbb{R}^N}& h_p(u(x)-u(y))(\phi(x)^q-\phi(y)^q)BK_{s,p}(x,y)B dydx\nonumber\\
        &\leq C(q)\epsilon\int_K\int_{\mathbb{R}^N}|u(x)-u(y)|^p\phi(x)^{q}K_{s,p}(x,y)dydx\nonumber\\
        &\ \ \ +C(\epsilon,p,q)B^p\iint_{Q(K)}|\phi(x)-\phi(y)|^pK_{s,p}(x,y)dydx.
    \end{align}
    Note that by \eqref{D-T2-f}, we have
    \begin{align}\label{D-L4-12}
        \int_K |f&(x,v,D_s^p u, D_{a,t}^q u)|v(x)dx \nonumber\\
        &\leq \int_K C_0 \left(|D_s^p u|^{\frac{p-1}{p}}+|D_{a,t}^q u|^{\frac{q-1}{q}}\right)B\phi(x)^qdx+\|h\|_{L^\infty(K)}|K|B,
    \end{align}
    where $C_0=\max\left\{\gamma_1(t)+\gamma_2(t),: t\in\left[-\|u\|_{L^\infty(\mathbb{R}^N)}, \|u\|_{L^\infty(\mathbb{R}^N)}\right]\right\}$. Making use of \eqref{D-L4-4} with $l=\frac{p}{p-1}$ $a=|D_s^p u|^{\frac{p-1}{p}}\phi(x)^{q-1}$ and $b=B\phi(x)$, and since $|\phi|\leq 1$ we have
    \begin{align}\label{D-L4-13}
         \int_K |D_s^p u|^{\frac{p-1}{p}}B\phi^q(x)dx&\leq \epsilon\int_K |D_s^p u|\phi(x)^{\frac{p(q-1)}{p-1}}dx +C'B^p\nonumber\\
         &\leq \epsilon C(\Lambda_1)  \int_K \int_{\mathbb{R}^N} |u(x)-u(y)|^p K_{s,p}(x,y)\phi(x)^qdydx+C'B^p.
    \end{align}
    where $C'=C'(\epsilon, p,K)$.  Similarly, we get
    \begin{align}\label{D-L4-14}
         \int_K |D_{a,t}^q u|^{\frac{q-1}{q}}B\phi^q(x)dx&\leq \epsilon\int_K |D_{a,t}^q u|\phi(x)^{q}dx +C(\epsilon,q, K)B^q\nonumber\\
         &\leq \epsilon C(\Lambda_2) \int_K \int_{\mathbb{R}^N} a(x,y)|u(x)-u(y)|^q K_{t,q}(x,y)\phi(x)^{q}dydx\nonumber\\
         &\ \ \ +C(\epsilon,q, K)B^q.
    \end{align}
    Using \eqref{D-L4-13} and \eqref{D-L4-14} in \eqref{D-L4-12}, we deduce
   \begin{align}\label{D-L4-15}   
      \int_K& |f(x,v,D_s^p u, D_{a,t}^q u)|v(x)dx \nonumber\\
      &\ \ \ \leq C(K,\gamma_1, \gamma_2,\Lambda_1,\Lambda_2)\bigg( \epsilon\int_K \int_{\mathbb{R}^N} |u(x)-u(y)|^p K_{s,p}(x,y)\phi(x)^qdydx \nonumber\\
      &\ \ \ +\epsilon\int_K \int_{\mathbb{R}^N} a(x,y)|u(x)-u(y)|^q K_{t,q}(x,y)\phi(x)^{q}dydx\nonumber\\
      &\ \ \ +C(\epsilon,p,q, K)\max\{B^p,B^q\}+\|h\|_{L^\infty(K)}B\bigg).
   \end{align}
   Combining \eqref{D-L4-3}, \eqref{D-L4-8'}, \eqref{D-L4-11'} and \eqref{D-L4-15} and choosing $\epsilon$ sufficiently small, we have
   \begin{align*}
       \int_K \int_{\mathbb{R}^N}&|u(x)-u(y)|^p K_{s,p}(x,y)\phi(x)^q dydx+\int_K \int_{\mathbb{R}^N} a(x,y) |u(x)-u(y)|^q K_{t,q}(x,y)\phi(x)^{q}dydx \\
       &\leq C(p,q,K,\gamma_1, \gamma_2,\Lambda_1,\Lambda_2,h)\max\{B^p,B^q,B\}\bigg(\iint_{Q(K)} |\phi(x)-\phi(y)|^p K_{s,p}(x,y)dydx\\
       &+\iint_{Q(K)}a(x,y) |\phi(x)-\phi(y)|^q K_{t,q}(x,y)dydx+1\bigg).
   \end{align*}
\end{proof}

\noindent The upcoming two {lemmas} help to prove Lemma \ref{D-L5} which is crucial in the proof of Theorem \ref{D-T2}. We first consider the case $q>\frac{2}{2-t}$.
\begin{lemma}\label{D-L6}
    Let $q>\frac{2}{2-t}$, $u\in L_{a,t,q}^{q-1}(\mathbb{R}^N)$ be bounded and $u_\epsilon$ be given by Definition \ref{D-IC}. Then, for any $\phi \in C_c^\infty(\Omega_{r(\epsilon)})$ with $\phi\geq 0$ and $K=\operatorname{supp}\phi$, we have
   \begin{align*}
        \int_K\bigg(2\operatorname{P.V.}\int_{\mathbb{R}^N}a(x,y)&h_q(u_\epsilon(x)-u_\epsilon(y))K_{t,q}(x,y)dy\bigg)\phi(x)dx \\
        &\leq \iint_{Q(K)}(a(x,y)h_q(u_\epsilon(x)-u_\epsilon(y))(\phi(x)-\phi(y))K_{t,q}(x,y)dydx.
    \end{align*}
\end{lemma}
\begin{proof}
    Throughout the proof, we denote
    $$G(v)(x)=2\operatorname{P.V.}\int_{\mathbb{R}^N}a(x,y)h_q(v(x)-v(y))K_{t,q}(x,y)dy.$$
    For each $d>0$, consider the standard mollifier $\rho$ and let $\rho_d(x)=d^{-N}\rho(\frac{x}{d})$. Then, define $v_d:\mathbb{R}^N\rightarrow\mathbb{R}$ by
    \begin{equation}\label{D-L6-1'}
        v_d(x):=\begin{cases}
            \int_{B_d(0)}\rho_d(z)u_\epsilon(x-z)dz, \ &x\in \Omega_{r(\epsilon)},\\
            u_\epsilon(x), &\text{ otherwise. }
        \end{cases}
    \end{equation}
    Then, $v_d\in C^2(\Omega_{r(\epsilon)})\cap L_{a,t,q}^{q-1}(\mathbb{R}^N)$ is smooth in $\Omega_{r(\epsilon)}$ for every $d>0$ and $v_d \rightarrow u_\epsilon$ pointwise a.e. 
    Now for $x\in K$, consider
    \begin{align}\label{D-L6-2}
        G(v_d)(x)=2(I_{1,r}v_d(x)+I_{2,r}v_d(x)),
    \end{align}
    where
    \begin{align*}
        I_{1,r}v(x)&=\int_{\mathbb{R}^N\setminus B_r(x)}a(x,y)h_q(v(x)-v(y))K_{t,q}(x,y)dy\\
        I_{2,r}v(x)&=\operatorname{P.V.}\int_{B_r(x)}a(x,y)h_q(v(x)-v(y))K_{t,q}(x,y)dy\nonumber\\
        %&=\lim\limits_{r'\rightarrow 0}\int_{B_r(x)}a(x,y)\left(h_q(v(x)-v(y))-\chi_{B_{r'}(x)}h_q(v(x)-v(y))\right)K_{t,q}(x,y)dy.\label{D-L6-5}
        &=\int_{B_r(x)}a(x,y)\left(h_q(v_d(x)-v_d(y))-\nabla v_d(x)\cdot(x-y)\right)K_{t,q}(x,y)dy.
    \end{align*}
    %By the odd symmetry of the integrand, for each $r'>0$ we obtain
    %\begin{align}\label{D-L6-6}
     %   \int_{B_r(x_0)}a(x,y)(1-\chi(B_{r'}(x))\nabla v_d(x)\cdot(x-y)dy=0.
    %\end{align}
    %Let $$ m=\max\{2r|\nabla v_d(x)|,\sup\limits_{B_{r}(x)}|v_d(x)-v_d(y)| \}$$. Then, $|a(x,y)\chi_{B_{r'}(x)}h_q(u(x)-u(y))K_{t,q}(x,y)|\leq Mm K_{t,q}(x,y)$ which is integrable in $B_{r}(x)$. 
   %Thus by using dominated convergence theorem and from \eqref{D-L6-5}, \eqref{D-L6-6}, we deduce
    %\begin{align}\label{D-L6-3}
     %    I_{2,r}=\int_{B_r(x)}a(x,y)\left(h_q(v_d(x)-v_d(y))-\nabla v_d(x)\cdot(x-y)\right)K_{t,q}(x,y)dy.
    %\end{align}
    Since $u_\epsilon\in L_{a,t,q}^{q-1}(\mathbb{R}^N)\cap L^\infty(\mathbb{R}^N)$,  we get a constant $C_1>0$ such that 
    \begin{align}\label{D-L6-7}
        |I_{1,r}v_d(x)|&\leq C\int_{\mathbb{R}^N\setminus B_r(x)}a(x,y)(1+|u_\epsilon(y)|^{q-1})K_{t,q}(x,y)dy\nonumber\\
        & \leq C_1, \text{ whenever }x\in K.
    \end{align}
    Then, from Lemma 2.4 in \cite{BM2021}, we have
    \begin{equation}\label{D-L6-9}
        h_q(t_1)-h_q(t_2)=(q-1)(t_1-t_2)\int_0^1\left|\tau t_1+(1-\tau)t_2)\right|^{q-2}d\tau, \ t_1,t_2\in\mathbb{R}.
    \end{equation}
    Also, from the Taylor's expansion of $v_d$, we get a $z_y\in \Omega$ such that
    \begin{equation}\label{D-L6-10}
        v_d(y)=v_d(x)+\nabla v_d(x)\cdot(y-x)+D^2v_d(z_y)(y-x)^2.
    \end{equation}
    Using \eqref{D-L6-9} with $t_1=v_d(x)-v_d(y)$ and $t_2=\nabla v_d(x)\cdot(x-y)$ and applying \eqref{D-L6-10} in $ I_{2,r}v_d(x)$, we deduce 
    \begin{align}\label{D-L6-8}
        I_{2,r}v_d(x)= (q&-1)\int_{B_r(x)}a(x,y)(-D^2v_d(z_y))(y-x)^2 \nonumber\\
        &\times\int_0^1\left|\tau(v_d(x)-v_d(y))+(1-\tau)\nabla v_d(x)\cdot(x-y)\right|^{q-2}d\tau K_{t,q}(x,y)dy.
    \end{align}
    By Lemma 3.1 in \cite{BM2021}, we {gets} $c>0$ such that $|D^2v_d|_{L^{\infty}(K)},\|\nabla v_d\|_{L^{\infty}(K)}\leq c$ in $\Omega_{r(\epsilon)}$ for all $d>0$. Then, for $q\geq 2$, utilizing Lemma 2.4 of \cite{BM2021}, we get
    \begin{equation}\label{D-L6-11}
        \int_0^1\left|\tau(v_d(x)-v_d(y))+(1-\tau)\nabla v_d(x)\cdot(x-y)\right|^{q-2}d\tau \leq C|x-y|^{q-2}.
    \end{equation}
    Also, define
    $$A_{r,x}=B_r(x)\cap\{y:D^2v_d(z_y)\geq 0\}.$$
    Thus, from \eqref{D-L6-8} and \eqref{D-L6-11}, we obtain
    \begin{align}\label{D-L6-12}
      I_{2,r}v_d(x)&\geq -M\Lambda_2 C(q-1)\int_{A_{r,x}}|x-y|^{q-(N+tq)}dy\nonumber\\
      &\geq -M\Lambda_2 C(q-1)\int_{B_r(x)}|x-y|^{q-(N+tq)}dy=-C_2 \text{ for all }d>0.
    \end{align}
    If $\frac{2}{2-t}<q<2$, from Lemma 2.4 of \cite{BM2021}, we have
    \begin{align}\label{D-L6-13}
        &\left|\int_0^1\left|\tau(v_d(x)-v_d(y))+(1-\tau)\nabla v_d(x)\cdot(x-y)\right|^{q-2}d\tau \right|\nonumber\\
        &\hspace{1.5cm}\leq C|v_d(x)-v_d(y)-\nabla v_d(x)\cdot(x-y)|^{q-2}\nonumber\\
        &\hspace{1.5cm}=|D^2v_d(z_y)(y-x)^2|^{q-2}\nonumber\\
        &\hspace{1.5cm}\leq C|y-x|^{2(q-2)}.
    \end{align}
    Using \eqref{D-L6-13} in \eqref{D-L6-8}, we deduce the for all $d>0$,
    \begin{equation}\label{D-L6-14}
        I_{2,r}v_d(x)\geq -M\Lambda_2(q-1)C\int_{B_r(x)}|x-y|^{2(q-2)+2-(N+tq)}=-C_3.
    \end{equation}
    Therefore, from \eqref{D-L6-7}, \eqref{D-L6-12} and \eqref{D-L6-14}, for $x\in K$, we have
    \begin{equation*}
        G(v_d)(x)\geq -2(C_1+C_2+C_3).
    \end{equation*}
    Then, applying the Fatou's lemma, we deduce
    \begin{equation}\label{D-L6-16}
        \int_{K}\liminf\limits_{d\rightarrow 0}G(v_d) \phi dx \leq \liminf\limits_{d\rightarrow 0}\int_{K}G(v_d) \phi dx.
    \end{equation}
    Now, from \eqref{D-L6-7}, \eqref{D-L6-12} and \eqref{D-L6-14}, we get a constant $C>0$ such that for all $d>0$,
    \begin{align}\label{D-L6-17}
        &\chi_{\mathbb{R}^N \setminus B_r(x)}a(x,y)h_q(v(x)-v(y))K_{t,q}(x,y)\nonumber\\
        &+\chi_{B_r(x)}a(x,y)\bigg(h_q(v_d(x)-v_d(y))-\nabla v_d(x)\cdot(x-y)\bigg)K_{t,q}(x,y) \\
        & \geq \begin{cases}
            -C\left(a(x,y)(1+|u_\epsilon(y)|^{q-1})|x-y|^{-(N+tq)} +|x-y|^{2(q-1)-(N+tq)}\right), \frac{2}{2-t}<q< 2,\\
            -C\left(a(x,y)(1+|u_\epsilon(y)|^{q-1})|x-y|^{-(N+tq)} +|x-y|^{q-(N+tq)}\right), \hspace{0.8cm} q\geq 2,            
        \end{cases}\nonumber
    \end{align}
     which is integrable since $u_\epsilon\in L_{a,t,q}^{q-1}(\mathbb{R}^N)$. We have, $v_d\rightarrow u_\epsilon$ pointwise in $\mathbb{R}^N$. Therefore, from \eqref{D-L6-2} and \eqref{D-L6-17}, using Fatou's lemma, we obtain 
     \begin{align}\label{D-L6-18}
         G(u_\epsilon)(x)&=2\operatorname{P.V.}\int_{\mathbb{R}^N}\liminf\limits_{d\rightarrow 0} a(x,y)h_q(v_d(x)-v_d(y))K_{t,q}(x,y)dy \nonumber\\
         &\leq \liminf\limits_{d\rightarrow 0}2\operatorname{P.V.}\int_{\mathbb{R}^N} a(x,y)h_q(v_d(x)-v_d(y))K_{t,q}(x,y)dy \nonumber\\
         &= \liminf\limits_{d\rightarrow 0}G(v_d)(x) 
     \end{align}
     for every $x\in K$. Substituting \eqref{D-L6-18} in \eqref{D-L6-16}, we have
     \begin{equation}\label{D-L6-19}
         \int_{K}G(u_\epsilon) \phi dx \leq \liminf\limits_{d\rightarrow 0}\int_{K}G(v_d) \phi dx.
     \end{equation}
     Now, since $v_d\in C^2(\Omega_{r(\epsilon)})\cap L_{a,t,q}^{q-1}(\mathbb{R}^N)$ is smooth in $\Omega_{r(\epsilon)}$ for every $d>0$, we have
     \begin{equation}\label{D-L6-20}
         \int_{K}G(v_d) \phi dx=\iint_{Q(K)}f_d(x,y)dydx,
     \end{equation}
     where $$f_d(x,y)=a(x,y)h_q(v_d(x)-v_d(y))(\phi(x)-\phi(y))K_{t,q}(x,y).$$
     Choose $0<\nu<\operatorname{dist}(K,\partial\Omega_{r(\epsilon)})$ and define $A:=\{x\in \Omega_{r(\epsilon)}: \operatorname{dist}(x,K)\leq \nu\}$. %Since $\operatorname{supp}(\phi)\subset A$, and the integrand in  \eqref{D-L6-20} is symmetric, \eqref{D-L6-20} becomes 
     %\begin{align}\label{D-L6-21}
      %   \int_{\mathbb{R}^N}G(v_d) \phi dx=\iint_{K\times A}&a(x,y)h_q(v_d(x)-v_d(y))(\phi(x)-\phi(y))K_{t,q}(x,y)dydx\nonumber\\
       %  &+2\iint_{K\times (\mathbb{R}^N\setminus K)}a(x,y)h_q(v_d(x)-v_d(y))(\phi(x)-\phi(y))K_{t,q}(x,y)dydx.
     %\end{align}
     Note that $a(x,y)\leq M$ and  $u_\epsilon\in L_{a,t,q}^{q-1}(\mathbb{R}^N)$ is bounded. Also, $|x-y|\geq \nu$ for all $x,y \in K\times(\mathbb{R}^N\setminus A)$. Thus, we obtain a constant $C>0$ such that
     \begin{equation}\label{D-L6-22}
         |f_d(x,y)|\leq Ca(x,y)\frac{1+|u_\epsilon(y)|^{q-1}}{|x-y|^{N+tq}}\in L^1(K\times(\mathbb{R}^N\setminus A)).
     \end{equation}
     Now consider $(x,y)\in K\times A$. Using \eqref{D-L4-4} with $a=\left(a(x,y)K_{t,q}(x,y)\right)^{\frac{q-1}{q}}|v_d(x)-v_d(y)|^{q-1}$, $b=\left(a(x,y)K_{t,q}(x,y)\right)^{\frac{1}{q}}|\phi(x)-\phi(y)|,\ \epsilon=\frac{1}{2}$, and $l=\frac{q}{q-1}$, we get
     \begin{equation}\label{D-L6-21}
         |f_d(x,y)|\leq Ca(x,y)K_{t,q}(x,y)\left(|v_d(x)-v_d(y)|^q+|\phi(x)-\phi(y)|^q\right).
     \end{equation}
     Since $u_\epsilon$ is locally Lipschitz and $|\rho|\leq 1$, for small $d>0$ and $x\in K,\ y\in A$,  we have
     \begin{equation*}
         |v_d(x)-v_d(y)|\leq \int_{B_d(0)}|u_\epsilon(x-z)-u_\epsilon(y-z)|\rho(z)dz\leq C|x-y|.
     \end{equation*}
     Also, $\phi$ is Lipschitz. Thus, there {exists} $C>0$ such that
     \begin{equation*}
         \max\{|v_d(x)-v_d(y)|,|\phi(x)-\phi(y)|\}\leq C'|x-y|
     \end{equation*}
     Then from \eqref{D-L6-21}, we have
     \begin{align}\label{D-L6-23}
         |f_d(x,y)|\leq MC\Lambda_2\frac{1}{|x-y|^{N+tq-q}} \in L^1(K\times A).
     \end{align}
     Then, \eqref{D-L6-22} and \eqref{D-L6-23} from we have
     \begin{equation*}
         |f_d(x,y)|\leq C\bigg(\chi_{K\times A}\frac{1}{|x-y|^{N+tq-q}}+\chi_{K\times(\mathbb{R}^N)\setminus A)}a(x,y)\frac{1+|u_\epsilon(y)|^{q-1}}{|x-y|^{N+tq}}\bigg),
     \end{equation*}
     which is integrable in $K\times \mathbb{R}^N$. Therefore, by the dominated convergence theorem,
     \begin{equation}\label{D-L6-25}
         \lim\limits_{d  \rightarrow  0}\int_{K}\int_{\mathbb{R}^N}f_d(x,y)dydx=\int_{K}\int_{\mathbb{R}^N}a(x,y)h_q(u_\epsilon(x)-u_\epsilon(y))(\phi(x)-\phi(y)K_{t,q}(x,y)dydx.
     \end{equation}
     By the symmetry of $f_d(x,y)$, proceeding similarly, we get
     \begin{equation}\label{D-L6-26'}
         \lim\limits_{d  \rightarrow  0}\int_{\mathbb{R}^N\setminus K}\int_{K}f_d(x,y)dydx=\int_{\mathbb{R}^N\setminus K}\int_{K}a(x,y)h_q(u_\epsilon(x)-u_\epsilon(y))(\phi(x)-\phi(y))K_{t,q}(x,y)dydx.
    \end{equation}
     From \eqref{D-L6-25} and \eqref{D-L6-26'}, we have
    \begin{equation}\label{D-L6-26}
         \lim\limits_{d  \rightarrow  0}\iint_{Q(K)} f_d(x,y)dydx=\iint_{Q(K)}a(x,y)h_q(u_\epsilon(x)-u_\epsilon(y))(\phi(x)-\phi(y))K_{t,q}(x,y)dydx.
    \end{equation}
    Thus, taking limit infimum as $d\rightarrow 0$ in \eqref{D-L6-20}, and using \eqref{D-L6-19} and \eqref{D-L6-26}, we get the required result.
\end{proof}

\noindent Next, we consider the case $1<q\leq \frac{2}{2-t}$. The proof of Lemma \ref{D-L6} cannot be applied in this case since  $$\operatorname{P.V.}\int_{\mathbb{R}^N}a(x,y)h_q(v(x)-v(y))K_{t,q}(x,y)dy$$
may not be well defined for $v \in L_{a,t,q}^{q-1}(\mathbb{R}^N)$ if $v\notin C_\beta^2(\mathcal{D})$ for a $\beta>\frac{sp}{p-1}$ and an open set $\mathcal{D}$ with $x\in\mathcal{D}$.

\begin{lemma}\label{D-L7}
    Let $1<q\leq \frac{2}{2-t}$ and $u\in L_{a,t,q}^{q-1}(\mathbb{R}^N)$ be bounded and $u_\epsilon$ be given by Definition \ref{D-IC}. Then, for any $\phi \in C_c^\infty(\Omega_{r(\epsilon)})$ with $\phi\geq 0, \ K=\operatorname{supp}\phi$, we have
    \begin{align*}
         \int_K\bigg(2\operatorname{P.V.}\int_{\mathbb{R}^N}a(x,y)&h_q(u_\epsilon(x)-u_\epsilon(y))K_{t,q}(x,y)dy\bigg)\phi(x)dx \\
        &\leq \iint_{Q(K)}(a(x,y)h_q(u_\epsilon(x)-u_\epsilon(y))(\phi(x)-\phi(y))K_{t,q}(x,y)dydx.
    \end{align*}
\end{lemma}
\begin{proof}
    For each $v:\mathbb{R}^N \rightarrow\mathbb{R}$ {and $\delta>0$}, we denote
    \begin{equation*}
        G_\delta(v)(x)=2\int_{\mathbb{R}^N} a(x,y)\frac{h_q(v(x)-v(y))}{(|x-y|+\delta)^{N+tq}}dy
    \end{equation*}
    Let $v_d$ be defined by \eqref{D-L6-1'}. Since $v_d \in C^2(\Omega_{r(\epsilon)})\cap L_{a,t,q}^{q-1}(\mathbb{R}^N)$, using integration by parts, we have
    \begin{equation*}
        \int_{\mathbb{R}^N}G_\delta(v_d)\phi dx=\int_{\mathbb{R}^N}\int_{\mathbb{R}^N}a(x,y)\frac{h_q(v_d(x)-v_d(y))(\phi(x)-\phi(y))}{(|x-y|+\delta)^{N+tq}}dydx.
    \end{equation*}
    Proceeding similar to the proof of Lemma \ref{D-L6}, we get
    \begin{equation}\label{D-L7-2}
        \int_{\mathbb{R}^N}G_\delta(u_\epsilon)\phi dx \leq \int_{\mathbb{R}^N}\int_{\mathbb{R}^N}a(x,y)\frac{h_q(u_\epsilon(x)-u_\epsilon(y))(\phi(x)-\phi(y))}{(|x-y|+\delta)^{N+tq}}dydx.
    \end{equation}
    Fix an $x\in\Omega_{r(\epsilon)}$. Given $z\in \Omega_{r(\epsilon)}$, define $v_z:\Omega \rightarrow\mathbb{R}$ by
    $$v_z(y)=u(z)+\frac{|y-z|^l}{l\epsilon^{l-1}}.$$
    Since $\Omega$ is a bounded domain, $u$ is bounded and for $l\geq 2$, there exists a constant $m>0$ such that for all $y,z\in\Omega_{r(\epsilon)}$, we get
    \begin{align*}
        |v_z(y)|&\leq |u(z)|+\frac{|z-y|^l}{l\epsilon^{l-1}}\leq m, \\
        |\nabla v_z(y)|&=\frac{|z-y|^{l-1}}{\epsilon^{l-1}}\leq m,\\
        -mI&\leq -(l-1)\frac{|z-y|^{l-2}}{\epsilon^{l-1}}I\leq D^2v_z(y)\leq (l-1)\frac{|z-y|^{l-2}}{\epsilon^{l-1}}I\leq mI,
    \end{align*}
    where $I$ is the identity matrix. Define
    $$V(y)=v_z(x)+\nabla v_z(x)\cdot(y-x).$$
    Since $v_z\in C^2(\Omega_{r(\epsilon)})$, it can be easily seen that there exist $b_1,b_2\in \mathbb{R}^N$ lying in the line segment joining $x$ and $y$ such that $$v_z(y)-V(y)=v_z(y)-v_z(x)-\nabla v_z(x).(y-x)=\left(\nabla v_z(b_1)-\nabla v_z(x)\right).(y-x)=D^2 v_z(b_2).(x-y)^2.$$
   Consequently, using Proposition 2.5 in \cite{FZ2023}, we deduce
    \begin{align}\label{D-L7-4}
        \bigg|&h_q(v_{z}(x)-v_{z}(y))-h_q(V(x)-V(y))\bigg|\nonumber\\
        &\leq C\bigg(|V(x)-V(y)|+|v_{z}(x)-v_{z}(y)-(V(x)-V(y))|\bigg)^{p-2}|v_{z}(x)-v_{z}(y)-(V(x)-V(y))|\nonumber\\
        &\leq C\bigg(|\nabla v_z(x)\cdot(y-x)|+|v_{z}(y)-V(y)|\bigg)^{q-2}|v_{z}(y)-V(y)|\nonumber\\
        &=C\bigg(|\nabla v_z(x)\cdot(y-x)|+|D^2 v_z(b_2)\cdot(x-y)^2|\bigg)^{q-2}|D^2 v_z(b_2)\cdot(x-y)^2|.
    \end{align}
    Since $q<2$, it can be easily observed that given $\mu\geq0$, the map $\tau\mapsto(\mu+\tau)^{q-2}\tau$ is non-decreasing in $[0,\infty)$. Let $r>0$ such that $B_r(x)\subset \Omega_{r(\epsilon)}$, $\omega_y= \sup\{|D^2 v_z(b)|:|x-b|<|x-y|\}$ and $\omega= \sup\limits_{b\in \Omega_{r(\epsilon)}}|D^2 v_z(b)|$.
    By Lemma 4.2 in \cite{FZ2023} and \eqref{D-L7-4}, we get
    \begin{align}\label{D-L7-3}
        \bigg|\int_{B_r(x)}&a(x,y)\frac{h_q(v_{z}(x)-v_{z}(y))}{\left(|x-y|+\delta\right)^{N+tq}}dy \bigg|\nonumber\\
        &\leq \bigg|\int_{B_r(x)}a(x,y)\frac{h_q(V(x)-V(y))}{\left(|x-y|+\delta\right)^{N+tq}}dy \bigg|\nonumber\\
        &\ \ \ +\bigg|\int_{B_r(x)}a(x,y)\frac{\left(h_q(v_{z}(x)-v_{z}(y))-h_q(V(x)-V(y))\right)}{\left(|x-y|+\delta\right)^{N+tq}}dy \bigg| \nonumber\\
        &=\bigg|\int_{B_r(x)}a(x,y)\frac{\left(h_q(v_{z}(x)-v_{z}(y))-h_q(V(x)-V(y))\right)}{\left(|x-y|+\delta\right)^{N+tq}}dy \bigg|\nonumber\\
        %&\leq \bigg|\int_{B_r(x)}a(x,y)\frac{\left(h_q(v_{z}(x)-v_{z}(y))-h_q(V(x)-V(y))\right)}{\left|x-y|+\delta\right)^{N+tq}}dy \bigg|\nonumber\\
        &\leq C\bigg|\int_{B_r(x)}a(x,y)\frac{\bigg(|\nabla v_z(x)\cdot(y-x)|+\omega_y|x-y|^2\bigg)^{q-2}\omega_y|x-y|^2}{\left(|x-y|+\delta\right)^{N+tq}}dy \bigg|.
    \end{align}
    First assume $x\notin B_{r}(z)$. Then, $|\nabla v_z(x)|\geq \frac{r^{l-1}}{\epsilon^{l-1}}>0$. Applying Lemma 3.5 in \cite{KKL2019} and using the fact that $q<2$, \eqref{D-L7-3} becomes
    \begin{align}\label{D-L7-5}
        \bigg|\int_{B_r(x)}&a(x,y)\frac{h_q(v_{z}(x)-v_{z}(y))}{\left|x-y|+\delta\right)^{N+tq}}dy \bigg|\nonumber\\
        &\leq C\bigg|\int_{B_r(x)}a(x,y)\frac{\bigg(|\nabla v_z(x)\cdot(y-x)|+\omega_y|x-y|^2\bigg)^{q-2}\omega_y|x-y|^2}{|x-y|^{N+tq}}dy \bigg|\nonumber\\
        &=C\left|\int_{B_r(x)}a(x,y)\bigg(\left|\frac{\nabla v_z(x)}{|\nabla v_z(x)|}\cdot\frac{(y-x)}{|y-x|}\right|+\frac{\omega_y|x-y|}{|\nabla v_z(x)|}\bigg)^{q-2}\frac{|\nabla v_z(x)|^{q-2}\omega_y|x-y|^q}{|x-y|^{N+tq}}dy \right|\nonumber\\
        &\leq C\omega\int_0^r \left(1+\frac{\omega \tau}{|\nabla v_z(x)|}\right)^{q-2}|\nabla v_z(x)|^{q-2}\tau^{q-tq-1}d\tau\nonumber\\
        &\leq C\omega \left(\frac{r^{l-1}}{\epsilon^{l-1}}\right)^{q-2}r^{q(1-t)}= C_0'.
    \end{align}
    Now, assume $x\in B_r(z)$. Then, given $y\in B_r(x)$ with $|x-y|=r'<r$, we have
    \begin{equation}\label{D-L7-6.}
        \omega_y\leq \sup\limits_{b\in B_{r'}(y)}|D^2 v_z(b)|\leq \frac{l-1}{\epsilon^{l-1}}|b-z|^{l-2}\leq C(|b-x|+|z-x|)^{l-2}= C(r'+|z-x|)^{l-2}.
    \end{equation}
    Substituting \eqref{D-L7-6.} in \eqref{D-L7-3} and proceeding similar to \eqref{D-L7-5}, we obtain 
    \begin{align}\label{D-L7-6}
                \bigg|&\int_{B_r(x)}a(x,y)\frac{h_q(v_{z}(x)-v_{z}(y))}{\left|x-y|+\delta\right)^{N+tq}}dy \bigg|\nonumber\\
        &\ \ \ \leq C\Bigg|\int_{B_r(x)}a(x,y)\frac{\bigg(|\nabla v_z(x)\cdot(y-x)|+\omega_y|x-y|^2\bigg)^{p-2}\omega_y|x-y|^2}{|x-y|^{N+tq}}dy \Bigg|\nonumber\\
        &\ \ \ \leq C\int_0^r \left(1+\frac{(\tau+|z-x|)^{l-2}\tau}{|\nabla v_z(x)|} \right)^{q-2}|\nabla v_z(x)|^{q-2}\tau^{q-2}(\tau+|z-x|)^{l-2}\tau^2\tau^{-tq-1}d\tau\nonumber\\
        &\ \ \ =C \int_0^r \left(1+\frac{(\tau+|z-x|)^{l-2}\tau}{|\nabla v_z(x)|} \right)^{q-2}|\nabla v_z(x)|^{q-2}(\tau+|z-x|)^{l-2}\tau^{q-tq-1}d\tau.
    \end{align}
    Let us now estimate the integral on the RHS of \eqref{D-L7-6} by dividing it into the range $[0,|z-x|]$ and $[|z-x|,r]$. Note that we have $l\geq \frac{sp}{p-1}\geq \frac{tq}{q-1}$. Thus, $l(q-1)-tq\geq 0$. Since $|\nabla v_z(x)|= \frac{|z-x|^{l-1}}{\epsilon^{l-1}}$ and $q<2$, we get
    \begin{align}\label{D-L7-14}
        \int_0^{|z-x|}\bigg(1+&\frac{(\tau+|z-x|)^{l-2}\tau}{|\nabla v_z(x)|} \bigg)^{q-2}|\nabla v_z(x)|^{q-2}(\tau+|z-x|)^{l-2}\tau^{q-tq-1}d\tau\nonumber\\
        %&\leq \int_0^{|z-x|}\left(1+\frac{(\tau+|z-x|)^{l-2}\tau}{|\nabla v_z(x)|} \right)^{q-2}|\nabla v_z(x)|^{q-2}\tau^{t-tq-1}d\tau\nonumber\\
        &\leq C\int_0^{|z-x|}|z-x|^{(q-2)(l-1)}|z-x|^{l-2}\tau^{q-tq-1}d\tau\nonumber\\
        &\leq C|z-x|^{l(q-1)-tq}\leq Cr^{l(q-1)-tq}=C_0''.
    \end{align}
    We also deduce
    \begin{align}\label{D-L7-15}
        \int_{|z-x|}^r\left(1+\frac{(\tau+|z-x|)^{l-2}\tau}{|\nabla v_z(x)|} \right)^{q-2}&|\nabla v_z(x)|^{q-2}(\tau+|z-x|)^{l-2}\tau^{t-tq-1}d\tau\nonumber\\
        %& \leq \int_{|z-x|}^r\left(1+\frac{(\tau+|z-x|)^{l-2}\tau}{|\nabla v_z(x)|} \right)^{q-2}|\nabla v_z(x)|^{p-2}\tau^{t-tq-1}d\tau\nonumber\\
        &\leq C\int_{|z-x|}^r\left(\frac{r^{l-2}\tau}{|\nabla v_z(x)|} \right)^{q-2}|\nabla v_z(x)|^{q-2}\tau^{l-2}\tau^{t-tq-1}d\tau \nonumber\\
        &\leq Cr^{l(q-1)-tq}=C_0'''.
    \end{align}
    From \eqref{D-L7-5}, \eqref{D-L7-6}, \eqref{D-L7-14} and \eqref{D-L7-15}, we get a constant $C_0$ independent of $z,\delta$ such that
    \begin{equation}\label{D-L7-16}
        \bigg|\int_{B_r(x)}a(x,y)\frac{h_q(v_{z}(x)-v_{z}(y))}{\left|x-y|+\delta\right)^{N+tq}}dy \bigg|\leq C_0.
    \end{equation}
    Now, by Lemma 3.1(v) in \cite{BM2021}, for every $x\in \Omega_{r(\epsilon)}$, we have an element $a_x\in B_{r(\epsilon)}(x)$ such that $u_\epsilon(x)=v_{a_x}(x)$. From the definition of $u_\epsilon$, we deduce
    \begin{equation}\label{D-L7-7.}
        u_\epsilon(x)-u_\epsilon(y)=v_{a_x}(x)-u_\epsilon(y)=v_{a_x}(x)-\inf\limits_{z\in\mathbb{R}^N}v_z(x)\geq v_{a_x}(x)-v_{a_x}(y).
    \end{equation}
    %Using Lemma 2.4 in \cite{BM2021} \ the equation,,
    Thus, using \eqref{D-L7-16} and \eqref{D-L7-7.}, we get
    \begin{align}\label{D-L7-7}
        \int_{B_r(x)}a(x,y)\frac{\left(h_q(u_\epsilon(x)-u_\epsilon(y))\right)}{\left(|x-y|+\delta\right)^{N+tq}}dy\geq \int_{B_r(x)}a(x,y)\frac{h_q(v_{a_x}(x)-v_{a_x}(y))}{\left(|x-y|+\delta\right)^{N+tq}}dy\geq -C_0,
    \end{align}
    where $C_0>0$ is independent of $\delta,x$. Now, since $u_\epsilon$ is bounded, we get a constant $C_1$ independent of $x,\delta$ such that
     \begin{align}\label{D-L7-8}
        \left|\int_{\mathbb{R}^N \setminus B_r(x)}a(x,y)\frac{h_q(u_\epsilon(x)-u_\epsilon(y))}{\left(|x-y|+\delta\right)^{N+tq}}dy\right|\leq C_1.
    \end{align}
    From \eqref{D-L7-7} and \eqref{D-L7-8}, we obtain a constant $C_0+C_1$ independent of $\delta$ such that
    \begin{equation*}
        G_\delta(u_\epsilon)(x)\phi\geq -2(C_0+C_1)\phi \in L^1(K).
    \end{equation*}
    Therefore, using Fatou's lemma, %for the integrating $G_\delta(x)+(C_0+C_1)\|\phi\|_{L^{\infty}(\Omega_{r(\epsilon)})}$,
    we deduce
    \begin{equation}\label{D-L7-10'}
        \int_K \liminf\limits_{\delta \rightarrow 0}G_\delta(u_\epsilon)\phi dx\leq \liminf\limits_{\delta \rightarrow 0}\int_K G_\delta(u_\epsilon)\phi dx .
    \end{equation} 
    Due to the odd symmetry of $h_q(\nabla(v_{a_x}(x)).(x-y))$ in $B_r(x)$, we get
    \begin{equation}\label{D-L7-10}
        G_\delta(u_\epsilon)(x)=2\int_{\mathbb{R}^N} a(x,y)\frac{h_q(u_\epsilon(x)-u_\epsilon(y))-h_q(\nabla(v_{a_x}(x)).(x-y))\chi_{B_r(x)}(y)}{(|x-y|+\delta)^{N+tq}}dy.
    \end{equation}
    Now, let $y\in B_r(x)$. Using \eqref{D-L7-3}, \eqref{D-L7-7.} and \eqref{D-L7-10}, we have
    \begin{align}\label{D-L7-11}
        a(x,y)&\frac{h_q(u_\epsilon(x)-u_\epsilon(y))-h_q(\nabla(v_{a_x}(x)).(x-y))\chi_{B_r(x)}(y)}{(|x-y|+\delta)^{N+tq}}\nonumber\\
        & \geq a(x,y)\frac{h_q(v_{a_x}(x)-v_{a_x}(y))-h_q(\nabla(v_{a_x}(x)).(x-y))}{(|x-y|+\delta)^{N+tq}}dy\nonumber\\
        &\geq -C\frac{\bigg(|\nabla v_z(x)\cdot(y-x)|+\omega_y|x-y|^2\bigg)^{p-2}\omega_y|x-y|^2}{|x-y|^{N+tq}}dy=f_1(x,y).
    \end{align}
    We get $f_1\in L^1(B_r(x))$ like the previous arguments. Now, for $y\in\mathbb{R}^N \setminus B_r(x)$, since $u_\epsilon$ is bounded, we have
    \begin{equation}\label{D-L7-11'}
        \left|a(x,y)\frac{h_q(u_\epsilon(x)-u_\epsilon(y))}{(|x-y|+\delta)^{N+tq}}\right|\leq C\frac{1+|u_\epsilon(y)|^{q-1}}{|x-y|^{N+tq}}=f_2(x,y)\in L^1(\mathbb{R}^N \setminus B_r(x)).
    \end{equation}
    Using \eqref{D-L7-11} and \eqref{D-L7-11'}, we get
    $$a(x,y)\frac{h_q(u_\epsilon(x)-u_\epsilon(y))}{(|x-y|+\delta)^{N+tq}}\geq \chi_{B_r(x)}(y)f_1(x,y)-\chi_{\mathbb{R}^N \setminus B_r(x)}(y)f_2(x,y) \in L^1(\mathbb{R}^N).$$
    Thus, using Fatou's lemma, we arrive at
    \begin{align}\label{D-L7-12}
        \liminf\limits_{\delta \rightarrow 0}G_\delta(u_\epsilon)(x)\geq 2\operatorname{P.V.}\int_{\mathbb{R}^N}a(x,y)\frac{h_q(u_\epsilon(x)-u_\epsilon(y))}{|x-y|^{N+tq}}dy.
    \end{align}
    From \eqref{D-L7-10'} and \eqref{D-L7-12}, we obtain
    \begin{equation}\label{D-L7-12'}
         \liminf\limits_{\delta \rightarrow 0}\int_K G_\delta(u_\epsilon)\phi dx \geq \int_K\left(2\operatorname{P.V.}\int_{\mathbb{R}^N}a(x,y)\frac{h_q(u_\epsilon(x)-u_\epsilon(y))}{|x-y|^{N+tq}}dy\right)\phi dx.
    \end{equation}
    Since $\frac{1}{|x-y|^{N+tq}}\geq \frac{1}{(|x-y|+\delta)^{N+tq}}$ for all $\delta>0$, using the Lipschitz continuity and the boundedness of $u_\epsilon$ and similar to the proof of \eqref{D-L6-26} in Lemma \ref{D-L6}, we get
    \begin{align}\label{D-L7-13}
        \lim\limits_{\delta \rightarrow 0}\iint_{Q(K)}a(x,y)&\frac{h_q(u_\epsilon(x)-u_\epsilon(y))(\phi(x)-\phi(y))}{(|x-y|+\delta)^{N+tq}}dydx \nonumber\\
        &=\iint_{Q(K)}a(x,y)\frac{h_q(u_\epsilon(x)-u_\epsilon(y))(\phi(x)-\phi(y))}{|x-y|^{N+tq}}dydx.
    \end{align}
    From \eqref{D-L7-2}, \eqref{D-L7-12'}, and \eqref{D-L7-13}, we deduce
    \begin{align*}
       \iint_{Q(K)}a(x,y)&h_q(u_\epsilon(x)-u_\epsilon(y))(\phi(x)-\phi(y))K_{t,q}(x,y)dydx\nonumber\\
       &\geq \frac{1}{\Lambda_2}\iint_{Q(K)}a(x,y)\frac{h_q(u_\epsilon(x)-u_\epsilon(y))(\phi(x)-\phi(y))}{|x-y|^{N+tq}}dydx\nonumber\\
       &\geq \frac{1}{\Lambda_2}\int_K\left(2\operatorname{P.V.}\int_{\mathbb{R}^N}a(x,y)\frac{h_q(u_\epsilon(x)-u_\epsilon(y))}{|x-y|^{N+tq}}dy\right)\phi dx\nonumber\\
       &\geq \int_K\bigg(2\operatorname{P.V.}\int_{\mathbb{R}^N}a(x,y)h_q(u_\epsilon(x)-u_\epsilon(y))K_{t,q}(x,y)dy\bigg)\phi dx.
    \end{align*}
    Hence, the proof is complete.
\end{proof}

\noindent Next, we show that given a {viscosity} supersolution $u$ of problem \eqref{D} where $f$ satisfies certain conditions, each infimal convolution $u_\epsilon$ of $u$ is a weak supersolution to the problem \eqref{D-L3-prob1}. 

\begin{lemma}\label{D-L5}
    Let $u$ be a bounded viscosity supersolution to problem \eqref{D} and $f=f(x,t,\zeta,\eta)$ be non-increasing in $t$ and continuous in $\Omega\times \mathbb{R}^3$. Then, $u_\epsilon$ is a weak supersolution to the problem \eqref{D-L3-prob1}.
\end{lemma}
\begin{proof}
    By Lemma 3.1(iii) in \cite{BM2021}, $u_\epsilon$ is twice differentiable at almost every point in $\Omega_{r(\epsilon)}$. Choose an $x_0 \in \Omega_{r(\epsilon)}$, where $u_\epsilon$ is twice differentiable and $B_r(x_0)\subset \Omega_{r(\epsilon)}$. First consider the case $p>\frac{2}{2-s}$ or $\nabla u_\epsilon(x_0)\neq 0$. Then, by Lemma 3.1(iii) in \cite{BM2021}, choose $\mu>0$ such that the function
    \begin{equation*}
        \phi(x)=u_\epsilon(x_0)+\nabla u_\epsilon(x_0)\cdot (x-x_0)+\frac{1}{2}\left(D^2u_\epsilon(x_0)-\mu I\right)\cdot (x-x_0)^2
    \end{equation*}
    satisfies $\phi \in C_c^2(B_r(x_0))$, $\phi(x_0)=u_\epsilon(x_0)$ and $\phi \leq u_\epsilon$ in $B_r(x_0)$. Then for each $0<r'\leq r$, define
    \begin{equation}\label{D-L5-7}
        \psi_{r'}(x)=\begin{cases}
            \phi(x)&, \ \text{if } x\in B_{r'}(x_0),\\
            u_\epsilon(x)&, \ \text{if } x \in \mathbb{R}^N \setminus B_{r'}(x_0).
        \end{cases}
    \end{equation}
    Since $u_\epsilon$ is a viscosity supersolution to problem \eqref{D-L3-prob1} by Lemma \ref{D-L3}, we deduce
    \begin{align}\label{D-L5-1}
        L_a\psi_{r'}(x_0)\geq f_\epsilon(x_0,u_\epsilon(x_0),D_s^p \psi_{r'}(x_0), D_{a,t}^q \psi_{r'}(x_0)).
    \end{align}
    Since {$u_\epsilon(x)=u_\epsilon(x_0)+o(|x-x_0|)$ and $\psi_{r'}(x)=\psi_{r'}(x_0)+o(|x-x_0|)$} in $B_{r'}(x_0)$, we have
    \begin{align*}
        \Bigg|\int_{B_{r'}(x_0)}a(x,y)\frac{|\psi_{r'}(x_0)-\psi_{r'}(x)|^q-|u_\epsilon(x_0)-u_\epsilon(x)|^q}{|x_0-x|^{N+tq}}dx\Bigg| &\leq M\int_{B_{r'}(x_0)}\frac{C|x-x_0|^q}{|x_0-x|^{N+tq}}dx \\
        &\leq C(r')^{q-tq} \rightarrow 0 \text{ as } r'\rightarrow 0.
    \end{align*}
    Thus, we deduce
    \begin{align}\label{D-L5-2}
        D_{a,t}^q \psi_{r'}(x_0)&=\int_{\mathbb{R}^N}a(x,y)\frac{|\psi_{r'}(x_0)-\psi_{r'}(x)|^q}{|x_0-x|^{N+tq}}dx \nonumber\\
        &= D_{a,t}^q u_\epsilon(x_0) +\int_{B_{r'}(x_0)}a(x,y)\frac{|\psi_{r'}(x_0)-\psi_{r'}(x)|-|u_\epsilon(x_0)-u_\epsilon(x)|^q}{|x_0-x|^{N+tq}}dx\nonumber\\
        &\rightarrow D_{a,t}^q u_\epsilon(x_0) \text{ as } r' \rightarrow 0,
    \end{align}
    because {$u_\epsilon(x)=u_\epsilon(x_0)+o(|x-x_0|)$ and $\psi_{r'}(x)=\psi_{r'}(x_0)+o(|x-x_0|)$} in $B_{r'}(x_0)$. Similarly, we obtain
    \begin{equation}\label{D-L5-3}
         D_s^p\psi_{r'}(x_0)\rightarrow D_s^p u_\epsilon(x_0) \text{ as } r' \rightarrow 0.
    \end{equation}
    Since $f$ is continuous, from the definition of $f_\epsilon$ and using \eqref{D-L5-2}, \eqref{D-L5-3}, we get
    \begin{equation}\label{D-L5-4}
        \lim\limits_{r' \rightarrow 0}f_\epsilon(x_0,u_\epsilon(x_0),D_s^p \psi_{r'}(x_0), D_{a,t}^q \psi_{r'}(x_0)) \geq f_\epsilon(x_0,u_\epsilon(x_0),D_s^p u_\epsilon(x_0), D_{a,t}^q u_\epsilon(x_0)).
    \end{equation}
    Now, Lemma 4.3 in \cite{FZ2023} gives
    \begin{equation}\label{D-L5-5}
        L_a\psi_{r'}(x_0) \rightarrow L_a u_\epsilon(x_0) \text{ as } r' \rightarrow 0. 
    \end{equation}
    Therefore, taking $r' \rightarrow 0$ in \eqref{D-L5-1}, from \eqref{D-L5-4} and \eqref{D-L5-5}, we get
    \begin{equation}\label{D-L5-6}
        L_a u_\epsilon (x_0) \geq f_\epsilon(x_0,u_\epsilon(x_0),D_s^p u_\epsilon(x_0), D_{a,t}^q u_\epsilon(x_0)).
    \end{equation}
    Now assume that $p\leq \frac{2}{2-s}$ and $\nabla u_\epsilon(x_0)=0$. Then, the function
    \begin{equation*}
        \phi(x)=u(x_0)-\frac{|x_0-x|^l}{l\epsilon^{l-1}} 
    \end{equation*}
     satisfies $\phi \in C_l^2(B_r(x_0))$ and $\phi \leq u_\epsilon$ and $\phi(x_0) = u_\epsilon(x_0)$ for a sufficiently small $r>0$. Now, define $\psi_{r'}$ by \eqref{D-L5-7} for $0<r'\leq r$. Then, \eqref{D-L5-1} holds for $\psi_{r'}$ since $u_\epsilon$ is a viscosity supersolution to \eqref{D-L3-prob1}. Proceeding similarly to the previous arguments and using \eqref{D-L5-2}, \eqref{D-L5-3}, we deduce that \eqref{D-L5-4} holds for this $\phi_{r'}$. Now, by Lemma 4.4 in \cite{FZ2023}, we obtain that
     \begin{equation}\label{D-L5-8}
          L_a\psi_{r'}(x_0) \rightarrow L_a u_\epsilon(x_0) \text{ as } r' \rightarrow 0.
     \end{equation}
     Therefore, taking $r' \rightarrow 0$ in \eqref{D-L5-1}, using \eqref{D-L5-4} and \eqref{D-L5-8}, we get
     \begin{equation}\label{D-L5-9}
          L_a u_\epsilon (x_0) \geq f_\epsilon(x_0,u_\epsilon(x_0),D_s^p u_\epsilon(x_0), D_{a,t}^q u_\epsilon(x_0)).
     \end{equation}
     It follows from \eqref{D-L5-6} and \eqref{D-L5-9} that for a.e. $x \in \Omega_{r(\epsilon)}$,
     \begin{equation}\label{D-L5-9'}
         L_a u_\epsilon(x) \geq f_\epsilon(x,u_\epsilon(x),D_s^p u_\epsilon(x), D_{a,t}^q u_\epsilon(x)).
     \end{equation}
     Let $\psi\in C_c^\infty(\Omega_{r(\epsilon)})$ with $\operatorname{supp} \psi=K$ and $\psi\geq0$. Then, multiplying \eqref{D-L5-9'} by $\psi$, integrating over all $x\in \Omega_{r(\epsilon)}$ and applying Lemma 3.4, Lemma 3.5 in \cite{BM2021} and Lemma \ref{D-L6} and Lemma \ref{D-L7}, we obtain 
    $$H_a(u_\epsilon,\psi)\geq \int_K f(x,u_\epsilon, D_s^p u_\epsilon, D_{a,t}^q u_\epsilon)\psi dx.$$
    This proves the desired result.
\end{proof}
\noindent We conclude this section with the proof of Theorem \ref{D-T2}.
\medskip

\noindent {\it{\bf{Proof of Theorem} \ref{D-T2}.}}
    Throughout the proof, we use the following notations for simplicity. 
    \begin{align*}
        U_p(x,y)&=(u(x)-u(y))(K_{s,p}(x,y))^{\frac{1}{p}}, \hspace{1cm} U_{p,\epsilon}(x,y)=(u_\epsilon(x)-u_\epsilon(y))(K_{s,p}(x,y))^{\frac{1}{p}},\\
        V_p(x,y)&=h_p(u(x)-u(y))(K_{s,p}(x,y))^{\frac{p-1}{p}}, \hspace{0.3cm} V_{p,\epsilon}(x,y)=h_p(u_\epsilon(x)-u_\epsilon(y))(K_{s,p}(x,y))^{\frac{p-1}{p}},\\
        U_q(x,y)&=(u(x)-u(y))(T(x,y))^{\frac{1}{q}}, \hspace{1.5cm} U_{q,\epsilon}(x,y)=(u_\epsilon(x)-u_\epsilon(y))(T(x,y))^{\frac{1}{q}}, \\
        V_q(x,y)&=h_q(u(x)-u(y))(T(x,y))^{\frac{q-1}{q}}, \hspace{0.7cm} V_{q,\epsilon}(x,y)=h_q(u_\epsilon(x)-u_\epsilon(y))(T(x,y))^{\frac{q-1}{q}},
    \end{align*}
    where $T(x,y)=a(x,y)K_{t,q}(x,y)$. Let $v \in C_c^\infty(\Omega)$ with $v \geq 0$ and $\operatorname{supp}v =K$.  Choose $\epsilon>0$ such that $K \subset \Omega_{r(\epsilon)}$. Then, $v \in C_c^\infty(\Omega_{r(\epsilon)})$. Let $K_1$ be a compact subset of $\Omega_{r(\epsilon)}$ with $K\subset\subset K_1$. By Lemma \ref{D-L5}, we have
    \begin{equation}\label{D-T2-1}
        H_a(u_\epsilon,v)\geq \int_{K} f_\epsilon(x,u_\epsilon, D_s^p u_\epsilon, D_{a,t}^q u_\epsilon)v dx.
    \end{equation}
   Choose $\psi \in C_c^{\infty}(\Omega_{r(\epsilon)})$ with $\operatorname{supp}\psi=K_0$, $K_1 \subset\subset K_0$, $0\leq \psi \leq 1$ and $\psi \equiv 1$ in $K_1$. Also, choose an $\epsilon_0>0$. Then, by Lemma 3.1(ii) in \cite{BM2021}, for all $0<\epsilon<\epsilon_0$, we have $\operatorname{osc}u_\epsilon \leq \sup\limits_{\mathbb{R}^N}u -\inf\limits_{\mathbb{R}^N}u_{\epsilon_0}$. By Lemma \ref{D-L4}, for all $\epsilon< \epsilon_0$, we obtain a constant $C=C(p,q,K_0, \gamma_1, \gamma_2,\Lambda_1,\Lambda_2,h, u)$ such that
   \begin{align}\label{D-T2-2}
        \int_{K_1}& \int_{\mathbb{R}^N}\left(U_{p,\epsilon}(x,y)^p+U_{q,\epsilon}(x,y)^q\right)dydx \nonumber\\
        &\leq \int_{K_0} \int_{\mathbb{R}^N}U_{p,\epsilon}(x,y)^p\psi(x)^qdydx+\int_{K_0} \int_{\mathbb{R}^N} U_{q,\epsilon}(x,y)^q(x)\psi(x)^{q}dydx\nonumber\\
       &\leq C\bigg(\iint_{Q(K_0)} |\psi(x)-\psi(y)|^p K_{s,p}(x,y)dydx\nonumber\\
       &\quad \quad \ \ +\iint_{Q(K_0)}a(x,y) |\psi(x)-\psi(y)|^q K_{t,q}(x,y)dydx+1\bigg)\nonumber\\
       &=M_0.
   \end{align}
    Since $L^p(K_1 \times \mathbb{R}^N)$ and $L^q(K_1 \times \mathbb{R}^N)$ are reflexive Banach spaces for $1<p<\infty$, there exist $w_1\in L^p(K_1 \times \mathbb{R}^N)$, $w_2\in L^q(K_1 \times \mathbb{R}^N)$ such that the weak convergences $U_{p,\epsilon} \rightharpoonup w_1 $ in $L^p(K_1 \times \mathbb{R}^N)$ and $U_{q,\epsilon} \rightharpoonup w_2$ in $L^q(K_1 \times \mathbb{R}^N)$ hold. But by Lemma 3.1(ii) in \cite{BM2021}, we get that $U_{p,\epsilon} \rightarrow U_p $ and $V_{p,\epsilon} \rightarrow V_p $ a.e. in $K_1 \times \mathbb{R}^N$. Since the limit is unique, we have $w_1=U_p$ and $w_2=U_q$. Now, from \eqref{D-T2-2}, we also have
    \begin{align*}
        \|V_{p,\epsilon}\|_{L^{\frac{p}{p-1}}(K_1 \times \mathbb{R}^N)}=\|U_{p,\epsilon}\|_{L^p(K_1 \times \mathbb{R}^N)}^{p-1} \leq M_0^{\frac{p-1}{p}}\\
        \|V_{q,\epsilon}\|_{L^{\frac{q}{q-1}}(K_1 \times \mathbb{R}^N)}=\|U_{q,\epsilon}\|_{L^q(K_1 \times \mathbb{R}^N)}^{q-1} \leq M_0^{\frac{q-1}{q}}.
    \end{align*}
    Therefore, proceeding as above, we get that $V_{p,\epsilon} \rightharpoonup V_p$ in $L^\frac{p}{p-1}(K_1 \times \mathbb{R}^N)$ and $V_{q,\epsilon} \rightharpoonup V_q$ in $L^\frac{q}{q-1}(K \times \mathbb{R}^N)$. Due to the odd symmetry of $V_{p,\epsilon}, U_{p,\epsilon}, V_{q,\epsilon}, U_{q,\epsilon}$, similarly, we also obtain that $U_{l,\epsilon}\rightharpoonup U_l$ in $L^l((\mathbb{R}^N \setminus K_1) \times K_1)$ and $V_{l,\epsilon}\rightharpoonup V_l$ in $L^{\frac{l}{l-1}}((\mathbb{R}^N \setminus K_1) \times K_1)$ for $l\in \{p,q\}$. Since $Q(K_1)=(K_1 \times \mathbb{R}^N)\cup ((\mathbb{R}^N \setminus K_1) \times K_1)$, we get
    \begin{equation}\label{D-T2-3}
        H_a(u_\epsilon,v)\rightarrow H_a(u,v) \text{ as } \epsilon\rightarrow 0.
    \end{equation}
     %we proceed similar to the proof of the equation (3.47)) in Theorem 1.1 in \cite{BM2021}.
    Now, choose a function $\psi_0\in C_c^{\infty}(\Omega)$ with $\psi_0 \equiv1$ in $K$ and $\operatorname{supp}\psi_0=K_1$. Then, define $w=(u-u_\epsilon)\psi_0$. By Lemma 3.1(ii) in \cite{BM2021}, $w\geq0$. Thus, \eqref{D-T2-1} gives
    \begin{equation}\label{D-T2-5}
        H_a(u_\epsilon,w)-\int_{K_1} f_\epsilon(x,u_\epsilon, D_s^p u_\epsilon, D_{a,t}^q u_\epsilon)w dx \geq 0.
    \end{equation}
    Choose $m=\max\{\|u_\epsilon\|_{L^{\infty}(K_0)}:\epsilon>0\}$.  Then, since $\gamma_1,\gamma_2$ and $h$ are continuous, by \eqref{D-T2-f} and \eqref{D-f_ep}, we get a constant $C>0$ such that
    \begin{align}\label{D-T2-4}
        -\int_{K_1} f_\epsilon(x,u_\epsilon, D_s^p u_\epsilon, D_{a,t}^q u_\epsilon)w dx &\leq \int_{K_1} |f(x,u_\epsilon, D_s^p u_\epsilon, D_{a,t}^q u_\epsilon)|w dx\nonumber\\
        &\leq C\int_{K_1} \left( |D_s^p u_\epsilon|^{\frac{p-1}{p}}w+|D_{a,t}^q u_\epsilon|^{\frac{q-1}{q}}w+w\right)dx.
    \end{align}
    Since $u,u_\epsilon,\psi_0$ are bounded in $K_0$ and $K_0$ is compact, by the dominated convergence theorem,
    $$\lim\limits_{\epsilon \rightarrow 0}\int_{K_1}|w|^l dx = 0 \text{ for all } l>0.$$
    Thus, using H\"older's inequality in the first two integrals in the RHS of \eqref{D-T2-4}, and from \eqref{D-T2-2}, we deduce
    \begin{align}\label{D-T2-6}
        -\int_{K_1} f_\epsilon&(x,u_\epsilon, D_s^p u_\epsilon, D_{a,t}^q u_\epsilon)w dx\nonumber\\
        &\leq C(\Lambda_1)\left(\int_{K_1} \int_{\mathbb{R}^N}|u_\epsilon(x)-u_\epsilon(y)|^p K_{s,p}(x,y)dydx\right)^{\frac{p-1}{p}}\left(\int_{K_1}|w|^p dx\right)^{\frac{1}{p}}\nonumber\\
        &\ \ \ +C(\Lambda_2)\left(\int_{K_1} \int_{\mathbb{R}^N} a(x,y) |u_\epsilon(x)-u_\epsilon(y)|^q K_{t,q}(x,y)dydx \right)^{\frac{q-1}{q}}\left(\int_{K_1}|w|^q dx\right)^{\frac{1}{q}}\nonumber\\
        &\ \ \ +\int_{K_1}|w| dx \rightarrow0 \text{ as } \epsilon \rightarrow 0.
    \end{align}
    From the definition of $w$, we have
    $$w(x)-w(y)=\bigg(u(x)-u(y)-(u_\epsilon(x)-u_\epsilon(y))\bigg)\psi_0(x)+\bigg(u(y)-u_\epsilon(y)\bigg)\bigg(\psi_0(x)-\psi_0(y)\bigg).$$
    Therefore, \eqref{D-T2-5} and \eqref{D-T2-6} gives
    \begin{equation}\label{D-T2-7}
        \lim\limits_{\epsilon \rightarrow 0}H_a(u_\epsilon,w)=\lim\limits_{\epsilon \rightarrow 0}(I_{1,p}+I_{2,p}+I_{3,p}+I_{1,q}+I_{2,q}+I_{3,q})\geq 0, 
    \end{equation}
    where 
    \begin{align*}
        I_{1,p}&=\iint_{Q(K_1)}h_p(u(x)-u(y))(u(x)-u(y)-(u_\epsilon(x)-u_\epsilon(y)))\psi_0(x)K_{s,p}(x,y)dydx\\
        &=\int_{K_1}\int_{\mathbb{R}^N}h_p(u(x)-u(y))(u(x)-u(y)-(u_\epsilon(x)-u_\epsilon(y)))\psi_0(x)K_{s,p}(x,y)dydx,\\
        I_{2,p}&=\iint_{Q(K_1)}\bigg(h_p(u_\epsilon(x)-u_\epsilon(y))-h_p(u(x)-u(y))\bigg)\nonumber\\
        &\hspace{2cm}\times(u(x)-u(y)-(u_\epsilon(x)-u_\epsilon(y)))\psi_0(x)K_{s,p}(x,y)dydx\\
        &=\int_{K_1}\int_{\mathbb{R}^N}\bigg(h_p(u_\epsilon(x)-u_\epsilon(y))-h_p(u(x)-u(y))\bigg)\nonumber\\
        &\hspace{2cm}\times(u(x)-u(y)-(u_\epsilon(x)-u_\epsilon(y)))\psi_0(x)K_{s,p}(x,y)dydx,\\
        I_{3,p}&=\iint_{Q(K_1)}h_p(u_\epsilon(x)-u_\epsilon(y))(u(y)-u_\epsilon(y))(\psi_0(x)-\psi_0(y)K_{s,p}(x,y)dydx,\\
        I_{1,q}&=\iint_{Q(K_1)}a(x,y)h_q(u(x)-u(y))(u(x)-u(y)-(u_\epsilon(x)-u_\epsilon(y)))\psi_0(x)K_{t,q}(x,y)dydx\\
        &=\int_{K_1}\int_{\mathbb{R}^N}a(x,y)h_q(u(x)-u(y))(u(x)-u(y)-(u_\epsilon(x)-u_\epsilon(y)))\psi_0(x)K_{t,q}(x,y)dydx,\\
        I_{2,q}&=\int_{K_1}\int_{\mathbb{R}^N}a(x,y)\bigg(h_q(u_\epsilon(x)-u_\epsilon(y))-h_q(u(x)-u(y))\bigg)\nonumber\\
        &\hspace{2cm}\times(u(x)-u(y)-(u_\epsilon(x)-u_\epsilon(y)))\psi_0(x)K_{t,q}(x,y)dydx,\\
        I_{3,q}&=\iint_{Q(K_1)}a(x,y)h_q(u_\epsilon(x)-u_\epsilon(y))(u(y)-u_\epsilon(y))(\psi_0(x)-\psi_0(y))K_{t,q}(x,y)dydx.
    \end{align*}
    Let $\|u\|_{L^{\infty}(K_0)} \leq m$. By Remark 3.2 in \cite{BM2021}, $\|u_\epsilon\|_{L^{\infty}(K_0)} \leq m$. Using H\"older's inequality and \eqref{D-T2-2}, we have
    \begin{align}\label{D-T2-8}
        |I_{3,q}|&\leq \left(\iint_{Q(K_1)}a(x,y) |u_\epsilon(x)-u_\epsilon(y)|^q K_{t,q}(x,y)dydx\right)^{\frac{q-1}{q}}\nonumber\\
        &\ \ \ \times\left(\iint_{Q(K_1)}a(x,y)|u(y)-u_\epsilon(y)|^q|\psi_0(x)-\psi_0(y)|^qK_{t,q}(x,y)dydx\right)^{\frac{1}{q}}\nonumber\\
        &\leq 2^{\frac{q-1}{q}}\left(\int_{K_1} \int_{\mathbb{R}^N} a(x,y) |u_\epsilon(x)-u_\epsilon(y)|^q K_{t,q}(x,y)dydx\right)^{\frac{q-1}{q}}\nonumber\\
        &\ \ \ \times\left(\iint_{Q(K_1)}a(x,y)|u(y)-u_\epsilon(y)|^q|\psi_0(x)-\psi_0(y)|^qK_{t,q}(x,y)dydx\right)^{\frac{1}{q}}\nonumber\\
        &\leq C\left(\iint_{Q(K_1)}a(x,y)|u(y)-u_\epsilon(y)|^q|\psi_0(x)-\psi_0(y)|^qK_{t,q}(x,y)dydx\right)^{\frac{1}{q}}.
    \end{align}
    Since $u_\epsilon(y)-u(y) \rightarrow 0$ pointwise for every $y\in \mathbb{R}^N$ and
    $$\left|a(x,y)|u(y)-u_\epsilon(y)|^q|\psi_0(x)-\psi_0(y)|^qK_{t,q}(x,y)\right|\leq (2m)^qa(x,y)|\psi_0(x)-\psi_0(y)|^qK_{t,q}(x,y)$$
    which is integrable on $Q(K_1)$, using dominated convergence theorem in \eqref{D-T2-8}, we get
    $I_{3,q} \rightarrow 0 \text{ as } \epsilon \rightarrow 0.$
    Now, $V_q \in L^{\frac{q}{q-1}}(K_1\times \mathbb{R}^N)$ and $V_{q, \epsilon} \rightharpoonup V_q$ in $L^{\frac{q}{q-1}}(K_1\times \mathbb{R}^N)$. %Also, $u(x)-u(y)-u_\epsilon(x)+u_\epsilon(y)$ is bounded for $x,y\in \Omega$.
    Therefore, using H\"older's inequality in $I_{1,q}$ and proceeding similarly to \eqref{D-T2-8}, we deduce that $I_{1,q} \rightarrow 0 \text{ as } \epsilon \rightarrow 0.$
    Following a similar way, we also obtain
    $I_{1,p} \rightarrow 0$ and $I_{3,p} \rightarrow 0$ as $\epsilon \rightarrow 0.$ Subsequently, \eqref{D-T2-7} becomes
    \begin{equation}\label{D-T2-13}
        \lim\limits_{\epsilon\rightarrow 0}(I_{2,p}+I_{2,q})\geq 0.
    \end{equation}
    For $t_1,t_2 \in \mathbb{R}$, consider the algebraic inequalities
    \begin{align}
        |t_1-t_2|^2&\leq C\left(h_l(t_1)-h_l(t_2)\right)(t_1-t_2)\left(|t_1|+|t_2|\right)^{2-l}, \ 1<l<2,\label{D-T2-14}\\
        |t_1-t_2|^l&\leq 2^{1-l}\left(h_l(t_1)-h_l(t_2)\right)(t_1-t_2),\  l\geq 2.\label{D-T2-15}
    \end{align}
    When $1<q<2$, using \eqref{D-T2-14} with $l=q$, $t_1=u(x)-u(y)$ and $t_2=u_\epsilon(x)-u_\epsilon(y)$, we obtain
    \begin{align}\label{D-T2-16}
        0&\leq \frac{\left|u(x)-u(y)-(u_\epsilon(x)-u_\epsilon(y))\right|^2}{\left(|u(x)-u(y)|+|u_\epsilon(x)-u_\epsilon(y)|\right)^{2-q}} \nonumber\\
        &\leq C\left(h_q(u(x)-u(y))-h_q(u_\epsilon(x)-u_\epsilon(y))\right)\left(u(x)-u(y)-(u_\epsilon(x)-u_\epsilon(y)\right).
    \end{align}
    From \eqref{D-T2-16} and since $\psi_0(x)\geq 0$ for all $x \in \mathbb{R}^N$, we get that $I_{2,q}\leq 0$. For $q\geq 2$, using \eqref{D-T2-15} with $l=q$, $t_1=u(x)-u(y)$ and $t_2=u_\epsilon(x)-u_\epsilon(y)$, we have
    \begin{align}\label{D-T2-17}
        0&\leq 2^{l-1}\left|u(x)-u(y)-(u_\epsilon(x)-u_\epsilon(y))\right|^q \nonumber\\
        &\leq \left(h_q(u(x)-u(y))-h_q(u_\epsilon(x)-u_\epsilon(y))\right)\left(u(x)-u(y)-(u_\epsilon(x)-u_\epsilon(y)\right).
    \end{align}
    Hence, $I_{2,q}\leq 0$. Similarly, we obtain that $I_{2,p}\leq 0$ for all $1<p<\infty$. Thus, from \eqref{D-T2-13}, we get
    \begin{equation*}
        \lim\limits_{\epsilon\rightarrow 0} I_{2,p}=\lim\limits_{\epsilon\rightarrow 0} I_{2,q}=0.
    \end{equation*}
    Now, assume that $1<q<2$. Using H\"older's inequality, and applying \eqref{D-T2-2}, \eqref{D-T2-16}, we deduce
    \begin{align}\label{D-T2-19}
        0&\leq \int_{K}\int_{\mathbb{R}^N}a(x,y)\left|u(x)-u(y)-(u_\epsilon(x)-u_\epsilon(y))\right|^q K_{t,q}(x,y)dydx\nonumber\\
        &=\int_{K}\int_{\mathbb{R}^N}a(x,y)\frac{\left|u(x)-u(y)-(u_\epsilon(x)-u_\epsilon(y))\right|^q}{\left(|u(x)-u(y)|+|u_\epsilon(x)-u_\epsilon(y)|\right)^{\frac{(2-q)q}{2}}}\nonumber\\
        & \hspace{2cm}\times\left(|u(x)-u(y)|+|u_\epsilon(x)-u_\epsilon(y)|\right)^{\frac{(2-q)q}{2}}K_{t,q}(x,y)(x,y)dydx\nonumber\\
        &\leq \left(\int_{K}\int_{\mathbb{R}^N}a(x,y)\frac{\left|u(x)-u(y)-(u_\epsilon(x)-u_\epsilon(y))\right|^2}{\left(|u(x)-u(y)|+|u_\epsilon(x)-u_\epsilon(y)|\right)^{(2-q)}}K_{t,q}(x,y)(x,y)dydx \right)^{\frac{q}{2}}\nonumber\\
        & \ \ \ \times \left(\int_{K}\int_{\mathbb{R}^N}a(x,y)\left(|u(x)-u(y)|+|u_\epsilon(x)-u_\epsilon(y)|\right)^q K_{t,q}(x,y)(x,y)dydx\right)^{\frac{2-q}{2}}\nonumber\\
        &\leq C\bigg(\int_{K}\int_{\mathbb{R}^N}a(x,y)\left(h_q(u(x)-u(y))-h_q(u_\epsilon(x)-u_\epsilon(y))\right)\nonumber\\
        & \ \ \ \hspace{2cm}\times\left(u(x)-u(y)-(u_\epsilon(x)-u_\epsilon(y)\right)K_{t,q}(x,y)(x,y)dydx \bigg)^{\frac{q}{2}}\nonumber\\ 
        &\leq C\bigg(\int_{K_1}\int_{\mathbb{R}^N}a(x,y)\left(h_q(u(x)-u(y))-h_q(u_\epsilon(x)-u_\epsilon(y))\right)\nonumber\\
        & \ \ \ \hspace{2cm}\times\left(u(x)-u(y)-(u_\epsilon(x)-u_\epsilon(y)\right)\psi_0(x)K_{t,q}(x,y)(x,y)dydx \bigg)^{\frac{q}{2}}\nonumber\\
        &=C(-I_{2,q})^{\frac{q}{2}} \rightarrow 0 \text{ as }\epsilon\rightarrow 0.
    \end{align}
    For the case $q\geq 2$, using \eqref{D-T2-17}, we have
    \begin{align}\label{D-T2-20}
        0&\leq \int_{K}\int_{\mathbb{R}^N}a(x,y)\left|u(x)-u(y)-(u_\epsilon(x)-u_\epsilon(y))\right|^q K_{t,q}(x,y)dydx\nonumber\\
        &\leq C \int_{K}\int_{\mathbb{R}^N}a(x,y)\left(h_q(u(x)-u(y))-h_q(u_\epsilon(x)-u_\epsilon(y))\right)\nonumber\\
        &\hspace{1.9cm} \times \left(u(x)-u(y)-(u_\epsilon(x)-u_\epsilon(y)\right)K_{t,q}(x,y)dydx \nonumber\\
        &\leq C \int_{K_1}\int_{\mathbb{R}^N}a(x,y)\left(h_q(u(x)-u(y))-h_q(u_\epsilon(x)-u_\epsilon(y))\right)\nonumber\\
        &\hspace{1.9cm} \times \left(u(x)-u(y)-(u_\epsilon(x)-u_\epsilon(y)\right)\psi_0(x)K_{t,q}(x,y)dydx \nonumber\\
        &= -C I_{2,q} \rightarrow 0 \text{ as }\epsilon\rightarrow 0.
    \end{align}
    Similarly, we get that for all $1<p<\infty$
    \begin{equation}\label{D-T2-21}
        \lim\limits_{\epsilon\rightarrow 0}\int_{K}\int_{\mathbb{R}^N}\left|u(x)-u(y)-(u_\epsilon(x)-u_\epsilon(y))\right|^p K_{s,p}(x,y)dydx=0.
    \end{equation}
    From \eqref{D-T2-19}, \eqref{D-T2-20} and \eqref{D-T2-21}, we have
    \begin{align*}
        \lim\limits_{\epsilon\rightarrow 0}&\bigg(\int_{K}\int_{\mathbb{R}^N}\left|u(x)-u(y)-(u_\epsilon(x)-u_\epsilon(y))\right|^p K_{s,p}(x,y)dydx \nonumber\\
        &\int_{K}\int_{\mathbb{R}^N}a(x,y)\left|u(x)-u(y)-(u_\epsilon(x)-u_\epsilon(y))\right|^q K_{t,q}(x,y)dydx \bigg)=0.
    \end{align*}
    Consequently, by lemma \ref{D-L2}, we have
    \begin{equation}\label{D-T2-22}
        \lim\limits_{\epsilon\rightarrow 0}\int_K f_\epsilon(x,u_\epsilon, D_s^p u_\epsilon, D_{a,t}^q u_\epsilon)v dx=\int_K f(x,u, D_s^p u, D_{a,t}^q u)v dx.
    \end{equation}
    Hence, applying the limit as $\epsilon\rightarrow 0$ in \eqref{D-T2-1} and using \eqref{D-T2-3}, and \eqref{D-T2-22}, we get
    \begin{equation*}
        H_a(u,v)\geq \int_K f(x,u, D_s^p u, D_{a,t}^q u)v dx.
    \end{equation*}
    This proves the desired result.
\hfill\qedsymbol{}

\section{Local Boundedness of weak solutions}\label{D-S6}
In this section, we make use of the double phase De Giorgi classes to prove the local boundedness of weak solutions to \eqref{D-1}, where the function $f=f(x,t)$ is a locally bounded function which satisfies \eqref{D-f-H}. For simplicity, throughout this section, we denote $B_r(0)=B_r$. Also, given $0<\alpha<1\leq l$, we use the notation $l_\alpha^*=\frac{Nl}{N-\alpha l}$. We first prove the following result which is important to prove Lemma \ref{D-L8}.
\begin{lemma}\label{D-L10}
Let $0<r<R$ and $P:[r,R]\rightarrow [0,\infty)$ be a function satisfying the following conditions:
\begin{itemize}
    \item $P(.)$ is bounded.
    \item There exist constants $0<\beta<1,\ D_i\geq 0, \ i\in\{0,1,2,3,4\}$, and $\alpha_i>0, \ i\in\{1,2,3,4\}$ such that for every $r\leq r'<r''\leq R$,
    \begin{equation}\label{D-L10-1}
        P(r')\leq \beta P(r'')+D_0+\sum\limits_{i=1}^4 \frac{D_i}{(r''-r')^{\alpha_i}}.
    \end{equation}
    Then, we get a constant $C=C(\alpha_1,\alpha_2,\alpha_3,\alpha_4)\geq 1$ such that
    \begin{equation}\label{D-L10-2}
        P(r)\leq C\left(D_0+\sum\limits_{i=1}^4 \frac{D_i}{(R-r)^{\alpha_i}}\right).
    \end{equation}
\end{itemize}
\end{lemma}
\begin{proof}
    Let $r_0=r$ and define $$r_j=r_{j-1}+\rho^{j-1}(1-\rho )(R-r),$$
    where the constant $0<\rho<1$ will be chosen later. Then, $(r_j)$ is an increasing sequence with $r_j\rightarrow R$ as $j\rightarrow \infty$. Let us take $\alpha_0=0$.\\
    \textbf{Claim: }For any $n\in\{0,1,2,...\}$. we have
    \begin{align}\label{D-L10-3}
        P(r)\leq \beta^n P(r_n)+\sum\limits_{i=0}^4\frac{D_i\Bigg(1-\bigg(\frac{\beta}{\rho^{\alpha_i}}\bigg)^n\Bigg)}{(1-\rho)^{\alpha_i}(R-r)^{\alpha_i}\Bigg(1-\bigg(\frac{\beta}{\rho^{\alpha_i}}\bigg)\Bigg)}.
    \end{align}
    We proceed by induction. For $n=0$, the inequality \eqref{D-L10-3} trivially holds. Now, assume that the claim holds for $n=k-1\geq 0$. Thus, we have
    \begin{equation}\label{D-L10-4}
        P(r)\leq \beta^{k-1} P(r_{k-1})+\sum\limits_{i=0}^4\frac{D_i\Bigg(1-\bigg(\frac{\beta}{\rho^{\alpha_i}}\bigg)^{k-1}\Bigg)}{(1-\rho)^{\alpha_i}(R-r)^{\alpha_i}\Bigg(1-\bigg(\frac{\beta}{\rho^{\alpha_i}}\bigg)\Bigg)}.
    \end{equation}
    Now, since $r_{k-1}\leq r_{k}$ and $r_k-r_{k-1}=\rho^{k-1}(1-\rho)(R-r)$, by \eqref{D-L10-1}, we have
    \begin{equation}\label{D-L10-5}
        P(r_{k-1})\leq \beta P(r_{k})+D_0+\sum\limits_{i=1}^4 \frac{D_i}{(\rho^{k-1}(1-\rho)(R-r))^{\alpha_i}}.
    \end{equation}
    Substituting \eqref{D-L10-5} in \eqref{D-L10-4}, we obtain
    \begin{align*}
        P(r)&\leq \beta^{k-1}\left(\beta P(r_{k})+\sum\limits_{i=0}^4 \frac{D_i}{(\rho^{(k-1){\alpha_i}}(1-\rho)^{\alpha_i}(R-r)^{\alpha_i}}\right)\\
        &\hspace{1.5cm} +\sum\limits_{i=0}^4\frac{D_i\Bigg(1-\bigg(\frac{\beta}{\rho^{\alpha_i}}\bigg)^{k-1}\Bigg)}{(1-\rho)^{\alpha_i}(R-r)^{\alpha_i}\Bigg(1-\bigg(\frac{\beta}{\rho^{\alpha_i}}\bigg)\Bigg)}\\
        &=\beta^k P(r_{k}) + \sum\limits_{i=0}^4 \frac{D_i}{(1-\rho)^{\alpha_i}(R-r)^{\alpha_i}}\left[\left(\frac{\beta}{\rho^{\alpha_i}}\right)^{k-1}+\frac{\Bigg(1-\bigg(\frac{\beta}{\rho^{\alpha_i}}\bigg)^{k-1}\Bigg)}{\Bigg(1-\bigg(\frac{\beta}{\rho^{\alpha_i}}\bigg)\Bigg)}\right]\\
        &=\beta^k P(r_{k}) + \sum\limits_{i=0}^4 \frac{D_i}{(1-\rho)^{\alpha_i}(R-r)^{\alpha_i}}\frac{\Bigg(1-\bigg(\frac{\beta}{\rho^{\alpha_i}}\bigg)^{k}\Bigg)}{\Bigg(1-\bigg(\frac{\beta}{\rho^{\alpha_i}}\bigg)\Bigg)}.
    \end{align*}
    Thus, the claim holds for all $n \in \{0,1,2,3,...\}$. Choose $\rho$ such that $\frac{\beta}{\rho^{\alpha_i}}<1$ for each $i\in {1,2,3,4}$. Then, by applying $n \rightarrow \infty$ in \eqref{D-L10-3} and taking
    $$C\geq \max\left\{\frac{1}{(1-\rho)^{\alpha_i}\Bigg(1-\bigg(\frac{\beta}{\rho^{\alpha_i}}\bigg)\Bigg)}: i\in \{1,2,3,4\}\right\},$$ 
    we arrive at the inequality \eqref{D-L10-2}.
\end{proof}

The next lemma establishes that each weak solution to problem \eqref{D-1} where $f$ satisfies \eqref{D-f-H} lies in a double phase De Giorgi class.

\begin{lemma}\label{D-L8}
    Let $u\in W^{t,q}(\Omega)$ be a weak subsolution to {problem} \eqref{D-1} where $f$ satisfies \eqref{D-f-H} for some constants $c_1,c_2\geq0$ and $0\in(1,q_t^*)$. Then, there exists $\tilde{R}>0, \ \tilde{k}\in[-\infty,1], \ \epsilon\in (0,\frac{tq}{N}]$ and a constant $\tilde{C}=C(N,p,q,s,t,\Lambda_1,\Lambda_2,c_2)\geq1$ such that the following holds:
    \begin{enumerate}[(i)]
        \item For $c_2=0$, $u\in \operatorname{DG}_a^+(\Omega, c_1^\frac{1}{q-1}, \tilde{C}, -\infty, \frac{tq}{N}, \frac{tq}{q-1}, \infty)$.
        \item For $c_2>0$ and $1<l\leq q$, $u\in \operatorname{DG}_a^+(\Omega, c_1^\frac{1}{q-1}, \tilde{C}, 1, \frac{tq}{N}, \frac{tq}{q-1}, 1)$. 
        \item For $c_2>0$ and $q<l<q_t^*$, $u\in \operatorname{DG}_a^+(\Omega, c_1^\frac{1}{q-1}, \tilde{C}, 0, 1-\frac{l}{q_\alpha^*}, \frac{tq}{q-1}, \tilde{R})$, where $0<\alpha\leq t$ is fixed and $\tilde{R}=\tilde{R}(N,q,t,\Lambda_2,c_2,\|u\|_{L^{q_\alpha^*}(\Omega)})$ for a fixed $0<\alpha\leq t$.
    \end{enumerate}
    Moreover, the same conclusion holds for $u$ and $\operatorname{DG}_a^-$ if $u$ is a weak supersolution to problem \eqref{D-1}.
\end{lemma}
\begin{proof}
    Without loss of generality, assume $x_0=0$. Choose an $0<R<\operatorname{dist}(x_0,\partial \Omega)$ and let $0<r\leq r_1<r_2\leq R$ and define $r_3=\frac{r_1+r_2}{2}$. %Since $\Omega$ is bounded, $a(x,y)>0$ and $a(.,.)$ is continuous, there exist an $m>0$ such that $a(x,y)>m$ for $x,y \in \overline \Omega$.
    Choose $\phi\in C_c^\infty(\mathbb{R}^N)$ with $0\leq \phi\leq 1$, $\operatorname{supp}\phi \subset B_{r_3}, \ \phi \equiv 1$ in $B_{r_1}$ and $\|\nabla \phi\|_{L^{\infty}(\mathbb{R}^N)} \leq \frac{4}{r_2-r_1}$. Given, $k\in\mathbb{R}$, define $\psi=u-k$ and $v=\phi^q \psi_+$. Then, since $u$ is a weak subsolution to problem \eqref{D-1}, we get
    \begin{equation}\label{D-L8-1}
        \iint_{Q(B_{r_2})}(A(x,y)+B(x,y))dydx\leq \int_{B_{r_2}}f(x,u)v(x)dx,
    \end{equation}
    where
    \begin{align*}
        A(x,y)&=h_p(u(x)-u(y))(v(x)-v(y))K_{s,p}(x,y),\\
        B(x,y)&=a(x,y)h_q(u(x)-u(y))(v(x)-v(y))K_{t,q}(x,y).
    \end{align*}
    Next, consider the LHS of \eqref{D-L8-1}. 
    \begin{align}\label{D-L8-6'}
        \iint_{Q(B_{r_2})}(A(x,y)+B(x,y))dydx&=I_1+I_2,
    \end{align}
    where $I_1$ and $I_2$ are given by
    \begin{align*}
        I_1&=\int_{B_{r_2}}\int_{B_{r_2}}(A(x,y)+B(x,y))dydx,\\
        I_2&=\iint_{Q(B_{r_2})\setminus (B_{r_2}\times B_{r_2})}(A(x,y)+B(x,y))dydx\nonumber\\
        &=2\int_{B_{r_2}}\int_{\mathbb{R}^N \setminus B_{r_2}}(A(x,y)+B(x,y))dydx.
    \end{align*}
    We first examine $I_1$. Let $x,y \in B_{r_2}$. We have three cases.\\
    {\bf Case 1:  $x,y \in \{u>k\}$.} Then, $u(x)=\psi_+(x)+k$ and $u(y)=\psi_+(y)+k$. Without loss of generality, assume that $u(x)\geq u(y)$. First let $u(x)>u(y)$. Consider the subcase $\phi(x)<\phi(y)$. From Lemma 4.3 in\cite{C2017}, for $l>1$ and $\epsilon>0$, we have
    \begin{equation}\label{D-L8-6}
        a^l-b^l\leq \epsilon a^l+\left(\frac{l-1}{\epsilon}\right)^{l-1}(a-b)^l, \ a\geq b\geq0.
    \end{equation}
    Denote $\phi_1=\phi^{\frac{q}{p}}$. Applying \eqref{D-L8-6} for $a=\phi_1(y),\ b=\phi_1(x),\ l=p$, and $\epsilon=\frac{\psi_+(x)-\psi_+(y)}{2\psi_+(x)}$, we deduce
    \begin{align}\label{D-L8-7}
        \phi_1(y)^p-\phi_1(x)^p\leq \frac{\psi_+(x)-\psi_+(y)}{2\psi_+(x)}\phi_1(y)^p+\left(\frac{2(p-1)\psi_+(x)}{\psi_+(x)-\psi_+(y)}\right)^{p-1}(\phi_1(y)-\phi_1(x))^p.
    \end{align}
    Observe that we also have
    \begin{equation}\label{D-L8-7'}
        v(x)-v(y)=(\psi_+(x)-\psi_+(y))\phi(y)^q-\psi_+(x)(\phi(y)^q-\phi(x)^q).
    \end{equation}
    Expanding $A(x,y)$ and substituting \eqref{D-L8-7} and \eqref{D-L8-7'}, we obtain
    \begin{align}\label{D-L8-8}
        A(x,y)&=h_p(\psi_+(x)-\psi_+(y))(v(x)-v(y))K_{s,p}(x,y)\nonumber\\
        &=\bigg((\psi_+(x)-\psi_+(y))^p\phi_1(y)^p-(\psi_+(x)-\psi_+(y))^{p-1}\psi_+(x)(\phi_1(y)^p-\phi_1(x)^p)\bigg)K_{s,p}(x,y)\nonumber\\
        &\geq \bigg((\psi_+(x)-\psi_+(y))^p\phi_1(y)^p-(\psi_+(x)-\psi_+(y))^{p-1}\psi_+(x)(\phi_1(y)^p-\phi_1(x)^p)\bigg)K_{s,p}(x,y)\nonumber\\
        %&\geq \bigg( (\psi_+(x)-\psi_+(y))^p\phi_1(y)^p-\frac{1}{2}(\psi_+(x)-\psi_+(y))^p\phi_1(y)^p-C(p)\psi_+(x)^p(\phi_1(y)-\phi_1(x))^p\bigg)K_{s,p}(x,y) \nonumber\\
        &\geq \bigg(\frac{1}{2}(\psi_+(x)-\psi_+(y))^p\phi(y)^q-C(p)\psi_+(x)^p(\phi_1(y)-\phi_1(x))^p\bigg)K_{s,p}(x,y).
    \end{align}
    Similarly, applying \eqref{D-L8-6} with $a=\phi(y),\ b=\phi(x),\ l=q$, and $\epsilon=\frac{\psi_+(x)-\psi_+(y)}{2\psi_+(x)}$, we have
    \begin{equation}\label{D-L8-9'}
        \phi(y)^q-\phi(x)^q \leq \frac{\psi_+(x)-\psi_+(y)}{2\psi_+(x)}\phi(y)^q+\left(\frac{2(q-1)\psi_+(x)}{\psi_+(x)-\psi_+(y)}\right)^{q-1}(\phi(y)-\phi(x))^q.
    \end{equation}
    Therefore, using \eqref{D-L8-7'} and \eqref{D-L8-9'} in $B(x,y)$, we get
    \begin{align}\label{D-L8-9}
        B(x,y)&\geq a(x,y)\bigg((\psi_+(x)-\psi_+(y))^q\phi(y)^q-(\psi_+(x)\nonumber\\
        &\hspace{2cm}-\psi_+(y)^{q-1}\psi_+(x)(\phi(y)^q-\phi(x)^q))\bigg)K_{t,q}(x,y)\nonumber\\
        &\geq a(x,y)\bigg(\frac{1}{2}(\psi_+(x)-\psi_+(y))^q\phi(y)^q-C(q)\psi_+(x)^q (\phi(y)-\phi(x))^q\bigg)K_{t,q}(x,y).
    \end{align}
    Next we consider the subcase $\phi(x)\geq \phi(y)$. Then, clearly for $l\in\{p,q\}$,
    \begin{align}\label{D-L8-10}
        h_l(u(x)-u(y))(v(x)-v(y))&\geq  h_l(\psi_+(x)-\psi_+(y))(\phi(x)^q\psi_+(x)-\phi(x)^q\psi_+(y))\nonumber\\
        &=(\psi_+(x)-\psi_+(y))^l \phi(x)^q.
    \end{align}
    As a result of \eqref{D-L8-10}, we deduce
    \begin{align}\label{D-L8-11}
        A(x,y)+B(x,y) \geq \bigg((&\psi_+(x)-\psi_+(y))^pK_{s,p}(x,y)\nonumber\\
        &+a(x,y)(\psi_+(x)-\psi_+(y))^q K_{t,q}(x,y) \bigg)\phi(x)^q.
    \end{align}
    Thus, from \eqref{D-L8-8}, \eqref{D-L8-9} and \eqref{D-L8-11}, we obtain 
    \begin{align}\label{D-L8-12}
        A(x,y)+B(x,y)\geq \frac{1}{2}&\bigg(|\psi_+(x)-\psi_+(y)|^pK_{s,p}(x,y)\nonumber\\
        &+a(x,y)|\psi_+(x)-\psi_+(y)|^q \bigg)\max\{\phi(x),\phi(y)\}^qK_{t,q}(x,y) \nonumber\\
        &-C\bigg(\max\{\psi_+(x),\psi_+(y)\}^p|\phi_1(y)-\phi_1(x)|^pK_{s,p}(x,y) \nonumber\\
        & + \max\{\psi_+(x),\psi_+(y)\}^q a(x,y)|\phi(y)-\phi(x)|^qK_{t,q}(x,y)\bigg)
    \end{align}
    Clearly, \eqref{D-L8-12} holds when $u(x)=u(y)$.\\
    
   \noindent {\bf Case 2: $x \in \{u>k\}, \ y \in \{u\leq k\}$. }
    In this case, $v(y)=0$ and $u(x)=\psi_+(x)+k,\ u(y)=-\psi_-(y)+k$. For $l'\geq 1$ and $a,b\geq 0$, we have
    \begin{equation}\label{D-L8-13}
         (a+b)^{l'}\geq a^{l'}+b^{l'},
    \end{equation}
    Let $l \geq 2$. using \eqref{D-L8-13} with {$a=\psi_+(x), \ b=\psi_-(y), \ l'=l-1$}, we get
    \begin{align}\label{D-L8-14}
        h_l(u(x)-u(y))(v(x)-v(y))&=(\psi_+(x)+\psi_-(y))^{l-1}\phi(x)^q \psi_+(x)\nonumber\\
        &\geq \phi(x)^q(\psi_+(x)^{l-1}+\psi_-(y)^{l-1})\psi_+(x)\nonumber\\
        &=\phi(x)^q(\psi_+(x)^{l}+\psi_-(y)^{l-1}\psi_+(x))\nonumber\\
        &=\phi(x)^q \bigg(|\psi_+(x)-\psi_+(y)|^l+\psi_-(y)^{l-1}\psi_+(x)\bigg).
    \end{align}
    Now suppose $1<l<2$. Using Jensen's inequality, we obtain that for $a,b>0, \ 0<l'<1$, 
    \begin{equation}\label{D-L8-13'}
        (a+b)^{l'}\geq 2^{l'-1}\left(a^{l'}+b^{l'}\right).
    \end{equation}
    Consequently, utilizing \eqref{D-L8-13'} for {$a=\psi_+(x), \ b=\psi_-(y), \ l'=l-1$}, we deduce
    \begin{equation}\label{D-L8-14'}
        h_l(u(x)-u(y))(v(x)-v(y))\geq C\phi(x)^q \bigg(|\psi_+(x)-\psi_+(y)|^l+\psi_-(y)^{l-1}\psi_+(x)\bigg).
    \end{equation}
    Therefore, from \eqref{D-L8-14} and \eqref{D-L8-14'}, we arrive at
    \begin{align}\label{D-L8-16}
        A(x,y)+&B(x,y) \nonumber\\&\geq  C'\bigg(|\psi_+(x)-\psi_+(y)|^pK_{s,p}(x,y)+a(x,y)|\psi_+(x)-\psi_+(y)|^q K_{t,q}(x,y)\nonumber\\
        &+\psi_+(x)\psi_-(y)^{p-1}K_{s,p}(x,y)+ a(x,y)\psi_+(x)\psi_-(y)^{q-1}K_{t,q}(x,y)\bigg)\phi(x)^q, 
    \end{align}
    where $C'=\min\{1, 2^{p-1},2^{q-1}\}$.\\
    
    \noindent{\bf Case 3: $x,y \in \{u\leq k\}$.}
    In this case, since $v(x)=v(y)=0$, we have
    \begin{equation}\label{D-L8-17}
        A(x,y)+B(x,y)=0.
    \end{equation}
    Therefore, from \eqref{D-L8-12}, \eqref{D-L8-16} and \eqref{D-L8-17}, we obtain that whenever $x,y \in B_{r_2}$,
    \begin{align}\label{D-L8-18}
        A&(x,y)+B(x,y)\nonumber\\
        &\geq \frac{1}{C}\Bigg(|\psi_+(x)-\psi_+(y)|^pK_{s,p}(x,y)+a(x,y)|\psi_+(x)-\psi_+(y)|^q K_{t,q}(x,y)\nonumber\\
        &\ \ \ +\psi_+(x)\psi_-(y)^{p-1}K_{s,p}(x,y)+ a(x,y)\psi_+(x)\psi_-(y)^{q-1}K_{t,q}(x,y)\Bigg)\max\{\phi(x),\phi(y)\}^q \nonumber\\
        &\ \ \ -C\Bigg(\max\{\psi_+(x),\psi_+(y)\}^p|\phi_1(y)-\phi_1(x)|^pK_{s,p}(x,y)\nonumber\\
        & \ \ \ + a(x,y)\max\{\psi_+(x),\psi_+(y)\}^q |\phi(y)-\phi(x)|^qK_{t,q}(x,y)\bigg)\Bigg).
    \end{align}
    Therefore, substituting \eqref{D-L8-18} in $I_1$ and using the fact that $\phi\equiv 1$ on $B_{r_1}$, we get 
    \begin{align}\label{D-L8-19}
        I_1\geq& \frac{1}{C}\Bigg([\psi_+]_{W_{s,p}(B_{r_1})}^p+[\psi_+]_{W_{t,q,a}(B_{r_1})}^q+\int_{B_{r_1}}\int_{B_{r_2}}\psi_+(x)\psi_-(y)^{p-1}K_{s,p}(x,y)dydx\nonumber\\
        &\ \ \ +\int_{B_{r_1}}\int_{B_{r_2}}a(x,y)\psi_+(x)\psi_-(y)^{q-1}K_{t,q}(x,y)dydx\Bigg)\nonumber\\
        &\ \ \ -C\Bigg(\int_{B_{r_2}}\int_{B_{r_2}}\bigg(\max\{\psi_+(x),\psi_+(y)\}^p|\phi_1(y)-\phi_1(x)|^pK_{s,p}(x,y) \nonumber\\
        &\ \ \ +a(x,y)\max\{\psi_+(x),\psi_+(y)\}^q |\phi(y)-\phi(x)|^q K_{t,q}(x,y)\bigg)dydx\Bigg)\nonumber\\
        &\geq \frac{1}{C}\Bigg([\psi_+]_{W_{s,p}(B_{r_1})}^p+[\psi_+]_{W_{t,q,a}(B_{r_1})}^q\nonumber\\
        &\ \ \ +\int_{B_{r_1}}\int_{B_{r_2}}\psi_+(x)\bigg(\frac{\psi_-(y)^{p-1}}{|x-y|^{N+sp}}+\frac{a(x,y)\psi_-(y)^{q-1}}{|x-y|^{N+tq}}\bigg)dydx\Bigg)\nonumber\\
        &\ \ \ -C\int_{B_{r_2}}\int_{B_{r_2}}\bigg(\max\{\psi_+(x),\psi_+(y)\}^p|\phi_1(y)-\phi_1(x)|^pK_{s,p}(x,y)\nonumber\\
        & \ \ \ + a(x,y)\max\{\psi_+(x),\psi_+(y)\}^q |\phi(y)-\phi(x)|^qK_{t,q}(x,y)\bigg)dydx.
    \end{align}
    Observe that by the definition of $\phi$,
    $$|\nabla \phi_1|=\frac{q}{p}|\phi^{\frac{q}{p}-1}|\cdot|\nabla \phi|\leq \frac{4}{r_2-r_1}.$$
    Also, for $x,y \in B_{r_2}$, $x-y \in B_{2r_2}$.
    Therefore, we have
    \begin{align}\label{D-L8-20}
        \int_{B_{r_2}}\int_{B_{r_2}}\max\{\psi_+(x),\psi_+(y)\}^p&|\phi_1(y)-\phi_1(x)|^pK_{s,p}(x,y)dydx\nonumber\\
        &\leq \int_{B_{r_2}}\int_{B_{r_2}}(\psi_+(x)^p+\psi_+(y)^p)|\phi_1(y)-\phi_1(x)|^pK_{s,p}(x,y)dydx \nonumber\\
        &\leq2\int_{B_{r_2}}\psi_+(x)^p\int_{B_{r_2}}\|\nabla \phi_1\|_{L^{\infty}(\mathbb{R}^N)}^{p}|x-y|^pK_{s,p}(x,y)dydx\nonumber\\
        &\leq\frac{C\Lambda_1}{(r_2-r_1)^p} \int_{B_{r_2}}\psi_+(x)^p\int_{B_{2r_2}}\frac{1}{|x-y|^{N+sp-p}}dydx\nonumber\\
        &\leq C \frac{R^{(1-s)p}}{(r_2-r_1)^p}\|\psi_+\|_{L^p(B(r_2))}.
    \end{align}
    Now, since $a(.,.)$ is a bounded function, similarly, we deduce
    \begin{equation}\label{D-L8-21}
        \int_{B_{r_2}}\int_{B_{r_2}}a(x,y)\max\{\psi_+(x),\psi_+(y)\}^q(\phi(y)-\phi(x))^qK_{t,q}(x,y)dydx \leq C\frac{R^{(1-t)q}}{(r_2-r_1)^q}\|\psi_+\|_{L^q(B(r_2))}. 
    \end{equation}
    Therefore, combining \eqref{D-L8-19}, \eqref{D-L8-20} and \eqref{D-L8-21}, we get a constant $C(\Lambda_1,\Lambda_2,N,p,q,s,t)$ such that
    \begin{align}\label{D-L8-22}
        I_1&\geq \frac{1}{C}\Bigg([\psi_+]_{W_{s,p}(B_{r_1})}^p+[\psi_+]_{W_{t,q,a}(B_{r_1})}^q\nonumber\\
        &\hspace{1cm}+\int_{B_{r_1}}\int_{B_{r_2}}\psi_+(x)\bigg(\frac{\psi_-(y)^{p-1}}{|x-y|^{N+sp}}+\frac{a(x,y)\psi_-(y)^{q-1}}{|x-y|^{N+tq}})\bigg)dydx\Bigg)\nonumber\\
        &\ \ \ -C\bigg(\frac{R^{(1-s)p}}{(r_2-r_1)^p}\|\psi_+\|_{L^p(B(r_2))}+\frac{R^{(1-t)q}}{(r_2-r_1)^q}\|\psi_+\|_{L^q(B(r_2))}\bigg).
    \end{align}
    Next, we estimate $I_2$. For any $\tilde{R}>0$, we have
    \begin{align}\label{D-L8-23}
        \int_{B_{r_2}}v(x)\int_{\mathbb{R}^N \setminus B_{r_2}}&a(x,y)h_q(u(x)-u(y))K_{t,q}(x,y)dydx \nonumber\\
        &= \int_{B_{r_3}}\phi(x)\psi_+(x)\int_{\mathbb{R}^N \setminus B_{r_2}}a(x,y)h_q(u(x)-u(y))K_{t,q}(x,y)dydx\nonumber\\
        &\geq J_2-J_1,
    \end{align}
    where 
    \begin{align*}
        J_1&=\int_{B_{r_3}\cap \operatorname{supp}\psi_+}\psi_+(x)\int_{\{u(x)<u(y)\}\setminus B_{r_2}}a(x,y)(u(y)-u(x))^{q-1}K_{t,q}(x,y)dydx,\\
        J_2&=\int_{B_{r_1}\cap \operatorname{supp}\psi_+}\psi_+(x)\int_{(B_{2\tilde{R}} \cap \{u(x)\geq u(y)\})\setminus B_{r_2}}a(x,y)(u(x)-u(y))^{q-1}K_{t,q}(x,y)dydx.
    \end{align*}
    We first evaluate $J_1$. For $|x|< r_3$ and $|y|> {r_2}$, we get
    $$|x-y|\geq |y|-|x| \geq |y|-\frac{r_1+r_2}{2}\geq |y|-\frac{r_1+r_2}{2}\frac{|y|}{r_2}=\frac{(r_2-r_1)|y|}{2r_2}.$$
    Also, when $x,y \in \operatorname{supp}\psi_+$ and $u(y)>u(x)$, we get $0<u(y)-u(x)\leq u(y)-k=\psi_+(y)$. Therefore, we have
    \begin{align}\label{D-L8-24}
        J_1&\leq \Lambda_2 \int_{B_{r_3}\cap \operatorname{supp}\psi_+}\psi_+(x)\int_{\{u(x)<u(y)\}\setminus B_{r_2}}a(x,y)\frac{(u(y)-u(x))^{q-1}}{|x-y|^{N+tq}}dydx\nonumber\\
        &\leq \Lambda_2\left(\frac{2r_2}{r_2-r_1}\right)^{N+tq}\int_{B_{r_3}\cap \operatorname{supp}\psi_+}\psi_+(x)\int_{\{u(x)<u(y)\}\setminus B_{r_2}}a(x,y)\frac{(u(y)-u(x))^{q-1}}{|y|^{N+tq}}dydx\nonumber\\
        &\leq \Lambda_2\left(\frac{2R}{r_2-r_1}\right)^{N+tq}\int_{B_{r_3}\cap \operatorname{supp}\psi_+}\psi_+(x)\int_{\mathbb{R}^N\setminus B_{r_2}}a(x,y)\frac{\psi_+(y)^{q-1}}{|y|^{N+tq}}dydx\nonumber\\
        &\leq \Lambda_2\left(\frac{2R}{r_2-r_1}\right)^{N+tq}r_2^{-tq}\int_{B_{r_3}}\psi_+(x)\left(\operatorname{Tail}_{a,t,q}(\psi_+,x_0,r_2)\right)^{q-1}dx\nonumber\\
        &\leq \Lambda_2\left(\frac{2R}{r_2-r_1}\right)^{N+tq}r_2^{-tq}\|\psi_+\|_{L^1(B_{r_3})}\left(\operatorname{Tail}_{a,t,q}(\psi_+,x_0,r_2)\right)^{q-1}\nonumber\\
        &\leq \Lambda_2\left(\frac{2R}{r_2-r_1}\right)^{N+tq}r^{-tq}\|\psi_+\|_{L^1(B_{R})}\left(\operatorname{Tail}_{a,t,q}(\psi_+,x_0,r)\right)^{q-1}
    \end{align}
    Next, consider $J_2$. For $x\in \operatorname{supp}\psi_+$, we have $\operatorname{supp}\psi_- \subset \{y:u(x)\geq u(y)\}$. Hence, we get
    \begin{align}\label{D-L8-25}
        J_2&\geq \int_{B_{r_1}\cap \operatorname{supp}\psi_+}\psi_+(x)\int_{(B_{2\tilde{R}} \cap \operatorname{supp}\psi_-)\setminus B_{r_2}}a(x,y)(u(x)-u(y))^{q-1}K_{t,q}(x,y)dydx\nonumber\\
        &= \int_{B_{r_1}\cap \operatorname{supp}\psi_+}\psi_+(x)\int_{(B_{2\tilde{R}} \cap \operatorname{supp}\psi_-)\setminus B_{r_2}}a(x,y)(\psi_+(x)+\psi_-(y))^{q-1}K_{t,q}(x,y)dydx\nonumber\\
        &\geq \frac{1}{\Lambda_2}\int_{B_{r_1}\cap \operatorname{supp}\psi_+}\psi_+(x)\int_{(B_{2\tilde{R}} \cap \operatorname{supp}\psi_-)\setminus B_{r_2}}a(x,y)\frac{\psi_-(y)^{q-1}}{|x-y|^{N+tq}}dydx\nonumber\\
        &= \frac{1}{\Lambda_2}\int_{B_{r_1}}\int_{B_{2\tilde{R}}\setminus B_{r_2}}a(x,y)\frac{\psi_+(x)\psi_-(y)^{q-1}}{|x-y|^{N+tq}}dydx.
    \end{align}
    From \eqref{D-L8-23}, \eqref{D-L8-24} and \eqref{D-L8-25}, we have
    \begin{align}\label{D-L8-26}
        \int_{B_{r_2}}v(x)\int_{\mathbb{R}^N \setminus B_{r_2}}&a(x,y)h_q(u(x)-u(y))K_{t,q}(x,y)dydx \nonumber\\
        &\geq \frac{1}{\Lambda_2}\int_{B_{r_1}}\int_{B_{2\tilde{R}}\setminus B_{r_2}}a(x,y)\frac{\psi_+(x)\psi_-(y)^{q-1}}{|x-y|^{N+tq}}dydx\nonumber\\
        &\ \ \ -\Lambda_2\left(\frac{2R}{r_2-r_1}\right)^{N+tq}r^{-tq}\|\psi_+\|_{L^1(B_{r_3})}\left(\operatorname{Tail}_{a,t,q}(\psi_+,x_0,r)\right)^{q-1}.
    \end{align}
    Proceeding similarly, we also deduce
    \begin{align}\label{D-L8-27}
        \int_{B_{r_2}}v(x)\int_{\mathbb{R}^N \setminus B_{r_2}}&h_p(u(x)-u(y))K_{s,p}(x,y)dydx\nonumber\\
        &\geq \frac{1}{\Lambda_1}\int_{B_{r_1}}\int_{B_{2\tilde{R}}\setminus B_{r_2}}\frac{\psi_+(x)\psi_-(y)^{p-1}}{|x-y|^{N+sp}}dydx\nonumber\\
        &\ \ \ -\Lambda_1\left(\frac{2R}{r_2-r_1}\right)^{N+sp}r^{-sp}\|\psi_+\|_{L^1(B_{r_3})}\left(\operatorname{Tail}_{s,p}(\psi_+,x_0,r)\right)^{p-1}.
    \end{align}
    Hence, substituting \eqref{D-L8-26} and \eqref{D-L8-27} in $I_2$, we obtain
    \begin{align}\label{D-L8-28}
        I_2&\geq \frac{1}{C}\bigg(\int_{B_{r_1}}\int_{B_{2\tilde{R}}\setminus B_{r_2}}\frac{\psi_+(x)\psi_-(y)^{p-1}}{|x-y|^{N+sp}}dydx+ \int_{B_{r_1}}\int_{B_{2\tilde{R}}\setminus B_{r_2}}a(x,y)\frac{\psi_+(x)\psi_-(y)^{q-1}}{|x-y|^{N+tq}}dydx\bigg) \nonumber\\
        &\ \ \ -C\Bigg(\bigg(\frac{2R}{r_2-r_1}\bigg)^{N+sp}r^{-sp}\|\psi_+\|_{L^1(B_{r_3})}\left(\operatorname{Tail}_{s,p}(\psi_+,x_0,r)\right)^{p-1} \nonumber\\
        &\ \ \ +\bigg(\frac{2R}{r_2-r_1}\bigg)^{N+tq}r^{-tq}\|\psi_+\|_{L^1(B_{r_3})}\left(\operatorname{Tail}_{a,t,q}(\psi_+,x_0,r)\right)^{q-1}\Bigg).
    \end{align}
    Therefore, from the equations \eqref{D-L8-1}, \eqref{D-L8-6'}, \eqref{D-L8-22} and \eqref{D-L8-28}, we obtain a constant $C_1=C_1(N,p,q,\Lambda_1,\Lambda_2,s,t)>0$ such that 
    \begin{align}\label{D-L8-28'}
        [\psi_+&]_{W_{s,p}(B_{r_1})}^p+[\psi_+]_{W_{t,q,a}(B_{r_1})}^q+\int_{B_{r_1}}\int_{B_{2\tilde{R}}}\psi_+(x)\bigg(\frac{\psi_-(y)^{p-1}}{|x-y|^{N+sp}}+\frac{a(x,y)\psi_-(y)^{q-1}}{|x-y|^{N+tq}}\bigg)dydx \nonumber\\
        & \leq C_1\Bigg(\frac{R^{(1-s)}}{(r_2-r_1)^p}\|\psi_+\|_{L^p(B(r_2))}+\frac{R^{(1-t)}}{(r_2-r_1)^q}\|\psi_+\|_{L^q(B(r_2))}\nonumber\\
        &\ \ \ +\bigg(\frac{2R}{r_2-r_1}\bigg)^{N+sp}r^{-sp}\|\psi_+\|_{L^1(B_{\frac{r_1+r_2}{2}})}\left(\operatorname{Tail}_{s,p}(\psi_+,x_0,r)\right)^{p-1}\nonumber\\
        &\ \ \ +\bigg(\frac{2R}{r_2-r_1}\bigg)^{N+tq}r^{-tq}\|\psi_+\|_{L^1(B_{\frac{r_1+r_2}{2}})}\left(\operatorname{Tail}_{a,t,q}(\psi_+,x_0,r)\right)^{q-1}\Bigg)\nonumber\\
        &\ \ \ +\int_{B_{r_2}}f(x,u)v(x)dx.
    \end{align}
    Now, from \eqref{D-f-H}, we have
    \begin{align}\label{D-L8-2}
        \int_{B_{r_2}}f(x,u)v(x)dx&\leq \int_{\operatorname{supp}\psi_+ \cap B_{r_2}}(c_1+c_2|u|^{l-1})v(x)dx,
    \end{align}
    where $l\in (1,q_t^*)$. \\
    \textbf{Case 1: } $c_2=0$.\\
    Note that $(r_2-r_1)R^{t-1}\leq R^t$. Now, using \eqref{D-L8-2} and substituting the Young's inequality \eqref{D-L4-4} with $a=R^tc_1, \ b=\frac{\psi_+(x)}{(r_2-r_1)R^{t-1}}, \ \epsilon=1, \ l'=\frac{q}{q-1}$, we get
    \begin{align}\label{D-L8-2'}
         \int_{\operatorname{supp}\psi_+ \cap B_{r_2}}c_1v(x)dx&\leq \int_{\operatorname{supp}\psi_+ \cap B_{r_2}}(r_2-r_1)R^{t-1}c_1\frac{\psi_+(x)}{(r_2-r_1)R^{t-1}}dx\nonumber\\
        &\leq C_1'\Bigg(\int_{\operatorname{supp}\psi_+ \cap B_{r_2}}((r_2-r_1)R^{t-1}c_1)^{\frac{q}{q-1}}dx\nonumber\\
        &\hspace{1cm}+\int_{\operatorname{supp}\psi_+ \cap B_{r_2}}\frac{\psi_+(x)^q}{(r_2-r_1)^q R^{(t-1)q}}dx\Bigg)\nonumber\\
        &\leq C'\left(R^{\frac{tq}{q-1}}c_1^{\frac{q}{q-1}}|\operatorname{supp}\psi_+ \cap B_{R}|+\frac{R^{(1-t)q}}{(r_2-r_1)^q}\|\psi_+\|_{L^q(B_R)}^q\right).
    \end{align}
    Therefore, substituting \eqref{D-L8-2}, \eqref{D-L8-2'} in \eqref{D-L8-28'} and letting $r_1=r, \ r_2=R$, we get the desired result for $\theta=c_1^\frac{1}{q-1}, \ \tilde{C}=\max\{C_1,C_1',1\}, \ \tilde{k}=-\infty, \ \tilde{\epsilon}=\frac{tq}{N}, \ \rho=\frac{tq}{q-1}, \ \tilde{R}=\infty$.\\
    \textbf{Case 2: } $c_2>0,\ 1<l\leq q$.\\
    First, assume that $1<l\leq q$. Let $k\geq 1, R\leq 1$. Whenever $x\in \operatorname{supp} \psi_+$, we have $u(x)\geq 1$. Therefore,
    \begin{equation}\label{D-L8-3}
        |u(x)|^{l-1}\leq u(x)^{q-1}=(\psi_+(x)+k)^{q-1}\leq 2^{q-1}(\psi_+(x)^{q-1}+|k|^{q-1}).
    \end{equation}
    Observe that using the Young's inequality \eqref{D-L4-4} with $a=\psi_+(x), \ b=|k|^{q-1}, \epsilon=1, \ l'=q$, we have
    \begin{equation}\label{D-L8-2''}
        |k|^{q-1}\psi_+(x)\leq C\left(\psi_+(x))^q+|k|^q\right).
    \end{equation}
    Since $|\phi|\leq 1$, we deduce from \eqref{D-L8-2}, \eqref{D-L8-2'}, \eqref{D-L8-3},  and \eqref{D-L8-2''} that
    \begin{align}\label{D-L8-4}
        \int_{B_{r_2}}f(x,u)v(x)dx&\leq C\bigg(\int_{\operatorname{supp}\psi+ \cap B_{r_2}}c_1\phi(x)^q\psi_+(x) dx\nonumber\\
        &\hspace{1cm}+ \int_{B_{r_2}} c_2\left(\psi_+(x)^{q}+|k|^{q-1}\psi_+(x)\right)\phi(x)^q dx \bigg)\nonumber\\
        &\leq C_2\bigg(R^{\frac{tq}{q-1}}c_1^{\frac{q}{q-1}}|\operatorname{supp}\psi_+ \cap B_{R}|+\frac{R^{(1-t)q}}{(r_2-r_1)^q}\|\psi_+\|_{L^q(B_R)}^q \nonumber\\
        &\hspace{1cm}+ \int_{\operatorname{supp}\psi_+ \cap B_{r_2}}\psi_+(x)^q dx+\int_{\operatorname{supp}\psi_+ \cap B_{r_2}}\psi_+(x)^q dx \nonumber\\
        & \hspace{1cm}  + |k|^q\cdot |\operatorname{supp}\psi_+ \cap B_{r_2}| \bigg),
    \end{align}
    where the constant $C_2>0$ depends on $c_2$ also. Now, since $R\leq 1$ and $(r_2-r_1)^q<R^q \leq R^{(1-t)q}$, \eqref{D-L8-4} becomes
    \begin{align*}
        \int_{B_{r_2}}&f(x,u)v(x)dx \nonumber\\
        &\leq C_2\bigg( R^{\frac{tq}{q-1}}c_1^{\frac{q}{q-1}}|\operatorname{supp}\psi_+ \cap B_R|+\frac{R^{(1-t)q}}{(r_2-r_1)^q}\|\psi_+\|_{L^q(B_R)}^q +\frac{|k|^q}{R^{tq}}|\operatorname{supp}\psi_+ \cap B_R|\bigg).
    \end{align*}
    Thus, we get the required result for $\theta=c_1^\frac{1}{q-1}, \ \tilde{C}=\max\{C_2,1\}, \ \tilde{k}=1, \ \tilde{\epsilon}=\frac{tq}{N}, \ \rho=\frac{tq}{q-1}, \ \tilde{R}=1$.\\
    \textbf{Case 3: } $c_2>0, \ q<l<q_t^*$.\\
    Let $k\geq 0$. Using \eqref{D-L8-2'}, we get
    \begin{equation}\label{D-L8-29}
        \int_{\operatorname{supp}\psi_+ \cap B_{r_2}}c_1v(x)dx \leq C_1'\left(R^{\frac{tq}{q-1}}c_1^{\frac{q}{q-1}}|\operatorname{supp}\psi_+ \cap B_{R}|+\frac{R^{(1-t)q}}{(r_2-r_1)^q}\|\psi_+\|_{L^q(B_R)}^q\right).
    \end{equation}
   Using the fact that $|\phi|\leq 1$ and utilizing Young's inequality \eqref{D-L4-4} with $a=\phi(x)\psi_+(x), \ b=|u(x)|, \ \epsilon=1, \ l'=l$, we obtain
    \begin{align}\label{D-L8-30}
         \int_{\operatorname{supp}\psi_+ \cap B_{r_2}}c_2|u(x)|^{l-1}v(x)dx&\leq \int_{\operatorname{supp}\psi_+ \cap B_{r_2}}c_2|u(x)|^{l-1}(|u(x)|+\phi(x)^q\psi_+(x))dx\nonumber\\
         &\leq\int_{\operatorname{supp}\psi_+ \cap B_{r_2}}c_2 (|u(x)|^l+\phi(x)|u(x)|^{l-1}\psi_+(x))dx \nonumber\\
         &\leq C\int_{\operatorname{supp}\psi_+ \cap B_{r_2}} \bigg(|u(x)|^l+\left(\phi(x)^l\psi_+(x)^l+|u(x)|^{l}\right)\bigg)dx\nonumber\\
         &\leq C\int_{\operatorname{supp}\psi_+ \cap B_{r_2}}(|u(x)|^l+\phi(x)^l\psi_+(x)^l)dx.
    \end{align}
    Since $k\geq 0$, whenever $x\in \operatorname{supp}\psi_+$, we have
    \begin{align}\label{D-L8-31}
        |u(x)|^l&=|\phi(x)\psi_+(x)+k\phi(x)+(1-\phi(x))u(x)|^l\nonumber\\
        &\leq 3^l\left(\phi(x)^l\psi_+(x)^l+k^l+(1-\phi(x))^l\psi_+(x)^l\right).
    \end{align}
    Therefore, employing \eqref{D-L8-31} in \eqref{D-L8-30}, we get     
    \begin{align}\label{D-L8-32}
        \int_{\operatorname{supp}\psi_+ \cap B_{r_2}}&c_2|u(x)|^{l-1}v(x)dx\nonumber\\
        &\leq  C\int_{\operatorname{supp}\psi_+ \cap B_{r_2}} \left(\phi(x)^l\psi_+(x)^l+|k|^l+(1-\phi(x))^l\psi_+(x)^l\right)dx.
    \end{align}
    For each $r'\geq 0$, define 
    \begin{align*}
        P(r')&=[\psi_+]_{W_{s,p}(B_{r'})}^p+[\psi_+]_{W_{t,q,a}(B_{r'})}^q\nonumber\\
        &\ \ \ +\int_{B_{r'}}\int_{B_{2\tilde{R}}}\psi_+(x)\bigg(\frac{\psi_-(y)^{p-1}}{|x-y|^{N+sp}}+\frac{a(x,y)\psi_-(y)^{q-1}}{|x-y|^{N+tq}}\bigg)dydx+\int_{B_{r'}}|u(x)|^l dx.
    \end{align*}
    Combining \eqref{D-L8-28'}, \eqref{D-L8-29} and \eqref{D-L8-32}, and adding $\int_{B_{r_1}}|u(x)|^l dx$ on both sides of the obtained equation, we find a constant $C_3=C_3(N,p,q,s,t,\Lambda_1,\Lambda_2,c_2)\geq 1$ such that
    \begin{align}\label{D-L8-33}
        P(r_1)& \leq C_3\Bigg(\frac{R^{(1-s)}}{(r_2-r_1)^p}\|\psi_+\|_{L^p(B(R))}^p+\frac{R^{(1-t)}}{(r_2-r_1)^q}\|\psi_+\|_{L^q(B(R))}^q\nonumber\\
        &\ \ \ +\bigg(\frac{R}{r_2-r_1}\bigg)^{N+sp}r^{-sp}\|\psi_+\|_{L^1(B_{R})}\left(\operatorname{Tail}_{s,p}(\psi_+,x_0,r)\right)^{p-1}\nonumber\\
        &\ \ \ +\bigg(\frac{R}{r_2-r_1}\bigg)^{N+tq}r^{-tq}\|\psi_+\|_{L^1(B_{R})}\left(\operatorname{Tail}_{a,t,q}(\psi_+,x_0,r)\right)^{q-1}+R^{\frac{tq}{q-1}}c_1^{\frac{q}{q-1}}|\operatorname{supp}\psi_+ \cap B_R|\nonumber\\
        &\ \ \ +\int_{B_{r_2}} (\phi(x)^l\psi_+(x)^l+|k|^l+(1-\phi(x))^l\psi_+(x)^l)dx\Bigg).
    \end{align}
    Now, since $1-\phi \equiv 0$ in $B_{r_1}$, we have
    \begin{align}\label{D-L8-34}
        \int_{B_{r_2}}(1-\phi(x))^l\psi_+(x)^l dx\leq \int_{B_{r_2}\setminus B_{r_1}}\psi_+(x)^l dx \leq P(r_2)-P(r_1).
    \end{align}
    We define 
    \begin{equation*}
        \alpha=
        \begin{cases}
            t,&\text{ if }tq<N, \\
            \max\left\{2t-1, \frac{(2l-q)N}{2lq}\right\}, &\text{ if } tq=N.
        \end{cases}
    \end{equation*}
    Then, we have $\alpha\leq t$ and $\alpha q<N$. Now, let
    \begin{equation}\label{D-L8-EPSILON}
        \epsilon=1-\frac{l}{q_\alpha^*}.
    \end{equation}
    For the case $tq=N$, observe that
    $$q_\alpha^*=\frac{Nq}{N-\alpha q}\geq \frac{Nq}{N-\frac{(2l-q)Nq}{2lq}}= 2l>l.$$
    Therefore, $0<\tilde{\epsilon}<1=\frac{tq}{N}$. Next, suppose $tq<N$. In this case 
    $$0<1-\frac{l}{q_\alpha^*}< 1-\frac{q}{q_t^*}=\frac{tq}{N}.$$
    Also, note that
    $$\frac{l-q}{q_\alpha^*}=1-\tilde{\epsilon}-\frac{N-\alpha q}{N}=\frac{\alpha q}{N}-\tilde{\epsilon}.$$
    Note that since $\alpha\leq t$, by Theorem \ref{fSS-CE}, we have $u\in W^{\alpha, q}(\mathbb{R}^N)\subset L^{q_\alpha^*}(\mathbb{R}^N)$. Now, using H\"older's inequality, we have
    \begin{align}\label{D-L8-35}
         \int_{B_{r_2}}|k|^l dx&= \int_{B_{r_2} \cap \operatorname{supp}\psi_+}|k|^l dx \nonumber\\
         &= |k|^q \int_{B_{r_2} \cap \operatorname{supp}\psi_+} |k|^{l-q} dx \nonumber\\
         &\leq |k|^q \left(\int_{B_{r_2}\cap \operatorname{supp}\psi_+} |k|^{(l-q)\frac{q_\alpha^*}{l-q}}\right)^{\frac{l-q}{q_\alpha^*}}\left(\int_{B_{r_2}\cap \operatorname{supp}\psi_+} 1dx\right)^{1-\frac{(l-q)}{q_\alpha^*}}.
    \end{align}
   Note that whenever $x\in \operatorname{supp}\psi_+$, we have $u(x)\geq k$. Therefore, \eqref{D-L8-35} becomes
   \begin{align}\label{D-L8-36}
       \int_{B_{r_2}}|k|^l dx&\leq |k|^q \left(\int_{B_{r_2}\cap \operatorname{supp}\psi_+} |u|^{q_\alpha^*}\right)^{\frac{l-q}{q_\alpha^*}}|B_{r_2}\cap \operatorname{supp}\psi_+|^{1-\frac{(l-q)}{q_\alpha^*}}\nonumber\\
       &\leq |k|^q \|u\|_{L^{q_\alpha^*}(\Omega)}^{l-q}|B_{r_2}\cap \operatorname{supp}\psi_+|^{1-\frac{tq}{N}+\tilde{\epsilon}+\frac{(t-\alpha)q}{N}}\nonumber\\
       &\leq |k|^q \|u\|_{L^{q_\alpha^*}(\Omega)}^{l-q} |B_R|^{\frac{(t-\alpha)q}{N}}|B_{R}\cap \operatorname{supp}\psi_+|^{1-\frac{tq}{N}+\tilde{\epsilon}}\nonumber\\
       &=|k|^q \|u\|_{L^{q_\alpha^*}(\Omega)}^{l-q} |B_1|^{\frac{(t-\alpha)q}{N}}R^{(t-\alpha)q}|B_{R}\cap \operatorname{supp}\psi_+|^{1-\frac{tq}{N}+\tilde{\epsilon}}.
   \end{align}
    Now, choose $$R\leq \min\Bigg\{1, \left(\frac{1}{\|u\|_{L^{q_\alpha^*}(\Omega)}^{l-q} |B_1|^{\frac{(t-\alpha)q}{N}}} \right)^{\frac{1}{N\tilde{\epsilon}}}\Bigg\}.$$
    Then, we have
    \begin{equation}\label{D-L8-37}
        \frac{1}{R^{N\tilde{\epsilon}}}\geq \|u\|_{L^{q_\alpha^*}(\Omega)}^{l-q} |B_1|^{\frac{(t-\alpha)t}{N}} \geq \|u\|_{L^{q_\alpha^*}(\Omega)}^{l-q} |B_1|^{\frac{(t-\alpha)q}{N}} R^{(t-\alpha)q}.
    \end{equation}
    Thus, combining \eqref{D-L8-36} and \eqref{D-L8-37}, we deduce
    \begin{equation}\label{D-L8-38}
        \int_{B_{r_2}}|k|^l dx\leq \frac{|k|^q}{R^{N\tilde{\epsilon}}} |B_{R}\cap \operatorname{supp}\psi_+|^{1-\frac{tq}{N}+\epsilon}.
    \end{equation}
    Since we have $$1=\tilde{\epsilon}+\frac{l-q}{q_\alpha^*}+\frac{q}{q_\alpha^*},$$ 
    using H\"older's inequality and $u\geq \psi_+$ in $\operatorname{supp}\psi_+$, we get
    \begin{align}\label{D-L8-39}
        \int_{B_{r_2}}\phi(x)^l\psi_+(x)^l dx &=\int_{B_{r_2}}1\times\phi(x)^{l-q}\psi_+(x)^{l-q} \times \phi(x)^{q}\psi_+(x)^{q} dx\nonumber\\
        &\leq \bigg(\int_{B_{r_2}}1dx\bigg)^{\tilde{\epsilon}} \bigg(\int_{B_{r_2}}\psi_+(x)^{q_\alpha^*}dx\bigg)^{\frac{l-q}{q_\alpha^*}}\bigg(\int_{B_{r_2}}(\phi(x)\psi_+(x))^{q_\alpha^*}dx\bigg)^{\frac{q}{q_\alpha^*}}\nonumber\\
        &\leq |B_R|^{\tilde{\epsilon}}\|u\|_{L^{q_\alpha^*}(\Omega)}^{l-q}\bigg(\int_{B_{r_2}}(\phi(x)\psi_+(x))^{q_\alpha^*}dx\bigg)^{\frac{q}{q_\alpha^*}}.
    \end{align}
    By Corollary 4.10 in \cite{C2017}, and since $r_2-r_3=\frac{r_2-r_1}{2}$, we have
    \begin{align}\label{D-L8-40}
        \bigg(\int_{B_{r_2}}(\phi(x)\psi_+(x))^{q_\alpha^*}dx\bigg)^{\frac{q}{q_\alpha^*}}&\leq C\left([\phi\psi_+]_{W^{\alpha,q}(B_{r_2})}^q+\frac{1}{(r_2-r_3)^{\alpha q}}\|\phi\psi_+\|_{L^q(B_{r_3})}^q\right)\nonumber\\
        &\leq C\left([\phi\psi_+]_{W^{\alpha,q}(B_{r_2})}^q+\frac{1}{(r_2-r_1)^{\alpha q}}\|\psi_+\|_{L^q(B_{R})}^q\right).
    \end{align}
    Now, since $|\phi|\leq 1,\ |\nabla \phi|<\frac{4}{r_2-r_1}$, we have 
    \begin{align}\label{D-L8-41}
        |\phi(x)\psi_+(x)-\phi(y)\psi_+(y)|^p&\leq 2^q\left(\phi(y)^p|\psi_+(x)-\psi_+(y)|^q+\psi_+(x)^q|\phi(x)-\phi(y)|^q\right)\nonumber\\
        &\leq 2^q\left(|\psi_+(x)-\psi_+(y)|^q+\psi_+(x)^q\frac{|x-y|^q)}{|r_2-r_1|^q}\right).
    \end{align}
    Since $r_2-r_1,R\leq 1$ and $\alpha<t$, we have $R^{(1-t)q}\geq (r_2-r_1)^{(1-t)q}\geq (r_2-r_1)^{(1-\alpha)q}$ and $R^{(1-t)q} \geq R^{(1-\alpha)q}$. Therefore, substituting \eqref{D-L8-41} in \eqref{D-L8-40}, we obtain
    \begin{align}\label{D-L8-42}
        \bigg(\int_{B_{r_2}}&(\phi(x)\psi_+(x))^{q_\alpha^*}dx\bigg)^{\frac{q}{q_\alpha^*}}\nonumber\\
        &\leq C\Bigg([\psi_+]_{W^{\alpha,q}(B_{r_2})}^q+\frac{1}{(r_2-r_1)^q}\int_{B_{r_2}}\psi_+(x)^q\int_{B_{r_2}}\frac{1}{|x-y|^{N+\alpha q-q}}dydx\nonumber\\
        &\hspace{1cm}+\frac{1}{(r_2-r_1)^{\alpha q}}\|\psi_+\|_{L^q(B_{R})}^q\Bigg)\nonumber\\
       % &\leq C\left([\psi_+]_{W^{\alpha,q}(B_{r_2})}^q+\frac{R^{(1-\alpha)q}}{(r_2-r_1)^q}\|\psi_+\|_{L^q(B_{R})}^q+\frac{1}{(r_2-r_1)^{\alpha q}}\|\psi_+\|_{L^q(B_{R})}^q\right)\nonumber\\
        &\leq C\left([\psi_+]_{W^{\alpha,q}(B_{r_2})}^q+\frac{R^{(1-t)q}}{(r_2-r_1)^q}\|\psi_+\|_{L^q(B_{R})}^q\right).
    \end{align}
    %Since $r_2-r_1<R$, we have $\frac{1}{(r_2-r_1)^\alpha t}$
    Combining \eqref{D-L8-39} and \eqref{D-L8-42}, we get
    \begin{align}\label{D-L8-43}
        \int\limits_{B_{r_2}}\phi(x)^l\psi_+(x)^l dx&\leq C_3'|B_1|^{\tilde{\epsilon}} R^{N\tilde{\epsilon}} \|u\|_{L^{q_\alpha^*}(\Omega)}^{l-q}\bigg([\psi_+]_{W^{\alpha,q}(B_{r_2})}^q+\frac{R^{(1-t)q}}{(r_2-r_1)^q}\|\psi_+\|_{L^q(B_{R})}^q\bigg).
    \end{align}
    Choose $R\leq  \min\Bigg\{\frac{1}{2},\left(\frac{1}{2C_3C_3'|B_1|^{\tilde{\epsilon}}\|u\|_{L^{q_\alpha^*}(\Omega)}^{l-q}}\right)^{\frac{1}{N\tilde{\epsilon}}}\Bigg\}$. Then, \eqref{D-L8-43} gives
    \begin{equation}\label{D-L8-44}
        C_3\int_{B_{r_2}}\phi(x)^l\psi_+(x)^l dx\leq \frac{1}{2}\left([\psi_+]_{W^{\alpha,q}(B_{r_2})}^q+\frac{R^{(1-t)q}}{(r_2-r_1)^q}\|\psi_+\|_{L^q(B_{R})}^q\right).
    \end{equation}
    Now, using Lemma 4.5 in \cite{C2017}, we get
    \begin{equation}\label{D-L8-45}
        [\psi_+]_{W^{\alpha,q}(B_{r_2})}^q\leq [\psi_+]_{W^{t,q}(B_{r_2})}^q+\frac{2^q|B_1|}{\alpha q}\chi_{(0,2R)}(1)\|u\|_{L^q(B_{r_2})}^q=[\psi_+]_{W^{t,q}(B_{r_2})}^q.
    \end{equation}
     Therefore, substituting \eqref{D-L8-45} in \eqref{D-L8-44}, we deduce
    \begin{align}\label{D-L8-46}
        C_3\int_{B_{r_2}}\phi(x)^l\psi_+(x)^l dx&\leq \frac{1}{2}\left([\psi_+]_{W^{t,q}(B_{r_2})}^q+\frac{R^{(1-t)q}}{(r_2-r_1)^q}\|\psi_+\|_{L^q(B_{R})}^q\right) \nonumber\\
        &\leq\frac{1}{2}  \left( P(r_2)+\frac{R^{(1-t)q}}{(r_2-r_1)^q}\|\psi_+\|_{L^q(B_{R})}^q\right).
    \end{align}
    Thus, we use $\tilde{\epsilon}$ given by \eqref{D-L8-EPSILON} and choose $\theta=c_1^{\frac{1}{q-1}}, \ \tilde{k}=0, \rho=1-\frac{l}{q_\alpha^*}$ and
    $$\tilde{R}=\min\Bigg\{\frac{1}{2}, \left(\frac{1}{2C_3C_3'|B_1|^{\tilde{\epsilon}}\|u\|_{L^{q_\alpha^*}(\Omega)}^{l-q}}\right)^{\frac{1}{N\tilde{\epsilon}}}, \left(\frac{1}{\|u\|_{L^{q_\alpha^*}(\Omega)}^{l-q} |B_1|^{\frac{(t-\alpha)q}{N}}} \right)^{\frac{1}{N\tilde{\epsilon}}} \Bigg\}.$$
    Let $R\leq \tilde{R},\ k\geq \tilde{k}$. Note that $P(.)$ is a non-decreasing function in $[0,R]$. It is easy to see that since $\tilde{\epsilon}\leq \frac{tq}{N}$ and $R\leq 1$,
    \begin{equation}\label{D-L8-46'}
        |B_{R}\cap \operatorname{supp}\psi_+|\leq \frac{|B_{R}\cap \operatorname{supp}\psi_+|.|B_R|^{\frac{tq}{N}-\tilde{\epsilon}}}{|B_{R}\cap \operatorname{supp}\psi_+|^{\frac{tq}{N}-\tilde{\epsilon}}}\leq C(N,t,q)|B_{R}\cap \operatorname{supp}\psi_+|^{1-\frac{tq}{N}+\tilde{\epsilon}}.
    \end{equation}
    Then, for each $0<r<R\leq \max\{\operatorname{dist}(x_0,\partial\Omega),\tilde{R}\}, k \geq \tilde{k} $, by combining \eqref{D-L8-33}, \eqref{D-L8-34}, \eqref{D-L8-38}, \eqref{D-L8-46} and \eqref{D-L8-46'}, we get a constant $C_4=C(N,p,q,s,t,\Lambda_1,\Lambda_2,c_2)\geq 1$ such that for $r\leq r_1\leq R$,
    \begin{align}\label{D-L8-47}
        P(r_1)& \leq C_4\Bigg(P(r_2)-P(r_1)+\frac{R^{(1-s)}}{(r_2-r_1)^p}\|\psi_+\|_{L^p(B(R))}^p+\frac{R^{(1-t)}}{(r_2-r_1)^q}\|\psi_+\|_{L^q(B(R))}^q\nonumber\\
        &\ \ \ +\bigg(\frac{R}{r_2-r_1}\bigg)^{N+sp}r^{-sp}\|\psi_+\|_{L^1(B_{R})}\left(\operatorname{Tail}_{s,p}(\psi_+,x_0,r)\right)^{p-1}\nonumber\\
        &\ \ \ +\bigg(\frac{R}{r_2-r_1}\bigg)^{N+tq}r^{-tq}\|\psi_+\|_{L^1(B_{R})}\left(\operatorname{Tail}_{a,t,q}(\psi_+,x_0,r)\right)^{q-1}\nonumber\\
        &\ \ \ +(R^{\frac{tq}{q-1}}c_1^{\frac{q}{q-1}}|\operatorname{supp}\psi_+ \cap B_R|^{1-\frac{tq}{N}+\epsilon}+\frac{|k|^q}{R^{N\tilde{\epsilon}}} |B_{R}\cap \operatorname{supp}\psi_+|^{1-\frac{tq}{N}+\tilde{\epsilon}}\Bigg).
    \end{align}
    Adding $C_4P(r_1)$ on both sides of \eqref{D-L8-47} and diving by $(1+C_4)$, we get
    \begin{equation*}
        P(r_1)\leq \gamma\left(P(r_2)+D_1+\frac{D_2}{|r_2-r_1|^{p}}+\frac{D_3}{|r_2-r_1|^{q}}+\frac{D_4}{|r_2-r_1|^{N+sp}}\frac{D_5}{|r_2-r_1|^{N+tq}}\right),
    \end{equation*}
    where $\gamma=\frac{C_4}{C_4+1}$ and $D_i$ are constants given by
    \begin{align*}
        D_1&=\left(R^{\frac{tq}{q-1}}c_1^{\frac{q}{q-1}}+\frac{|k|^q}{R^{N\tilde{\epsilon}}}\right) |B_{R}\cap \operatorname{supp}\psi_+|^{1-\frac{tq}{N}+\tilde{\epsilon}}\\
        D_2&={R^{(1-s)}}\|\psi_+\|_{L^p(B(R))}^p, \ D_3={R^{(1-t)}}\|\psi_+\|_{L^q(B(R))}^q \\
        D_4&={R}^{N+sp}r_2^{-sp}\|\psi_+\|_{L^1(B_{R})}\left(\operatorname{Tail}_{s,p}(\psi_+,x_0,r)\right)^{p-1} \\
        D_5&=\bigg(\frac{R}{r_2-r_1}\bigg)^{N+tq}\|\psi_+\|_{L^1(B_{R})}\left(\operatorname{Tail}_{a,t,q}(\psi_+,x_0,r)\right)^{q-1}.
    \end{align*}
    Also note that for every $r'\in(0,R)$, we have
    $$[\psi_+]_{W_{s,p}(B_{r'})}^p+[\psi_+]_{W_{t,q,a}(B_{r'})}^q+\int_{B_{r'}}\int_{B_{2\tilde{R}}}\psi_+(x)\bigg(\frac{\psi_-(y)^{p-1}}{|x-y|^{N+sp}}+\frac{a(x,y)\psi_-(y)^{q-1}}{|x-y|^{N+tq}}\bigg)dydx \leq P(R).$$
    Hence, using Lemma \ref{D-L10}, the required result follows.
\end{proof}

\noindent The following lemma gives the local boundedness from above for the functions in the double phase De Giorgi class $\operatorname{DG}_a^+(\mathcal{D},\theta, \tilde{C}, \tilde{k},\tilde{\epsilon},\rho,\tilde{R})$.

\begin{lemma}\label{D-L9}
    Let $tq\leq N$ and $\mathcal{D}\subset\mathbb{R}^N$ be an open set. Consider $u\in \operatorname{DG}_a^+(\mathcal{D},\theta, \tilde{C}, \tilde{k},\tilde{\epsilon},\rho,\tilde{R})$ such that $\theta,\rho\geq 0, \ \tilde{C}\geq 1, \ \tilde{k}\geq 0,\ \ 0<\tilde{\epsilon}\leq \frac{tq}{N}$ and $\tilde{R}\leq \infty$. There {exist} constants $0<\rho'=\rho'(N,t,q,\tilde{\epsilon})$ and $C=C(N,s,p,t,q,\tilde{\epsilon},\tilde{C},\tilde{k})\geq 1$ such that for any $B_{2R}(x_0)\subset \mathcal{D}$ with $2R\leq \tilde{R}$ and $0<\epsilon\leq\tilde{\epsilon}$, the following holds:
    \begin{align*}
        \sup\limits_{B_R(x_0)}u&\leq \operatorname{Tail}_{s,p}((u-\tilde{k})_+,x_0,R)+ \operatorname{Tail}_{a,t,q}((u-\tilde{k})_+,x_0,R)+R^{\frac{\rho+N\tilde{\epsilon}}{q}}\theta+2\tilde{k}\\
        &\ \ \ +C\left[\frac{1}{R^N}\left(\|(u-\tilde{k})_+\|_{L^p(B_{2R}(x_0))}^p+\left(\frac{R}{2}\right)^{sp-tq}\|(u-\tilde{k})_+\|_{L^q(B_{2R}(x_0))}^q\right)\right]^{\frac{1}{p}}.
    \end{align*}
\end{lemma}
\begin{proof}
     For each $k\in \mathbb{R}$, let $v_k=(u-k)_+$. Since $a(.,.)$ is continuous and strictly positive in $\overline{\mathcal{D}}$, we have $\inf\{a(x,y):x,y \in \overline{\mathcal{D}}\}>0$. Hence, choose $m>1$ such that
    \begin{equation}\label{D-L9-1''}
        [h]_{W_{t,q}(B_{r'})}\leq m [h]_{W_{t,q,a}(B_{r'})}
    \end{equation}
    for every $0<r'<\min\{\operatorname{dist}(x_0,\partial\mathcal{D}),\tilde{R}\}$ and $h\in W_{t,q}(\mathcal{D})$.
    Let $x_0\in \Omega$. Without loss of generality, we assume that $x_0=0$. Choose $R>0$ such that $2R\leq \min\{\tilde{R}, \operatorname{dis
    }(x_0, \partial\Omega)\}$. Define $A_k=\operatorname{supp}v_k, \ a_{k,r'}=\|v_k\|_{L^p(B_{r'})}^p$ and $b_{k,r'}=\|v_k\|_{L^q(B_{r'})}^q$ for each $k\geq k', r'>0$. We first define two sequences $(r_i)$ and $(k_i)$ by 
    $$r_i=R(1+\frac{1}{2^i}),\ k_i=k_0+m_0(1-\frac{1}{2^i}),$$
    where $i\in \{0,1,2,...\}$. Then, $(r_i)$ is a decreasing sequence with $R\leq r_i\leq 2R$ and $r_i\rightarrow R$ as $i \rightarrow\infty$. Also, $(k_i)$ is an increasing sequence converging to $k_0+m_0$, where the constant $m_0$ is to be chosen later. Fix an $i\in \{0,1,2,...\}$ and let $r_i'=\frac{r_i+r_{i+1}}{2}$. Choose $\zeta\in C_c^\infty(\mathbb{R}^N)$ satisfying $\operatorname{supp}\zeta \subset B_{r_i'},\ 0\leq \zeta \leq 1, \ \zeta \equiv 1$ in $B_{r_{i+1}}$ and $\|\nabla\zeta\|_{L^\infty(\mathbb{R}^N)}\leq \frac{c}{r_{i}-r_{i+1}}$. For simplicity, denote $v_{k_i}=v_i$ and $A_{k_i}=A_i$. Define $w=\zeta v_{{i+1}}$. Then, we have $w\in W_0^{t,q}(B_{r_i'})$. Let $0<\epsilon <\tilde{\epsilon}$. Then, let 
    \begin{equation}\label{D-L9-beta}
        \beta=\begin{cases}
            t, &\ \text{if }tq<N, \\
            \max\{2t-1, t-\frac{N\epsilon}{2q}\}, &\ \text{if }tq=N.
        \end{cases}
    \end{equation}
    Then, we have $\beta \leq t, \beta q<N$. Since $w\in W_0^{t,q}(B_{r_i'})$, by Theorem \ref{fSS-CE}, we have $w\in  W_0^{\beta,q}(B_{r_i'})$. Note that $r_i'-r_{i+1}=\frac{r_i-r_{i+1}}{2}$. Thus, using Corollary 4.10 in \cite{C2017} we get a constant $C>0$ such that
    \begin{align}\label{D-L9-1}
        \|w\|_{L^{q_\beta^*}(B_{r_i'})}^q&\leq C\frac{1-\beta}{(N-\beta q)^{q-1}}\left([w]_{W^{\beta,q}(B_{r_i'})}^q+\frac{1}{(r_i'-r_{i+1})^{\beta q}}\|w\|_{L^q(B_{r_{i+1}})}^q \right)\nonumber\\
        %&\leq C\frac{1-\beta}{(N-\beta q)^{q-1}}\left([w]_{W^{\beta,q}(B_{r_3})}^q+\frac{1}{(r_2-r_1)^{\beta q}}\|w\|_{L^q(B_{r_3})}^q \right)\nonumber\\
        &\leq C(N,t,q,\epsilon)\left([w]_{W^{\beta,q}(B_{r_i'})}^q+\frac{1}{(r_i-r_{i+1})^{ \beta q}}\|v_{{i+1}}\|_{L^q(B_{r_i})}^q \right).
    \end{align}
    Using H\"older's inequality and \eqref{D-L9-1}, we deduce
    \begin{align}\label{D-L9-2}
        b_{k_{i+1},r_{i+1}}&=\int_{B_{r_{i+1}} \cap A_{{i+1}}}|v_{i+1}|^q dx \nonumber\\
        &\leq \left(\int_{B_{r_{i+1}} \cap A_{{i+1}}} |v_{i+1}|^{q_\beta^*}\right)^{\frac{N-\beta q}{N}}\left( \int_{B_{r_{i+1}} \cap A_{{i+1}}} 1 dx\right)^{\frac{\beta q}{N}}\nonumber\\
        &\leq \|w\|_{L_{q_\beta^*}(B_{r_i'})}^q|B_{r_{i+1}} \cap A_{{i+1}}|^{\frac{\beta q}{N}}\nonumber\\
        &\leq C|B_{r_{i+1}} \cap A_{{i+1}}|^{\frac{\beta q}{N}}\left([w]_{W^{\beta,q}(B_{r_i'})}^q+\frac{1}{(r_i-r_{i+1})^{\beta q}}\|v_{i+1}\|_{L^q(B_{r_i})}^q \right).%\nonumber\\
        %&\leq |B_{r_{i+1}} \cap A_{{i+1}}|^{\frac{\beta q}{N}}C\frac{1-\beta}{(N-\beta q)^{q-1}}\left([w]_{W^{\beta,q}(B_{r_i'})}^q+\frac{1}{(r_i-r_{i+1})^{\beta q}}\frac{r_i^q}{(r_i-r_{i+1})^q}\|v_{i+1}\|_{L^q(B_{r_i})}^q \right).
    \end{align}
    Using Lemma 4.5 in \cite{C2017}%with $\delta=r_i-r_{i+1}<r_i'$
    , we have
    \begin{align}\label{D-L9-3}
        [w]_{W^{\beta,q}(B_{r_{i'}})}^q&\leq (r_i-r_{i+1})^{tq-\beta q} [w]_{W^{t,q}(B_{r_i'})}^q+2^q\frac{|B_1|}{\beta q (r_i-r_{i+1})^{\beta q}}\chi_{[0,2r_i']}(r_i-r_{i+1})\|w\|_{L^q(B_{r_i'})}^q\nonumber\\
        &=(r_i-r_{i+1})^{tq-\beta q} [w]_{W^{t,q}(B_{r_i'})}^q+\frac{C}{(r_i-r_{i+1})^{\beta q}}\|v_{i+1}\|_{L^q(B_{r_i})}^q.
    \end{align}
    Now, for all $x,y \in \mathbb{R}^N$, we have
    \begin{equation}\label{D-L9-4}
        |w(x)-w(y)|^q\leq 2^q(\zeta(y)^q|v_{i+1}(x)-v_{i+1}(y)|^q+v_{i+1}(x)^q|\zeta(x)-\zeta(y)|^q).
    \end{equation}
    Utilizing \eqref{D-L9-4} and since $|\zeta|\leq 1, \ |\nabla \zeta| \leq \frac{C}{r_2-r_1}$, we get
    \begin{align}\label{D-L9-4'}
        [w]_{W^{t,q}(B_{r_i'})}^q&\leq 2^q\bigg([v_{i+1}]_{W^{t,q}(B_{r_i'})}^q+\int_{B_{r_i'}}v_{i+1}(x)^q\int_{B_{r_i'}}\frac{|\zeta(x)-\zeta(y)|^q}{|x-y|^{N+tq}}dydx\bigg)\nonumber\\
        &\leq 2^q \bigg([v_{i+1}]_{W^{t,q}(B_{r_i'})}^q+\int_{B_{r_i'}}v_{i+1}(x)^q\int_{B_{r_i}}\frac{\|\nabla\zeta\|_{L^\infty(\mathbb{R}^N)}^q|x-y|^{q}}{|x-y|^{N+tq}}dydx\bigg)\nonumber\\
        &\leq C\bigg([v_{i+1}]_{W^{t,q}(B_{r_i'})}^q+\frac{1}{(r_i-r_{i+1})^q}\int_{B_{r_i'}}v_{i+1}(x)^q\int_{B_{r_i}}\frac{1}{|x-y|^{N+tq-q}}dydx\bigg)\nonumber\\
        &\leq C\bigg([v_{i+1}]_{W^{t,q}(B_{r_i'})}^q+\frac{r_i^{(q-tq)}}{(r_i-r_{i+1})^q}\|v_{i+1}\|_{L^q(B_{r_i})}^q\bigg).
    \end{align}
    Combining \eqref{D-L9-3} and \eqref{D-L9-4'} with \eqref{D-L9-2} and since $\frac{r_{i}}{r_i-r_{i+1}}\geq 1$, we deduce
    \begin{align}\label{D-L9-5}
        b_{k_{i+1},r_{i+1}}&\leq \frac{C|B_{r_{i+1}} \cap A_{i+1}|^{\frac{\beta q}{N}}}{(r_i-r_{i+1})^{\beta q}}\bigg((r_i-r_{i+1})^{tq}[v_{i+1}]_{W^{t,q}(B_{r_i'})}^q+\frac{r_{i}^q}{(r_i-r_{i+1})^q}\|v_{i+1}\|_{L^q(B_{r_i})}^q \bigg).%\nonumber\\
        %&\leq \frac{C|B_{r_{i+1}} \cap A_{i+1}|^{\frac{\beta q}{N}}}{(r_i-r_{i+1})^{\beta q}}\left((r_i-r_{i+1})^{tq}[v_{i+1}]_{W^{t,q}(B_{r_3})}^q+\frac{r_i^{(t-tq)}}{(1-t)q(r_i-r_{i+1})^q}\|v_{i+1}\|_{L^q(B_{r_i})}\right).
    \end{align}
    Now, let $\alpha= \frac{\beta q}{p}$. Then, we have
    $\alpha p=\beta q <N$. Also, since $\alpha p =\beta q \leq tq \leq sp$, we have $\alpha \leq s$.
    Thus, proceeding similarly, we obtain
    \begin{align}\label{D-L9-6}
        a_{k_{i+1},r_{i+1}}
        &\leq \frac{C|B_{r_{i+1}} \cap A_{i+1}|^{\frac{\alpha p}{N}}}{(r_i-r_{i+1})^{\alpha p}}\left((r_i-r_{i+1})^{sp}[v_{i+1}]_{W^{s,p}(B_{r_i'})}^p+\frac{r_{i}^p}{(r_i-r_{i+1})^p}\|v_{i+1}\|_{L^p(B_{r_i})}^p\right).
    \end{align}
    Note that $r_i-r_{i+1}=\frac{R}{2^{i+1}}$. Therefore, define the sequence $(c_i)$ for $i\in\{0,1,2,...\}$ by
    $$c_i=a_{k_i,r_i}+\left(\frac{R}{2^{i+1}}\right)^{sp-tq}b_{k_i,r_i}=a_{k_i,r_i}+(r_i-r_{i+1})^{sp-tq}b_{k_i,r_i}.$$
    From \eqref{D-L9-5} and \eqref{D-L9-6}, we have
    \begin{align}\label{D-L9-7}
        c_{i+1}&\leq \frac{C|B_{r_{i+1}} \cap A_{i+1}|^{\frac{\alpha p}{N}}}{(r_i-r_{i+1})^{\alpha p-sp}}\bigg([v_k]_{W^{s,p}(B_{r_i'})}^p+[v_{i+1}]_{W^{t,q}(B_{r_i'})}^q\nonumber\\
        &\ \ \ +\frac{r_{i}^p}{(r_i-r_{i+1})^{p+sp}}\|v_{i+1}\|_{L^p(B_{r_i})}+\frac{r_{i}^q}{(r_i-r_{i+1})^{q+tq}}\|v_{i+1}\|_{L^q(B_{r_i'})}^q\bigg).
    \end{align}
    Now, since $u\in \operatorname{DG}_a^+(\mathcal{D},\theta, \tilde{C}, \tilde{k},\tilde{\epsilon},\rho,\tilde{R})$, using \eqref{D-L9-1''} and the fact that $\frac{r_i}{r_i-r_{i+1}}\geq1$, we deduce
    \begin{align}\label{D-L9-1'}
    [v_{i+1}]_{W_{s,p}(B_{r_i'})}^p+[v_{i+1}]_{W_{t,q}(B_{r_i'})}^q &\leq m\Bigg([v_{i+1}]_{W_{s,p}(B_{r_i'})}^p+[v_{i+1}]_{W_{t,q,a}(B_{r_i'})}^q\Bigg) \nonumber\\
        & \leq C\Bigg(\frac{r_i^{p}}{(r_i-r_i')^{p+sp}}\|v_{i+1}\|_{L^p(B(r_i))}^p\nonumber\\
        &\ \  +\frac{r_i^{q}}{(r_i-r_{i+1})^{q+tq}}\|v_{i+1}\|_{L^q(B(r_i))}^q\nonumber\\
        &\ \  +\bigg(\frac{r_i}{r_i-r_i'}\bigg)^{N+sp}R^{-sp}\|v_{i+1}\|_{L^1(B_{r_i})}\left(\operatorname{Tail}_{s,p}(v_{i+1},x_0,R)\right)^{p-1}\nonumber\\
        &\ \  +\bigg(\frac{r_i}{r_i-r_{i+1}}\bigg)^{N+tq}R^{-tq}\|v_{i+1}\|_{L^1(B_{r_i})}\left(\operatorname{Tail}_{a,t,q}(v_{i+1},x_0,R)\right)^{q-1}\nonumber\\
        &\ \  +\left(r_i^{\rho}\theta^{q}+\frac{|k_{i+1}|^q}{r_i^{N\tilde{\epsilon}}}\right)|\operatorname{supp}v_{i+1} \cap B_{r_i}|^{1-\frac{tq}{N}+\tilde{\epsilon}}\Bigg).
    \end{align}
    Substituting \eqref{D-L9-1'} in \eqref{D-L9-7}, we obtain
    \begin{align}\label{D-L9-8}
        c_{i+1}&\leq \frac{C|B_{r_{i+1}} \cap A_{i+1}|^{\frac{\alpha p}{N}}}{(r_i-r_{i+1})^{\alpha p-sp}}\Bigg(\frac{r_i^{p}}{(r_i-r_i')^{p+sp}}\|v_{i+1}\|_{L^p(B(r_i))}^p+\frac{r_i^{q}}{(r_i-r_{i+1})^{q+tq}}\|v_{i+1}\|_{L^q(B(r_i))}^q\nonumber\\
        &\ \ \ +\bigg(\frac{r_i}{r_i-r_i'}\bigg)^{N+sp}R^{-sp}\|v_{i+1}\|_{L^1(B_{r_i})}\left(\operatorname{Tail}_{s,p}(v_{i+1},x_0,R)\right)^{p-1}\nonumber\\
        &\ \ \ +\bigg(\frac{r_i}{r_i-r_{i+1}}\bigg)^{N+tq}R^{-tq}\|v_{i+1}\|_{L^1(B_{r_i})}\left(\operatorname{Tail}_{a,t,q}(v_{i+1},x_0,R)\right)^{q-1}\nonumber\\
        &\ \ \ +\left(r_i^{\rho}\theta^{q}+\frac{|k_{i+1}|^q}{r_i^{N\tilde{\epsilon}}}\right)|\operatorname{supp}v_{i+1} \cap B_{r_i}|^{1-\frac{tq}{N}+\tilde{\epsilon}}\Bigg)\nonumber\\
        &= B_1+B_2,
    \end{align}
    where $B_1$ and $B_2$ are given by
    \begin{align*}
        B_1&=C|B_{r_{i+1}} \cap A_{i+1}|^{\frac{\beta q}{N}}(r_i-r_{i+1})^{sp-tq}(r_i-r_{i+1})^{tq-\beta q}\Bigg(\frac{r_i^{q}}{(r_i-r_i')^{q+tq}}\|v_{i+1}\|_{L^q(B(r_i))}^q\nonumber\\
        &\ \ \ +\bigg(\frac{r_i}{r_i-r_{i+1}}\bigg)^{N+tq}R^{-tq}\|v_{i+1}\|_{L^1(B_{r_i})}\left(\operatorname{Tail}_{a,t,q}(v_{i+1},x_0,R)\right)^{q-1} \nonumber\\
        &\ \ \ +\left(r_i^{\frac{tq}{q-1}}c_1^{\frac{q}{q-1}}+\frac{|k_{i+1}|^q}{r_i^{N\epsilon}}\right)|\operatorname{supp}v_{i+1} \cap B_{r_i}|^{1-\frac{tq}{N}+\epsilon}\Bigg) \text{ and }\\
        B_2&=C|B_{r_{i+1}} \cap A_{i+1}|^{\frac{\alpha p}{N}}(r_i-r_{i+1})^{sp-\alpha p}\Bigg(\frac{r_i^{p}}{(r_i-r_i')^{p+sp}}\|v_{i+1}\|_{L^p(B(r_i))}^p\nonumber\\
        &\ \ \ +\bigg(\frac{r_i}{r_i-r_i'}\bigg)^{N+sp}R^{-sp}\|v_{i+1}\|_{L^1(B_{r_i})}\left(\operatorname{Tail}_{s,p}(v_{i+1},x_0,R)\right)^{p-1}\Bigg).
    \end{align*}
    Since $r_i\geq r_{i+1}$, we have $|B_{r_{i+1}}|\leq |B_{r_i}|$. Also, since $(k_i)$ is an increasing sequence, we have $ A_{i+1}\subset A_i$ and $u-k_i \geq u-k_{i+1}$ in $A_{i+1}$. Therefore,
    \begin{align*}
        \|v_{i+1}\|_{L^q(B(r_i))}^q=\int_{A_{i+1}\cap B_{r_i}}(u(x)-k_{i+1})^q dx\leq \int_{A_{i}\cap B_{r_i}}(u(x)-k_{i})^q dx=\|v_{i}\|_{L^q(B(r_i))}^q.
    \end{align*}
    Similarly, we have $\|v_{i+1}\|_{L^p(B(r_i))}^p\leq \|v_{i}\|_{L^p(B(r_i))}^p$.
    Define 
    $$\rho'=\tilde{\epsilon}-\frac{tq}{N}+\frac{\beta q}{N}.$$ Observe that by \eqref{D-L9-beta}, we have $0<\tilde{\epsilon}-\frac{\epsilon}{2}\leq\rho'\leq \frac{\beta q}{N}$. Also, we have that $\frac{|B_{r_{i}} \cap A|}{|B_{r_{i}}|}\leq 1$ for every $A\subset \mathbb{R}^N$ and $|B_{r_i}|=C r^N$ for every $r>0$. Hence we get
    \begin{align}\label{D-L9-9}
        |B_{r_{i+1}} \cap A_{i+1}&|^{\frac{\beta q}{N}}(r_i-r_{i+1})^{tq-\beta q}\frac{r_i^{q}}{(r_i-r_i')^{q+tq}}\|v_{i+1}\|_{L^q(B(r_i))}^q\nonumber\\
        &\leq C\left(\frac{|B_{r_{i}} \cap A_{i+1}|}{|B_{r_i}|}\right)^{\frac{\beta q}{N}}\frac{r_i^{q}}{(r_i-r_{i+1})^{q}}\|v_{i}\|_{L^q(B(r_i))}^q \nonumber\\
        &\leq C\left(\frac{|B_{r_{i}|} \cap A_{i+1}|}{|B_{r_i}|}\right)^{\rho'}\frac{r_i^{q}}{(r_i-r_{i+1})^{q}}b_{k_i,r_i}.
    \end{align}
    Since $k_{i}\leq k_{i+1}$, whenever $x\in A_{i+1}$, we have $v_i=u(x)-k_{i}\geq k_{i+1}-k_i\geq 0$. Thus,
        $$\|v_i\|_{L^q(B_{r_i})}^q \geq \int_{B_{r_i}\cap A_i}(k_{i+1}-k_i)^{q-1}v_i(x)dx=(k_{i+1}-k_i)^{q-1}\|v_i\|_{L^1(B_{r_i})}.$$
    It can be seen similarly that $\|v_i\|_{L^p(B_{r_i})}^p\geq (k_{i+1}-k_i)^{p-1}\|v_i\|_{L^1(B_{r_i})}$. Therefore, we deduce
    \begin{align}\label{D-L9-10}
        &|B_{r_{i+1}} \cap A_{i+1}|^{\frac{\beta q}{N}}(r_i-r_{i+1})^{tq-\beta q}\bigg(\frac{r_i}{r_i-r_{i+1}}\bigg)^{N+tq}R^{-tq}\|v_{i+1}\|_{L^1(B_{R})}\left(\operatorname{Tail}_{a,t,q}(v_{i+1},x_0,R)\right)^{q-1}\nonumber\\
        &\leq C\left(\frac{|B_{r_{i}} \cap A_{i+1}|}{|B_{r_i}|}\right)^{\frac{\beta q}{N}}\frac{r_i^{N+tq+\beta q}}{(r_i-r_{i+1})^{N+\beta q}}R^{-tq}\left(\operatorname{Tail}_{a,t,q}(v_{i+1},x_0,R)\right)^{q-1}\frac{\|v_i\|_{L^q(B_{r_i})}^q}{(k_{i+1}-k_i)^{q-1}}\nonumber\\
        &\leq C\left(\frac{|B_{r_{i}} \cap A_{i+1}|}{|B_{r_i}|}\right)^{\rho'}\frac{r_i^{N+tq+\beta q}}{(r_i-r_{i+1})^{N+\beta q}}\frac{R^{-tq}}{(k_{i+1}-k_i)^{q-1}}\left(\operatorname{Tail}_{a,t,q}(v_{i+1},x_0,R)\right)^{q-1}b_{k_i,r_i}.
    \end{align}
    Now, since $v_i \geq k_{i+1}-k_i$ in $A_{i+1}$, we also get that
    $\|v_{i+1}\|_{L^q(B_{r_i})}^q \geq (k_{i+1}-k_i)^q|B_{r_i}\cap A_{i+1}|$ and $\|v_{i+1}\|_{L^p(B_{r_i})}^p \geq (k_{i+1}-k_i)^p|B_{r_i}\cap A_{i+1}|$. Thus, we obtain that
    \begin{align}\label{D-L9-11}
        |B_{r_{i+1}} \cap A_{i+1}&|^{\frac{\beta q}{N}}(r_i-r_{i+1})^{tq-\beta q}\left(r_i^{\rho}\theta^q+\frac{|k_{i+1}|^q}{r_i^{N\tilde{\epsilon}}}\right)|B_{r_{i}} \cap A_{i+1}|^{1-\frac{tq}{N}+\tilde{\epsilon}} \nonumber\\
        &\leq|B_{r_{i+1}} \cap A_{i}|^{1+\rho'}r_i^{tq-\beta q}\left(\frac{r_i^{\rho+N\tilde{\epsilon}}\theta^q+|k_{i+1}|^q}{r_i^{N\tilde{\epsilon}}}\right)\nonumber\\
        &\leq C\left(\frac{|B_{r_{i}} \cap A_{i+1}|}{|B_{r_i}|}\right)^{\rho'}\left(r_i^{\rho+N\tilde{\epsilon}}\theta^q+|k_{i+1}|^q\right)\frac{\|v_{i+1}\|_{L^q(B_{r_i})}^q}{(k_{i+1}-k_i)^q}\nonumber\\
        &=C\left(\frac{|B_{r_{i}} \cap A_{i+1}|}{|B_{r_i}|}\right)^{\rho'}\left(\frac{r_i^{\rho+N\tilde{\epsilon}}\theta^q+|k_{i+1}|^q}{(k_{i+1}-k_i)^q}\right)b_{k_i,r_i}.
    \end{align}
    Combining the equations \eqref{D-L9-9}, \eqref{D-L9-10}, and \eqref{D-L9-11}, we get
    \begin{align}\label{D-L9-13}
        B_1 &\leq C\left(\frac{|B_{r_{i}} \cap A_{i+1}|}{|B_{r_i}|}\right)^{\rho'}\Bigg(\frac{r_i^{q}}{(r_i-r_{i+1})^{q}}\nonumber\\
        &\hspace{1cm}+\frac{r_i^{N+tq+\beta q}}{(r_i-r_{i+1})^{N+\beta q}}\frac{R^{-tq}}{(k_{i+1}-k_i)^{q-1}}\left(\operatorname{Tail}_{a,t,q}(v_{i+1},x_0,R)\right)^{q-1}\nonumber\\
        &\hspace{1cm}+\left(\frac{r_i^{\rho+N\tilde{\epsilon}}\theta^q+|k_{i+1}|^q}{(k_{i+1}-k_i)^q}\right)\Bigg)\left(\frac{R}{2^{i+1}}\right)^{sp-tq}b_{k_i,r_i}.
    \end{align}
    Similarly, we obtain that
    \begin{align}\label{D-L9-14}
        B_2 &\leq C\left(\frac{|B_{r_{i}} \cap A_{i+1}|}{|B_{r_i}|}\right)^{\rho'}\Bigg(\frac{r_i^{p}}{(r_i-r_{i+1})^{p}}+\frac{r_i^{N+sp+\alpha p}}{(r_i-r_{i+1})^{N+\alpha p}}\frac{R^{-sp}}{(k_{i+1}-k_i)^{p-1}}\nonumber\\
        &\hspace{1cm} \times(\operatorname{Tail}_{s,p}(v_{i+1},x_0,R))^{p-1}\Bigg)a_{k_i,r_i}.
    \end{align}
    Now, for each $i \in \{0,1,2,...\}$, we have $r_i \leq 2R$ and $k_{i+1}^q\leq 2^q(\tilde{k}^q+m_0^q)$. Also, since $v_{i+1}\leq v_0$ in $A_{i+1}$, we have
    \begin{align*}
        \operatorname{Tail}_{a,t,q}(v_{i+1},x_0,R) &\leq \operatorname{Tail}_{a,t,q}(v_{0},x_0,R), \text{ and }\\
        \operatorname{Tail}_{s,p}(v_{i+1},x_0,R) &\leq \operatorname{Tail}_{s,p}(v_{0},x_0,R).
    \end{align*}
    Thus, by combining \eqref{D-L9-8}, \eqref{D-L9-13} and \eqref{D-L9-14}, we deduce
    \begin{align}\label{D-L9-15}  
        c_{i+1}&\leq  C\left(\frac{|B_{r_{i}} \cap A_{i+1}|}{|B_{r_i}|}\right)^{\rho'}\Bigg(\frac{R^{p}}{(r_i-r_{i+1})^{p}}\nonumber\\
        &\ \ \ +\frac{R^{N+\alpha p}}{(r_i-r_{i+1})^{N+\alpha p}(k_{i+1}-k_i)^{p-1}}\left(\operatorname{Tail}_{s,p}(v_{0},x_0,R)\right)^{p-1}\nonumber\\
        &\ \ \ +\frac{R^{q}}{(r_i-r_{i+1})^{q}}+\frac{R^{N+\beta q}}{(r_i-r_{i+1})^{N+\beta q}(k_{i+1}-k_i)^{q-1}}\left(\operatorname{Tail}_{a,t,q}(v_{0},x_0,R)\right)^{q-1}\nonumber\\
        & \ \ \ +\bigg(\frac{R^{\rho+N\tilde{\epsilon}}\theta^q+\tilde{k}^q+m_0^q}{(k_{i+1}-k_i)^q}\bigg)\Bigg)c_i.
    \end{align}
    Now, since $R\leq r_i$, we have
    \begin{align}\label{D-L9-16}
        c_i^{\rho'}&\geq \|v_i\|_{L^p(B_{r_i})}^{p\rho'} \geq \left((k_{i+1}-k_i)^p|B_{r_i}\cap A_{i+1}|\right)^{\rho'}\nonumber\\
        &=C\left(\frac{|B_{r_i}\cap A_{i+1}|}{|B_{r_i}|}\right)^{\rho'}(k_{i+1}-k_i)^{p\rho'}{r_i^{N\rho'}}\geq C\left(\frac{|B_{r_i}\cap A_{i+1}|}{|B_{r_i}|}\right)^{\rho'}(k_{i+1}-k_i)^{p\rho'}R^{N\rho'}.
    \end{align}
    Substituting \eqref{D-L9-16} in \eqref{D-L9-15}, we have
    \begin{align}\label{D-L9-17}
        c_{i+1}&\leq  \frac{C}{(k_{i+1}-k_i)^{p\rho'}R^{N\rho'}}\Bigg(\frac{R^{p}}{(r_i-r_{i+1})^{p}}+\frac{R^{N+\alpha p}}{(r_i-r_{i+1})^{N+\alpha p}(k_{i+1}-k_i)^{p-1}}\left(\operatorname{Tail}_{s,p}(v_{0},x_0,R)\right)^{p-1}\nonumber\\
        &\ \ \ +\frac{R^{q}}{(r_i-r_{i+1})^{q}}+\frac{R^{N+\beta q}}{(r_i-r_{i+1})^{N+\beta q}(k_{i+1}-k_i)^{q-1}}\left(\operatorname{Tail}_{a,t,q}(v_{0},x_0,R)\right)^{q-1}\nonumber\\
        & \ \ \ +\left(\frac{R^{\rho+N\tilde{\epsilon}}\theta^q+\tilde{k}^q+m_0^q}{(k_{i+1}-k_i)^q}\right)\Bigg)c_i^{1+\rho'}.
    \end{align}
    Now, note that for all $i\geq 1$, we have $k_{i+1}-k_i=\frac{m_0}{2^{i+1}}$. Also, 
    $$\frac{R}{(r_i-r_{i+1})}=\frac{2^{i+1}R}{R}= 2^{i+1}.$$
    Thus, \eqref{D-L9-17} becomes
    \begin{align}\label{D-L9-18}
        c_{i+1}&\leq \frac{C2^{(N+3p+2q)i}}{m_0^{p\rho'}R^{N\rho'}}\Bigg(2+\frac{\left(\operatorname{Tail}_{s,p}(v_{0},x_0,R)\right)^{p-1}}{m_0^{p-1}}\nonumber\\
        &\ \ \ +\frac{\left(\operatorname{Tail}_{a,t,q}(v_{0},x_0,R)\right)^{q-1}}{m_0^{q-1}}+\frac{R^{\rho+N\tilde{\epsilon}}\theta^q+\tilde{k}^q+m_0^q}{m_0^q}\Bigg).
    \end{align}
    Given $\nu \in (0,1)$, let 
    \begin{align*}
        m_0&=\operatorname{Tail}_{s,p}(v_{0},x_0,R)+ \operatorname{Tail}_{a,t,q}(v_{0},x_0,R)+R^{\frac{\rho+N\tilde{\epsilon}}{q}}\theta+\tilde{k}\\
        &\ \ \ +\left[\frac{2^{\frac{N+3p+2q}{\rho'^2}}}{R^N}\left(\|v_{0}\|_{L^p(B_{2R})}^p+\left(\frac{R}{2}\right)^{sp-tq}\|v_{0}\|_{L^q(B_{2R})}^q\right)\right]^{\frac{1}{p}}.
    \end{align*}
    Then, from \eqref{D-L9-18}, we get a constant $C>0$ such that
    \begin{align*}
        c_{i+1}\leq C\frac{2^{(N+3p+2q)i}}{m_0^{p\rho'}R^{N\rho'}}c_i^{1+\rho'}
    \end{align*}
    and
    \begin{align*}
        c_0=\|u\|_{L^p(B_{2R})}^p+\left(\frac{R}{2}\right)^{sp-tq}\|u\|_{L^q(B_{2R})}^q\leq \left(\frac{C}{m_0^{p\rho'}R^{N\rho'}}\right)^{\frac{-1}{\rho'}}2^{\frac{-(N+3p+2q)}{\rho'^2}}.
    \end{align*}
    Hence, applying Lemma 7.1 in \cite{G2003},  we obtain 
    $$\|(u-\tilde{k}-m_0)_+\|_{L^p(B_{R})}^p=\lim\limits_{i\rightarrow \infty}c_i =0.$$ 
    Thus, we infer that $\sup\limits_{B_R}u\leq m_0+\tilde{k}$. This completes the proof.
\end{proof}
\noindent With the help of Lemma \ref{D-L8} and Lemma \ref{D-L9}, we prove Theorem \ref{D-T3}. 
\medskip

\noindent {\it{\bf{Proof of Theorem} \ref{D-T3}.}}
     Using Lemma \ref{D-L8}, we obtain $\theta\geq0,\ \tilde{C}\geq1,\tilde{k}\geq -\infty,\ 0<\tilde{\epsilon}\leq \frac{tq}{N},\ \rho\geq 0$ and $\tilde{R}\leq \infty$ such that $u\in \operatorname{DG}_a(\Omega,\theta,\tilde{C},\tilde{k}, \tilde{\epsilon},\rho,\tilde{R})$. Choose $k'\geq \tilde{k}$ such that $k'\geq 0$. Then using {Remark} \ref{D-DGC-R}, $\pm u\in \operatorname{DG}_a(\Omega,\theta,\tilde{C},k', \tilde{\epsilon},\rho,\tilde{R})$. Now, applying Lemma \ref{D-L9}, we deduce that for any $0<2R<\max\left\{\operatorname{dist}(x_0,\partial\Omega),\tilde{R}\right\}$, we have 
     \begin{align*}
        \sup\limits_{B_R(x_0)}|u|&\leq \operatorname{Tail}_{s,p}((u-k')_+,x_0,R)+ \operatorname{Tail}_{a,t,q}((u-k')_+,x_0,R)+R^{\frac{\rho+N\tilde{\epsilon}}{q}}\theta+2k'\\
        &\ \ \ +C\left[\frac{1}{R^N}\left(\|(u-k')_+\|_{L^p(B_{2R}(x_0))}^p+\left(\frac{R}{2}\right)^{sp-tq}\|(u-k')_+\|_{L^q(B_{2R}(x_0))}^q\right)\right]^{\frac{1}{p}}.
    \end{align*}
     This proves the desired result.
\hfill\qedsymbol{}

\section*{Conflict of interest statement}
On behalf of all authors, the corresponding author states that there is no conflict of interest.
\section*{Data availability statement}
Data sharing does not apply to this article as no datasets were generated or analysed during the current study.
\section*{Acknowledgement}
The author R. Lakshmi thanks the financial support provided by the Ministry of Education (formerly known as  MHRD), Government of India. S. Ghosh acknowledges the research facilities available at the Department of Mathematics, NIT Calicut under the DST-FIST support, Govt. of India [Project no. SR/FST/MS-1/2019/40 Dated. 07.01 2020.] S. Ghosh would also like to thank the Ghent Analysis \& PDE Centre, Ghent University, Belgium for the support during his research visit. C. Zhang was supported by the National Natural Science Foundation of China (No. 12471128). 

%\bibliographystyle{plain}
%\bibliography{ref}

\begin{thebibliography}{10}

\bibitem{BCM2015}
P.~Baroni, M.~Colombo, and G.~Mingione.
\newblock Harnack inequalities for double phase functionals.
\newblock {\em Nonlinear Anal.}, 121:206--222, 2015.

\bibitem{BCM2018}
P.~Baroni, M.~Colombo, and G.~Mingione.
\newblock Regularity for general functionals with double phase.
\newblock {\em Calc. Var. Partial Differential Equations}, 57(2):Paper No. 62, 48 pp., 2018.

\bibitem{BM2021}
B.~Barrios and M.~Medina.
\newblock Equivalence of weak and viscosity solutions in fractional non-homogeneous problems.
\newblock {\em Math. Ann.}, 381(3-4):1979--2012, 2021.

\bibitem{BDVV2022}
S.~Biagi, S.~Dipierro, E.~Valdinoci, and E.~Vecchi.
\newblock Mixed local and nonlocal elliptic operators: regularity and maximum principles.
\newblock {\em Comm. Partial Differential Equations}, 47(3):585--629, 2022.

\bibitem{BV2023}
S.~Biagi and E.~Vecchi.
\newblock Multiplicity of positive solutions for mixed local-nonlocal singular critical problems.
\newblock {\em Calc. Var. Partial Differential Equations}, 63(9):Paper No. 221, 45 pp., 2024.

\bibitem{BMS2023}
A.~Biswas, M.~Modasiya, and A.~Sen.
\newblock Boundary regularity of mixed local-nonlocal operators and its application.
\newblock {\em Ann. Mat. Pura Appl. (4)}, 202(2):679--710, 2023.

\bibitem{BL2017}
L.~Brasco and E.~Lindgren.
\newblock Higher {S}obolev regularity for the fractional {$p$}-{L}aplace equation in the superquadratic case.
\newblock {\em Adv. Math.}, 304:300--354, 2017.

\bibitem{BKO2023}
S.~Byun, H.~Kim, and J.~Ok.
\newblock Local {H}\"older continuity for fractional nonlocal equations with general growth.
\newblock {\em Math. Ann.}, 387(1-2):807--846, 2023.

\bibitem{BO2020}
S.~Byun and J.~Oh.
\newblock Regularity results for generalized double phase functionals.
\newblock {\em Anal. PDE}, 13(5):1269--1300, 2020.

\bibitem{BOS2022}
S.~Byun, J.~Ok, and K.~Song.
\newblock H\"older regularity for weak solutions to nonlocal double phase problems.
\newblock {\em J. Math. Pures Appl. (9)}, 168:110--142, 2022.

\bibitem{CCV2011}
L.~Caffarelli, C.~H. Chan, and A.~Vasseur.
\newblock Regularity theory for parabolic nonlinear integral operators.
\newblock {\em J. Amer. Math. Soc.}, 24(3):849--869, 2011.

\bibitem{CS2007}
L.~Caffarelli and L.~Silvestre.
\newblock An extension problem related to the fractional {L}aplacian.
\newblock {\em Comm. Partial Differential Equations}, 32(7-9):1245--1260, 2007.

\bibitem{CL2018}
W.~Chen and C.~Li.
\newblock Maximum principles for the fractional {$p$}-{L}aplacian and symmetry of solutions.
\newblock {\em Adv. Math.}, 335:735--758, 2018.

\bibitem{CM2015}
M.~Colombo and G.~Mingione.
\newblock Regularity for double phase variational problems.
\newblock {\em Arch. Ration. Mech. Anal.}, 215(2):443--496, 2015.

\bibitem{C2017}
M.~Cozzi.
\newblock Regularity results and {H}arnack inequalities for minimizes and solutions of nonlocal problems: a unified approach via fractional {D}e {G}iorgi classes.
\newblock {\em J. Funct. Anal.}, 272(11):4762--4837, 2017.

\bibitem{DM2024}
C.~De~Filippis and G.~Mingione.
\newblock Gradient regularity in mixed local and nonlocal problems.
\newblock {\em Math. Ann.}, 388(1):261--328, 2024.

\bibitem{DP2019}
C.~De~Filippis and G.~Palatucci.
\newblock H\"older regularity for nonlocal double phase equations.
\newblock {\em J. Differential Equations}, 267(1):547--586, 2019.

\bibitem{DD2012}
F.~Demengel and G.~Demengel.
\newblock {\em Functional spaces for the theory of elliptic partial differential equations}.
\newblock Springer, London; EDP Sciences, Les Ulis, 2012.

\bibitem{DKP2014}
A.~Di~Castro, T.~Kuusi, and G.~Palatucci.
\newblock Nonlocal {H}arnack inequalities.
\newblock {\em J. Funct. Anal.}, 267(6):1807--1836, 2014.

\bibitem{DKP2016}
A.~Di~Castro, T.~Kuusi, and G.~Palatucci.
\newblock Local behavior of fractional {$p$}-minimizers.
\newblock {\em Ann. Inst. H. Poincar\'e{} C Anal. Non Lin\'eaire}, 33(5):1279--1299, 2016.

\bibitem{NPV2012}
E.~Di~Nezza, G.~Palatucci, and E.~Valdinoci.
\newblock Hitchhiker's guide to the fractional {S}obolev spaces.
\newblock {\em Bull. Sci. Math.}, 136(5):521--573, 2012.

\bibitem{D1983}
E.~DiBenedetto.
\newblock {$C\sp{1+\alpha }$}\ local regularity of weak solutions of degenerate elliptic equations.
\newblock {\em Nonlinear Anal.}, 7(8):827--850, 1983.

\bibitem{E1982}
L~C. Evans.
\newblock A new proof of local {$C\sp{1,\alpha }$}\ regularity for solutions of certain degenerate elliptic p.d.e.
\newblock {\em J. Differential Equations}, 45(3):356--373, 1982.

\bibitem{FRZ2024}
Y.~Fang, V~D. R{\u{a}}dulescu, and C.~Zhang.
\newblock Equivalence of weak and viscosity solutions for the nonhomogeneous double phase equation.
\newblock {\em Math. Ann.}, 388(3):2519--2559, 2024.

\bibitem{FZ2022}
Y.~Fang and C.~Zhang.
\newblock Equivalence between distributional and viscosity solutions for the double phase equation.
\newblock {\em Adv. Calc. Var.}, 15(4):811--829, 2022.

\bibitem{FZ2023}
Y.~Fang and C.~Zhang.
\newblock On weak and viscosity solutions of nonlocal double phase equations.
\newblock {\em Int. Math. Res. Not. IMRN}, 2023(5):3746--3789, 2023.

\bibitem{GK2022}
P.~Garain and J.~Kinnunen.
\newblock On the regularity theory for mixed local and nonlocal quasilinear elliptic equations.
\newblock {\em Trans. Amer. Math. Soc.}, 375(8):5393--5423, 2022.

\bibitem{GL2023}
P.~Garain and E.~Lindgren.
\newblock Higher {H}\"older regularity for mixed local and nonlocal degenerate elliptic equations.
\newblock {\em Calc. Var. Partial Differential Equations}, 62(2):Paper No. 67, 36 pp., 2023.

\bibitem{G2003}
E.~Giusti.
\newblock {\em Direct methods in the calculus of variations}.
\newblock World Scientific Publishing Co., Inc., River Edge, NJ, 2003.

\bibitem{IMS2016}
A.~Iannizzotto, S.~Mosconi, and M.~Squassina.
\newblock Global {H}\"older regularity for the fractional {$p$}-{L}aplacian.
\newblock {\em Rev. Mat. Iberoam.}, 32(4):1353--1392, 2016.

\bibitem{I1995}
H.~Ishii.
\newblock On the equivalence of two notions of weak solutions, viscosity solutions and distribution solutions.
\newblock {\em Funkcial. Ekvac.}, 38(1):101--120, 1995.

\bibitem{JJ2012}
V.~Julin and P.~Juutinen.
\newblock A new proof for the equivalence of weak and viscosity solutions for the {$p$}-{L}aplace equation.
\newblock {\em Comm. Partial Differential Equations}, 37(5):934--946, 2012.

\bibitem{JLM2001}
P.~Juutinen, P.~Lindqvist, and J~J. Manfredi.
\newblock On the equivalence of viscosity solutions and weak solutions for a quasi-linear equation.
\newblock {\em SIAM J. Math. Anal.}, 33(3):699--717, 2001.

\bibitem{KKL2019}
J.~Korvenp{\"a}{\"a}, T.~Kuusi, and E.~Lindgren.
\newblock Equivalence of solutions to fractional {$p$}-{L}aplace type equations.
\newblock {\em J. Math. Pures Appl. (9)}, 132:1--26, 2019.

\bibitem{LG2025}
R.~Lakshmi and S.~Ghosh.
\newblock Equivalence of weak and viscosity solutions to mixed local and nonlocal $p$-{L}aplace equation.
\newblock {\em Submitted for publication}, 2025.

\bibitem{L2023}
G.~Leoni.
\newblock {\em A first course in fractional {S}obolev spaces}, volume 229.
\newblock American Mathematical Society, Providence, RI, 2023.

\bibitem{Lewis1983}
J~L. Lewis.
\newblock Regularity of the derivatives of solutions to certain degenerate elliptic equations.
\newblock {\em Indiana Univ. Math. J.}, 32(6):849--858, 1983.

\bibitem{L2016}
E.~Lindgren.
\newblock H\"older estimates for viscosity solutions of equations of fractional {$p$}-{L}aplace type.
\newblock {\em NoDEA Nonlinear Differential Equations Appl.}, 23(5):Art. 55, 18, 2016.

\bibitem{Lions1983}
P.-L. Lions.
\newblock Optimal control of diffusion processes and Hamilton--Jacobi--Bellman equations part 2: viscosity solutions and uniqueness.
\newblock {\em Comm. Partial Differential Equations}, 8(11):1229--1276, 1983.

\bibitem{M1991}
P.~Marcellini.
\newblock Regularity and existence of solutions of elliptic equations with {$p,q$}-growth conditions.
\newblock {\em J. Differential Equations}, 90(1):1--30, 1991.

\bibitem{MO2019}
M.~Medina and P.~Ochoa.
\newblock On viscosity and weak solutions for non-homogeneous {$p$}-{L}aplace equations.
\newblock {\em Adv. Nonlinear Anal.}, 8(1):468--481, 2019.

\bibitem{SM2022}
J.~M. Scott and T.~Mengesha.
\newblock Self-improving inequalities for bounded weak solutions to nonlocal double phase equations.
\newblock {\em Commun. Pure Appl. Anal.}, 21(1):183--212, 2022.

\bibitem{SZ2023}
B.~Shang and C.~Zhang.
\newblock A strong maximum principle for mixed local and nonlocal $p$-{L}aplace equations.
\newblock {\em Asymptot. Anal.}, 133(1-2):1--12, 2023.

\bibitem{S2018}
J.~Siltakoski.
\newblock Equivalence of viscosity and weak solutions for the normalized {$p(x)$}-{L}aplacian.
\newblock {\em Calc. Var. Partial Differential Equations}, 57(4):Paper No. 95, 20 pp., 2018.

\bibitem{SVWZ2022}
X.~Su, E.~Valdinoci, Y.~Wei, and J.~Zhang.
\newblock Regularity results for solutions of mixed local and nonlocal elliptic equations.
\newblock {\em Math. Z.}, 302(3):1855--1878, 2022.

\end{thebibliography}

\end{document}